\documentclass[reqno,11pt]{amsart}
\usepackage[utf8]{inputenc}
\usepackage{amsmath, latexsym, amsfonts, amssymb, amsthm, amscd, bm}
\usepackage{mathrsfs}
\usepackage{graphics,epsf,psfrag}
\usepackage{graphicx}
\usepackage[lofdepth,lotdepth]{subfig}
\usepackage[dvipsnames]{xcolor}
\usepackage{hyperref}
\usepackage{placeins}

\numberwithin{equation}{section}

\newtheorem{theorem}{Theorem}[section]
\newtheorem{lemma}[theorem]{Lemma}
\newtheorem{proposition}[theorem]{Proposition}

\newtheorem{rem}[theorem]{Remark}
\newtheorem{definition}[theorem]{Definition}
\newtheorem{hyp}[theorem]{Hypothesis}

\newcommand{\ind}{\mathbf{1}}

\renewcommand{\tilde}{\widetilde}
\xdefinecolor{red}{named}{Red}

\newcommand{\cF}{{\ensuremath{\mathcal F}} }
\newcommand{\cP}{{\ensuremath{\mathcal P}} }
\newcommand{\cE}{{\ensuremath{\mathcal E}} }
\newcommand{\cH}{{\ensuremath{\mathcal H}} }
\newcommand{\cC}{{\ensuremath{\mathcal C}} }

\newcommand{\cL}{{\ensuremath{\mathcal L}} }

\newcommand{\cV}{{\ensuremath{\mathcal V}} }

\newcommand{\cG}{{\ensuremath{\mathcal G}} }
\newcommand{\cM}{{\ensuremath{\mathcal M}} }
\newcommand{\cO}{{\ensuremath{\mathcal O}} }

\DeclareMathSymbol{\leqslant}{\mathalpha}{AMSa}{"36} 
\DeclareMathSymbol{\geqslant}{\mathalpha}{AMSa}{"3E} 
\DeclareMathSymbol{\eset}{\mathalpha}{AMSb}{"3F}     
\renewcommand{\leq}{\;\leqslant\;}                   
\renewcommand{\geq}{\;\geqslant\;}                   
\newcommand{\dd}{\,\text{\rm d}}             

\newcommand{\bbE}{{\ensuremath{\mathbb E}} }

\newcommand{\bbN}{{\ensuremath{\mathbb N}} }

\newcommand{\bbP}{{\ensuremath{\mathbb P}} }

\newcommand{\bbR}{{\ensuremath{\mathbb R}} }

\newcommand{\ga}{\alpha}
\newcommand{\gb}{\beta}
\newcommand{\gd}{\delta}
\newcommand{\gep}{\varepsilon}       

\newcommand{\gl}{\lambda}

\newcommand{\gs}{\sigma}

\makeatletter
\def\captionfont@{\footnotesize}
\def\captionheadfont@{\scshape}

\long\def\@makecaption#1#2{%
  \vspace{2mm}
  \setbox\@tempboxa\vbox{\color@setgroup
    \advance\hsize-6pc\noindent
    \captionfont@\captionheadfont@#1\@xp\@ifnotempty\@xp
        {\@cdr#2\@nil}{.\captionfont@\upshape\enspace#2}%
    \unskip\kern-6pc\par
    \global\setbox\@ne\lastbox\color@endgroup}%
  \ifhbox\@ne 
    \setbox\@ne\hbox{\unhbox\@ne\unskip\unskip\unpenalty\unkern}%
  \fi
  \ifdim\wd\@tempboxa=\z@ 
    \setbox\@ne\hbox to\columnwidth{\hss\kern-6pc\box\@ne\hss}%
  \else 
    \setbox\@ne\vbox{\unvbox\@tempboxa\parskip\z@skip
        \noindent\unhbox\@ne\advance\hsize-6pc\par}%
\fi
  \ifnum\@tempcnta<64 
    \addvspace\abovecaptionskip
    \moveright 3pc\box\@ne
  \else 
    \moveright 3pc\box\@ne
    \nobreak
    \vskip\belowcaptionskip
  \fi
\relax
}
\makeatother
\def\writefig#1 #2 #3 {\rlap{\kern #1 truecm
\raise #2 truecm \hbox{#3}}}

\newcommand{\Tr}{\text{\rm Tr}}

\title[Emergence of oscillatory behaviors for mean-field excitable systems]
{Emergence of oscillatory behaviors for excitable systems with noise and mean-field interaction,\\ a slow-fast dynamics approach.}

\author{Eric Lu\c{c}on}
\address{Laboratoire MAP5 (UMR CNRS 8145), Universit\'e Paris Descartes, Sorbonne Paris Cit\'e, 75270 Paris, France, \url{eric.lucon@paridescartes.fr}.
}

\author{Christophe Poquet}
\address{Univ Lyon, Université Claude Bernard Lyon 1, CNRS UMR 5208, Institut Camille Jordan, F-69622 Villeurbanne, France, \url{poquet@math.univ-lyon1.fr}}


\date{\today}

\begin{document}

\begin{abstract}
We consider the long-time dynamics of a general class of nonlinear Fokker-Planck equations, describing the large population behavior of mean-field interacting units. Our main motivation concerns (but not exclusively) the case where the individual dynamics is excitable, i.e. when each isolated dynamics rests in a stable state, whereas a sufficiently strong perturbation induces a large excursion in the phase space. We address the question of the emergence of oscillatory behaviors induced by noise and interaction in such systems. We tackle this problem
by considering this model as a slow-fast system (the mean value of the process giving the slow dynamics) in the regime of small individual dynamics and by proving
the existence of a positively stable invariant manifold, whose slow dynamics is at first order the dynamics of a single individual averaged with a Gaussian kernel. We consider applications of this result to Stuart-Landau, FitzHugh-Nagumo and Cucker-Smale oscillators.
\end{abstract}

\maketitle

\section{Introduction}
\subsection{ Microscopic models of mean-field excitable units}

The aim of the paper is to address the long-time behavior of a class of nonlinear Fokker-Planck PDEs arising as the limit in large population of a system of mean-field interacting units. The dynamics of each isolated unit is described by a $d$ dimensional variable $X_{ t}$ solution to the system
\begin{equation}
\label{eq:IDS_intro}
{\rm d}X_{ t} = F(X_{ t}) {\rm d}t,
\end{equation}
for a certain functional $F: \mathbb{ R}^{ d} \to \mathbb{ R}^{ d}$. Our main (but not exclusive) interest concerns situations where \eqref{eq:IDS_intro} describes the dynamics of excitable units. Excitability is a phenomenon that has been widely observed in physics (see \cite{LINDNER2004321} and references therein) and life sciences (\cite{bressloff2014stochastic,LINDNER2004321} and references therein), especially in neuroscience \cite{MR2263523}. Informally speaking, we consider as \emph{excitable}, any system \eqref{eq:IDS_intro} that, without any perturbation, would stay in a stable resting state, whereas a sufficiently strong perturbation would force the system to leave this resting state and come back, resulting in a large excursion in the phase space.  A prominent example of excitable system is given by the FitzHugh-Nagumo model \cite{FitzHugh1961,Nagumo1962,LINDNER2004321,MR1779040} described by $X=(x,y)\in \mathbb{ R}^{ 2}$ and
\begin{equation}
\label{eq:FHN_u1}
F(x,y)\, =\,  \left( x-\frac{x^3}{3}-y,\frac{1}{\tau}(x+a-by) \right)\, ,
\end{equation}
for $a\geq0$, $\tau>0$ and $b\in \mathbb{ R}$. The FitzHugh-Nagumo model, introduced as a two-dimensional idealization of the dynamics of the activity of one neuron (where $x$ is the membrane potential and $y$ a recovery variable), has proven to be a simple prototype for excitability \cite{MR2263523,LINDNER2004321}. In this context, the dynamical description made above can be understood as a spiking activity.

Suppose now that we are given $N\gg1$ copies of \eqref{eq:IDS_intro}, $X_{ 1}, \ldots, X_{ N}$, each of them perturbed by thermal noise and within a linear mean-field interaction (see for example \cite{22657695,MR3392551} for similar models in the context of neuroscience):
\begin{equation}
\label{eq:part_syst}
{\rm d}X_{ i, t}= \left(\delta F(X_{ i, t}) -K \left(X_{ i, t}- \frac{ 1}{ N}\sum_{ j=1}^{ N}X_{ j, t}\right)\right)\dd t+\sqrt{2}\gs \dd B_{ i, t}\, ,\ i=1,\ldots, N\, ,\ t\geq0\, ,
\end{equation}
where $B_{ 1}, \ldots, B_{ N}$ is a collection of independent standard Brownian motions in $ \mathbb{ R}^{ d}$ (modeling thermal noise in the system), $K$ and $\sigma$ are two diagonal matrices, with positive coefficients $k_1, \ldots k_d$ and $\gs_1, \ldots, \gs_d$. In the context of neuronal activity, the presence of noise can be intrinsic to each neuron (e.g. the random switching of ion channels \cite{MR3024608,MR2779558}) or can come from the random input from other neurons. 

\subsection{Structured dynamics induced by noise and interaction} At an informal level, a rather general question is to ask if noise and interaction may induce for \eqref{eq:part_syst} a structured dynamics (e.g. synchronization, collective periodic behavior, see for example \cite{Ko2010} in the context of circadian rhythms) that is not originally observed for the isolated system \eqref{eq:IDS_intro}. 

This question of the influence of noise and interaction on mean-field systems has been a longstanding issue in the literature. Several papers from the physics literature have studied the existence of coherent structures for excitable systems (such as coherence resonance \cite{PhysRevE.60.7270}, pattern formation \cite{MR1697197} or wave propagation, see \cite{LINDNER2004321} and references therein). A first mathematical result showing that noise and interaction may induce periodic behaviors is due to Scheutzow \cite{MR808166}. Related works for various mean-field systems may be found in \cite{collet2016rhythmic,Collet:2015aa,dai2014noise,Ditlevsen:2015fk,Mischler2016,MR2902610}. In the case of excitable systems, some special attention has been given, due to the simplicity of their isolated dynamics, to the particular case of phase oscillators (the Active rotators model \cite{Sakaguchi1986}, a generalization of the Kuramoto model \cite{Kuramoto1975}), both from a perspective of physics (\cite{Shinomoto1986,Sakaguchi1988a}) or neuroscience (\cite{MR833476}) and from a mathematical point of view (\cite{doi:10.1137/110846452,MR3207725}), and apparition of periodic behaviors induced by noise and interaction have been proved in this case. One should mention at this point that the intrinsic nature of excitable systems such as \eqref{eq:part_syst} (and the main difficulty in the analysis) is their absence of reversibility: we are naturally dealing with non-equilibrium systems. 

To our knowledge, our paper provides the first rigorous proof of periodic behaviors induced by noise and interaction in the FitzHugh Nagumo model.

\subsection{ The mean-field limit point of view}
The point of view we adopt in this paper is to consider the large population limit $N\to \infty$ in \eqref{eq:part_syst}: standard propagation of chaos results \cite{SznitSflour} (see also \cite{MR3392551,LucSta2014} for results which cover our present case) show that the empirical measure $ \mu_{ N, t}:= \frac{ 1}{ N} \sum_{ j=1}^{ N} \delta_{ X_{ j, t}}$ of the particle system \eqref{eq:part_syst} is well described in the limit  $N\to\infty$ by the following nonlinear Fokker-Planck equation
\begin{equation}\label{eq:PDE}
\partial_t \mu_t(x)\, =\, \nabla\cdot (\gs^2 \nabla \mu_t)(x) + \nabla \cdot\left( K\mu_t(x)\left(x-\int_{\bbR^d}z\mu_t( {\rm d}z)\right)\right)-\delta \nabla\cdot \big(\mu_t(x) F(x)\big)\, ,t\geq0\, ,
\end{equation}
whose solution $ (\mu_{ t})_{ t\geq0}$ is a probability measure-valued process on $ \mathbb{ R}^{ d}$, describing the law of a typical particle $X_{ t}$ in an infinite population. This nonlinear process is formally described by
\begin{equation}\label{eq:Mc Kean}
\dd X_t\, =\, \Big(\delta F(X_t) -K\big(X_t-\bbE[X_t]\big)\Big)\dd t+\sqrt{2}\gs \dd B_t\, , t\geq0\, .
\end{equation}
Such nonlinear process (that is interacting with its own law through $ \mathbb{ E} \left[X_{ t}\right]$) has been studied since McKean \cite{McKean1967} and Sznitman \cite{SznitSflour}. Well-posedness results for both \eqref{eq:PDE} and \eqref{eq:Mc Kean} will be provided in the following.

\subsection{Slow-fast dynamics and invariant manifold}
Our approach is based on the fact that, since the two first terms of the right hand side of \eqref{eq:PDE} leave the mean $m_t =\int_{\bbR^d} x \mu_t(\dd x)$ invariant, this PDE is in fact equivalent to the system
\begin{equation}\label{eq:slow fast PDE}
\left\{
\begin{array}{rl}
\partial_t p_t(x)& =\, \nabla\cdot(\gs^2 \nabla p_t(x))+\nabla\cdot (Kp_t(x)x)+ \nabla\cdot (p_t(x)(\dot m_t-\gd F(x+m_t)))\\
\dot m_t& =\,  \gd \int_{\bbR^d} F(x+m_t) p_t(\dd x)
\end{array}
\right. \, ,
\end{equation}
where $p_t$ is the centered version of $\nu_t$, i.e. satisfies for all test function $\varphi$
\begin{equation}
\int_{\bbR^d} \varphi(x)p_t(\dd x)\, =\, \int_{\bbR^d} \varphi(x-m_t)\mu_t(\dd x)\, .
\end{equation}
Taking $\gd$ small, the system \eqref{eq:slow fast PDE} defines a slow-fast dynamics, the infinite dimensional one given by $p_t$ being the fast one. Following a classical approach for such systems, one would like then to consider the dynamics of $p_t$ with $\gd=0$, which is simply an Ornstein-Uhlenbeck dynamics, with exponential convergence to the Gaussian measure of density $q_{0,\gs^2 K^{-1}}$ (more details will be given in Section \ref{subsec:OU}), where
\begin{equation}
\label{eq:qm0_Gamma}
q_{m,\Gamma}(x)=\frac{1}{((2\pi)^{ d} \det(\Gamma))^{\frac{1}{2}}}\exp \left(-\frac{1}{2}(x-m)\cdot\Gamma^{-1}(x-m) \right),\ x\in \mathbb{ R}^{ d}\, .
\end{equation}
One would then replace $p_t$ by this limit in the equation of evolution $m_t$, obtaining the approximation
\begin{equation}\label{eq:approx m_t}
\dot m_t\, \approx\, \gd \int_{\bbR^d} F(x+m_t)q_{0,\gs^2K^{-1}}(x)\dd x\, =\, \gd \int_{\bbR^d} F(x)q_{m_t,\gs^2K^{-1}}(x)\dd x\, ,
\end{equation} 
which simply corresponds to replacing the right-hand side of \eqref{eq:IDS_intro} with its average with respect to a Gaussian measure centered in $m_t$, and slowing down the dynamics by a factor $\gd$. The purpose of this paper is make this approximation rigorous, and thus reducing drastically the dimension of the problem: the point being then to look for structured dynamics for the $d$-dimensional problem \eqref{eq:approx m_t}.

\medskip

To prove this approximation, we follow the founding arguments of Fenichel \cite{fenichel1971persistence,fenichel1979geometric} who solved this problem in finite dimension relying on the persistence under perturbation of normally hyperbolic manifolds: in our case the manifold $\cM^0=\{(q_{0,\gs^2K^{-1}},m):\, m\in \bbR^d\}$ is a stable manifold of stationary solutions for \eqref{eq:slow fast PDE} with $\gd=0$, and our aim is to prove that it persists in an invariant manifold $\cM^\gd$ for $\gd>0$ small, and that the phase dynamics on $\cM^\gd$ can be approximated by \eqref{eq:approx m_t} (more precisely we will only prove the existence of positively invariant manifolds, see Theorem \ref{th:main1} for a precise result). The persistence result of Fenichel has been generalized in several directions (see for example \cite{hirsch1977invariant,wiggins2013normally,sell2013dynamics,Bates1998}), in particular for infinite dimensional systems. We could not apply directly the very general result of \cite{Bates1998} in our situation, the main difficulty that arises in our case being the fact that the function $F$ we consider may be nonlinear (in particular in the case of excitable systems), with $\vert F(x)\vert$ growing rapidly at infinity. We will tackle this issue by studying the existence and $C^1$ character of $\cM^\gd$ in two different weighted $L^2$ spaces in which the Ornstein Uhlenbeck dynamics contracts, and it will in fact be sufficient to prove that the approximation \eqref{eq:approx m_t} is valid in $C^1$, so that this approximation describes accurately the phase dynamics on $\cM^\gd$. For other studies of slow-fast dimensional systems that do not apply to our context, see for example \cite{MR1022538,menonhaller}. Note that the works \cite{doi:10.1137/110846452,MR3207725} made on the Active Rotators model are also based on persistence of stable manifold, but on simpler models, with the dynamics of a single unit defined on the unit circle (so without problems of growth at infinity).

\medskip

An analysis related to this work, in the case of FitzHugh-Nagumo oscillators, has recently been made in \cite{Mischler2016}, showing in particular the existence of equilibria for the limit mean-field dynamics. Our analysis differs in two ways: \cite{Mischler2016} concerns the kinetic case where no noise and interaction is imposed for the recovery variable $y$, whereas our analysis requires that the noise and interaction is present on each component. Secondly, in \cite{Mischler2016} the mean field model is considered in the limit of small interaction, whereas we are concerned with the (somehow opposite) case where the dynamics is small w.r.t. the interaction.
\subsection{Organisation of the paper.}
We present in Section~\ref{sec:model} the model, the assumptions and main results. Section~\ref{sec:simulations} concerns examples and numerical simulations. The main regularity and a priori estimates are gathered in Section~\ref{sec:well_posed_apriori}. The existence of the perturbed manifold is proven in Section~\ref{sec:persistence} and it $ \mathcal{ C}^{ 1}$-character in Section~\ref{sec:regularity_manifold}.

\section{Assumptions and main results}
\label{sec:model}
\subsection{ Notations and first definitions}
We denote by $x \cdot y$ and $ \left\vert x \right\vert$ as the Euclidean scalar product and norm in $ \mathbb{ R}^{ d}$. We will also use the notation
\begin{equation}
\vert x\vert_A\, =\, \left(x\cdot A x\right)^{ \frac{ 1}{ 2}}\, ,
\end{equation}
as the Euclidean norm twisted by some positive symmetric matrix $A$. We will mainly use this twisted norm in the following for the choice of $A:= K \sigma^{ -2}$. For this norm and any $R>0$, we denote by $B_{ R}:= \left\lbrace x\in \mathbb{ R}^{ d},\ \left\vert x \right\vert\leq R\right\rbrace$ as the corresponding closed ball of radius $R>0$. Moreover, for any mapping open subset $\cO$ and any $x\in \cO \mapsto g(x)\in \mathbb{ R}^{ d}$, we denote by $ \left\Vert g \right\Vert_{ \mathcal{ C}^{ 1}(\cO, \mathbb{ R}^{ d})}= \sup_{ x\in \cO} \left\vert g(x) \right\vert + \sum_{ k=1, \ldots, d}\sup_{ x\in B_{ R}} \left\vert \partial_{ x_{ k}} g(x) \right\vert$ as the usual $ \mathcal{ C}^{ 1}$-norm of $g$. Define also
\begin{align}
\label{eq:kminmax}
\underline{ k}=\min \left(k_{ 1}, \ldots, k_{ d}\right),&\ \bar{ k}=\max \left(k_{ 1}, \ldots, k_{ d}\right),\\
\label{eq:sigminmax}
\underline{ \sigma}=\min \left(\sigma_{ 1}, \ldots, \sigma_{ d}\right),&\ \bar{ \sigma}=\max \left(\sigma_{ 1}, \ldots, \sigma_{ d}\right).
\end{align}
The analysis of \eqref{eq:PDE} will require the definition of weighted norms: let $x \mapsto w(x)$ a measurable positive weight. We define here the corresponding $L^{ 2}$ and $H^{ 1}$ norms as
\begin{equation}
\label{eq:L2norm_w}
\left\Vert u \right\Vert_{ L^{ 2}(w)}= \left(\int_{ \mathbb{ R}^{ d}} \left\vert u(x) \right\vert^{ 2} w(x) {\rm d}x\right)^{ \frac{ 1}{ 2}}\, ,  \qquad  
\Vert u\Vert_{H^1(w)}= \left(\Vert u\Vert^2_{L^2(w)}+\sum_{i=1}^d\Vert \partial_{x_i} u\Vert^2_{L^2(w)}\right)^{ \frac{ 1}{ 2}}\, .
\end{equation}
We use also the notation $ \left\langle u\, ,\, v\right\rangle_{ L^{ 2}(w)}$ and $ \left\langle u\, ,\, v\right\rangle_{ H^{ 1}(w)}$ for the corresponding scalar products. We will use the family of weights defined as, for any $ \theta\in \mathbb{ R}$,
\begin{equation}
\label{eq:w_alpha}
w_{ \theta}(x):=\exp\left(\frac{\theta}{2} \vert x\vert_{K\gs^{-2}}^2\right),\ x\in \mathbb{ R}^{ d}\,.
\end{equation} 
\subsection{ Main assumptions}\label{sec:asumption}
We make here the following hypotheses on $F$:
\begin{hyp}\label{hyp F}
\begin{enumerate}
\item There exists a positive constant $C_F$ such that
\begin{equation}
\label{hyp:F_Lipschitz}
(F(x)-F(y))\cdot(x-y)\, \leq\, C_F \vert x-y\vert^2,\ x, y\in \mathbb{ R}^{ d}\, .
\end{equation}
\item There exist positive constants $c_F$, $C_F$ and $r$ such that
\begin{equation}
\label{hyp:F_dot_x_bound}
F(x)\cdot K\gs^{-2} x\, \leq \, C_F \ind_{\{ \left\vert x \right\vert_{ K \sigma^{ -2}}\leq r\}} - c_F \vert x\vert_{ K \sigma^{ -2}}^2\,,\ x\in \mathbb{ R}^{ d}
\end{equation}
\item The following limits holds:
\begin{equation}
\label{hyp:F_dot_x_larger}
\lim_{ \vert x \vert \rightarrow \infty} \frac{\vert \partial_{x_k} F(x)\vert}{ F(x)\cdot K\gs^{-2} x}\, =\, 0 \, \text{ for } k=1,\ldots,d\, .
\end{equation}
\item There exist $C_F>0$ and $\gep>0$ such that
\begin{equation}
\label{hyp:bound_F_exp}
\max \left( \left\vert F(x) \right\vert, \sup_{ k=1, \ldots, d} \left\vert \partial_{ x_{ k}}F (x)\right\vert, \sup_{ k, l=1, \ldots, d} \left\vert \partial_{ x_{ k}, x_{ l}}^{ 2} F^{ (l)}(x) \right\vert\right) \leq C_{ F} w_\gep(x),\ x\in \mathbb{ R}^{ d}.
\end{equation}
\item There exists bounded open subset $\cV$ of $\bbR^d$ with smooth boundary such that for all $m\in \partial\cV$
\begin{equation}
n_{\partial\cV}(m) \cdot \int_{\bbR^d} F(x)q_{m,\gs^2K^{-1}}(x)\dd x \, <\, 0
\end{equation}
where $n_{\partial \cV}(m)$ denotes the exterior normal of $\partial \cV$ at $m$.
\end{enumerate}
\end{hyp}

The point of the main results below is to show the existence of a positively invariant manifold $\cM^\gd$ for \eqref{eq:slow fast PDE} (i.e. a manifold such that $(p_t,m_t)\in \cM$ for all $t>0$ as soon as $(p_0,m_0)\in \cM$, where $(p_t,m_t)$ is solution of \eqref{eq:slow fast PDE}) defined for $m\in \cV$. We give in the following Lemma a sufficient condition for the point (5) of Hypothesis \ref{hyp F} to be satisfied.

\begin{lemma}
\label{lem:F_vs_q0}
Suppose that $F$ satisfies
\begin{equation}
\label{hyp:F_dot_x_bound 2}
F(x)\cdot K\gs^{-2} x\, \leq \, C_F \ind_{\{ \left\vert x \right\vert_{ K \sigma^{ -2}}\leq r\}} - c_F \vert F(x)\vert_{ K \sigma^{ -2}}\vert x\vert_{ K \sigma^{ -2}}\,,\ x\in \mathbb{ R}^{ d}\, .
\end{equation}
Then there exists a $ \rho_{ 0}>0$ such that, if $F$ satisfies the point (4) of Hypothesis \ref{hyp F} with $ 0<\varepsilon \leq \rho_{ 0} $, there exists a $L_0>0$ such that $F$ satisfies the point (5) of Hypothesis \ref{hyp F}  with $\cV=\{x\in \bbR^d:\, \vert x\vert_{K\gs^{-2}}\leq L_0\}$.
\end{lemma}
The proof of Lemma~\ref{lem:F_vs_q0} is postponed to Appendix~\ref{sec:proof_lemma_F_VS_q0}. We suppose finally that
\begin{equation}
\label{hyp:epsilon_small}
\varepsilon\, <\, \frac{ 1}{ 5(\Tr(K)+4 \underline{ k})}
\end{equation}
this condition being related to \eqref{eq:def_alpha_beta} below.

\subsection{Positively invariant manifold and its phase dynamics} 

The main result of the paper is the following: 
\begin{theorem}\label{th:main1}
Under the assumptions of Section~\ref{sec:asumption} there exist $\ga\in  \left(0,1\right)$, $\bar \gd>0$ and $C>0$ such that the following is true: for all $0\leq \gd\leq \bar \gd$, there exists a positively invariant manifold $ \mathcal{ M}^\gd=\{(p_m^\gd,m):\, m\in \cV\}$ for \eqref{eq:PDE}, where $p^\gd_m$ is a probability measures on $ \mathbb{ R}^{ d}$ for all $m\in \cV$, and $\cM^\gd$ is a perturbation of size $ \delta$ of the manifold $ \mathcal{ M}^{ 0}$ in the following sense:
\begin{equation}
\sup_{m\in \cV} \left\Vert p^\gd_m-q_{0,\gs^2 K^{-1}}\right\Vert_{L^2 (w_\ga)}\, \leq\, C \gd\, .
\end{equation}
\end{theorem}

\begin{rem}\label{rem:stab}
The proof of Theorem \ref{th:main1} implies in particular that $ \mathcal{ M}^\gd$ is stable in the following sense: there exist $\beta\in(0, \ga)$,  positive constants $\lambda, c, c',c''$ and $C'$ such that the following is true: if $(p_0,m_0)$ satisfies $m_0\in \cV$, $\int_{\bbR^d} x p_0(\dd x)=0$, $\int_{\bbR^d} w_\ga(x+m_0)p_0(\dd x)\leq c$, $\Vert p_0-q_{0,\gs^2K^{-1}}\Vert_{L^2(w_\ga)}\leq c' \gd$ and $\Vert p_0-p^\gd_{m_0}\Vert_{L^2(w_\beta)}\leq c''\gd $, then for $(p_t,m_t)$ the solution of \eqref{eq:slow fast PDE}
with initial condition $(p_0,m_0)$ we have for all $t>0$:
\begin{equation}\label{eq:contract rem}
m_t\in \cV\, , \quad \text{and}\quad \left\Vert p_t-p^\gd_{m_t}\right\Vert_{L^2 (w_\beta)}\, \leq\, C' e^{-\lambda t}\left\Vert p_0-p_{m_0}^\gd\right\Vert_{L^2(w_\beta)}\, .
\end{equation}
\end{rem}

For any $\gd\in [0,\bar\gd]$ and any $m_0\in \cV$ we denote by $t\mapsto m^\gd_t$ the phase dynamics of the solution of \eqref{eq:slow fast PDE} starting from $(p^\gd_{m_0},m_0)\in \mathcal{ M}^\gd$.
\begin{theorem}\label{th:main2}
The trajectory $t\mapsto m^\gd_t$ is a $ \mathcal{ C}^1$-perturbation (slowed down by a factor $\gd$) of the dynamics given by the equation $\dot m_t =\int_{\bbR^d}F(u)q_{m_t,\gs^2K^{-1}}(u)\dd u$: there exist $C''>0$ and a $ \mathcal{ C}^1$-mapping $g^\gd:\cV\rightarrow\bbR^d$ such that
\begin{equation}\label{eq:dot m th}
\dot m^\gd_t\, =\, \gd\int_{\bbR^d}F(u)q_{m^\gd_t,\gs^2K^{-1}}(u)\dd u+ \gd^2 g^\gd(m^\gd_t)\, ,
\end{equation}
and $\Vert g^\gd\Vert_{ \mathcal{ C}^1(\cV,\bbR^d)}\leq C''$.
\end{theorem}
\begin{rem}[Dependence of the results in $ \varepsilon$]
Recall the definition in $ \varepsilon>0$ in \eqref{hyp:bound_F_exp}. One can choose the constants $ \alpha= \alpha(\varepsilon)$ appearing in Theorem~\ref{th:main1}, $ \beta= \beta(\varepsilon)$ and $ \lambda= \lambda(\varepsilon,\gd)$ in Remark~\ref{rem:stab} so that they satisfy the following asymptotics as $ \varepsilon\to 0$: 
\begin{equation}
\ga(\varepsilon)\xrightarrow[\gep\rightarrow 0]{}1,\ \beta(\varepsilon)\xrightarrow[\gep\rightarrow 0]{}1\text{ and }\lambda(\varepsilon,\gd)\xrightarrow[\gep,\gd\rightarrow 0]{} \underline{ k}.
\end{equation}
\end{rem}

\section{Examples and simulations}
\label{sec:simulations}
\subsection{A general principle: the mean-field model viewed as a perturbation of the isolated deterministic system}
The general point of view given by Theorem \ref{th:main1} and Theorem \ref{th:main2} is to see the nonlinear dynamics \eqref{eq:Mc Kean} (or equivalently \eqref{eq:slow fast PDE}) as a modification (under noise and interaction) of the isolated deterministic system (IDS):
\begin{equation}
\label{eq:IDS}
\dot x_t\, =\, F(x_t)\, .
\end{equation}
As already mentioned in the Introduction, a general (and rather informal) issue at this point is to question the influence of noise and interaction on the dynamical properties of the IDS \eqref{eq:IDS}. We highlight here two main scenarios which are of significant importance in this context: first, \emph{persistence of dynamics under perturbation} (i.e. when the dynamics observed for \eqref{eq:Mc Kean} is similar the dynamics of the IDS \eqref{eq:IDS}, see Section~\ref{sec:persistence_under_perturbation}) and secondly, \emph{emergence of structured dynamics under noise and interaction} (i.e. when noise and interaction are at the origin of  dynamics that differ from the dynamics of the IDS, see Section~\ref{sec:emergence_dynamics_perturbation}).

As elucidated by Theorem~\ref{th:main2}, understanding the dynamics of the mean-value of the perturbed process \eqref{eq:PDE} boils down to understanding the system
\begin{equation}\label{eq:conv F of order epsilon}
\dot x_t\, =\, \int_{\bbR^d} F(u) q_{x_t, \gs^2 K^{-1}}(u) \dd u\, ,
\end{equation}
depending on the parameters $\varpi_1=\gs_1^{2}/k_1,\ldots ,\varpi_d=\gs_d^2/k_d$. The main issue here is to understand if the dynamics of \eqref{eq:conv F of order epsilon} may (or may not) differ significantly from the dynamics of \eqref{eq:IDS}. We will illustrate below this analysis with several examples:
\subsubsection{ Stuart-Landau oscillators.} Consider in $\bbR^2$ the function
\begin{equation}
\label{eq:stuart_landau}
F(x,y)\, =\, \left(x(a-(x^2+y^2))-\omega y , y(a-(x^2+y^2)+ \omega x) 
\right)\, ,
\end{equation}
where $a>0$, $ \omega\in \mathbb{ R}$. In polar coordinates, \eqref{eq:stuart_landau} corresponds to the dynamics given by $\dot \theta=\omega$, $\dot r=r(a-r^2)$, which admits the stable limits cycle $r_t=\sqrt{a}$, $\theta_t=\theta_0+\omega t$. The main remark here is that, considering $\varpi_1=\varpi_2=\varpi$, \eqref{eq:conv F of order epsilon} defines again a Stuart Landau model:
\begin{multline}
\label{eq:SL_av}
\int_{\bbR^2} F(u) q_{(x,y), \gs^2K^{-1}}(u)\dd u\\ =\, \left( \omega y - x(a-4 \varpi^2-(x^2+y^2)), -\omega x- y(a-4 \varpi^2-(x^2+y^2)) 
\right)\, .
\end{multline}
In particular the point (5) of Hypothesis \ref{hyp F} is satisfied for $\cV$ being a $B_R$ with $R$ large enough.

\subsubsection{FitzHugh-Nagumo oscillators}
Recall the definition of the FitzHugh-Nagumo dynamics in \eqref{eq:FHN_u1}. For the purpose of the analysis below, we introduce one more parameter $u\in(0, 1]$ and consider
\begin{equation}
\label{eq:FHN_u}
F_{ u, a, b, \tau}(x, y)\, =\,  \left( ux-\frac{x^3}{3}-y,\frac{1}{\tau}(x+a-by) \right)\, ,
\end{equation}
Starting from \eqref{eq:FHN_u1}, a direct calculation shows that
\begin{align}
\int_{\bbR^2} F_{ 1, a, b, \tau}(z) q_{(x,y),\gs^2 K^{-1}}(z)\dd z&= \left( (1- \varpi_1) x-\frac{x^3}{3}-y,\frac{1}{\tau}(x+a-by)
\right) \nonumber\\
&= F_{ 1- \varpi_1, a, b, \tau}(x, y)\, , \label{eq:FHN_transf}
\end{align}
and so \eqref{eq:FHN_transf} defines again a FitzHugh Nagumo system, where the parameter $u=1$ has been changed into $u=1- \varpi_1$.
Again, the point (5) of Hypothesis \ref{hyp F} is satisfied for $\cV$ being a $B_R$ with $R$ large enough.

\subsection{Persistence of dynamics under perturbation}
\label{sec:persistence_under_perturbation}
Suppose here that the IDS \eqref{eq:IDS} possesses a dynamical structure persistent under $\mathcal{ C}^{ 1}$-perturbation that is included in a bounded open set $\cV_0$ with smooth boundaries. Examples of such persistent structures are hyperbolic fixed-points, limit cycles, and more generally normally hyperbolic invariant manifolds \cite{fenichel1971persistence,wiggins2013normally}, but also chaotic strutures, as given by Lorentz-like flows \cite{guckenheimer1979structural,Araujo2010}.

Consider now small parameters $ \varpi_1,\ldots,\varpi_d$ in \eqref{eq:conv F of order epsilon}. Then, classical convolution arguments show that the dynamics given by \eqref{eq:conv F of order epsilon}
is a $ \mathcal{ C}^1$-perturbation (with perturbation of order $\max_i \varpi_i$) of the IDS, so that the point (5) of Hypothesis \ref{hyp F} is satisfied for $\cV=\cV_0$   and \eqref{eq:conv F of order epsilon} admits a similar persistent structure (a perturbed version of the initial one) for $ \max_i \varpi_i$ small enough. Then, according to Theorem \ref{th:main2}, the phase dynamics on $ \mathcal{ M}^\gd$ is a (slowed down) $ \mathcal{ C}^1$-perturbation of \eqref{eq:conv F of order epsilon}, which means that \eqref{eq:PDE} admits also a similar persistent structure for $\gd$ small enough (which depends on $\max_i \varpi_i$), and which is stable in the sense of Remark \ref{rem:stab}.

This is in particular true for the Stuart-Landau model with $\varpi_1=\varpi_2=\varpi$: we see from \eqref{eq:SL_av} that \eqref{eq:PDE} possesses a limit cycle as soon as $ \varpi^2<a/4$ and $\gd$ is small enough. This is also true in the FitzHugh-Nagumo case \eqref{eq:FHN_u}: if the parameters $(a, b, \tau)$ in \eqref{eq:FHN_u1} are chosen so that the isolated system $(\dot x, \dot y)= F(x,y)$ is away from a bifurcation point and if $ \varpi_1 \in[0, 1)$ is small enough, \eqref{eq:PDE} possesses the same type of dynamics as the IDS. In particular, it is well known (see \cite{MR1779040} for a complete study of the bifurcations of this model) that this model admits a limit cycle for an appropriate choice of parameters. This analysis shows the persistence of this limit cycle for the synchronized system \eqref{eq:PDE}, at least when the noise is small with respect to the interaction (as it had already been observed in \cite{MR2098593}).

A similar result had already been obtained in \cite{MR843504} for the mean-field Brusselator model, where the persistence of the periodic dynamics for the McKean process is proved, relying on other arguments, that allow in particular the use of diffusions terms depending on the positions but that do not ensure local stability. 

\subsection{Emergence of dynamics under noise and interaction}
\label{sec:emergence_dynamics_perturbation}
Suppose now that \eqref{eq:IDS} exhibits a bifurcation and that a careful choice of parameters in the functional $F$ brings \eqref{eq:IDS} close to the bifurcation point. It may be that the introduction of the parameters $ \varpi_1,\ldots,\varpi_d$ in \eqref{eq:conv F of order epsilon}  makes the system cross this bifurcation point: we would then be precisely in a situation where the addition of noise and interaction induce a structured dynamical behavior that is not initially present in the IDS \eqref{eq:IDS} for this choice of parameters.

\subsubsection{A crucial simple example:} suppose that $d=2$ and that around the origin the mapping $F$ is given by
\begin{equation}
\label{eq:F_toymodel}
F(x,y)\, =\,  \left( x^2-a,-b y \right)\, ,
\end{equation}
with $b>0$. Remark here that $F=\nabla V$ with $V(x,y)=ax-\frac{x^3}{3}+\frac{b x^2}{2}$. System \eqref{eq:F_toymodel} is a simple prototype of a dynamics with a saddle-node bifurcation (see Figure~\ref{fig:toy_model}): when $a>0$, \eqref{eq:F_toymodel} admits two stationary points $(-\sqrt{a},0)$ (stable) and $(\sqrt{a},0)$ (unstable). When $a>0$ goes to $0$, these two points collide so that the dynamics on $x$ simply boils down to a drift to the right when $a<0$.
\begin{figure}[ht]
\centering
\subfloat[$a>0$]{\includegraphics[width=0.35\textwidth]{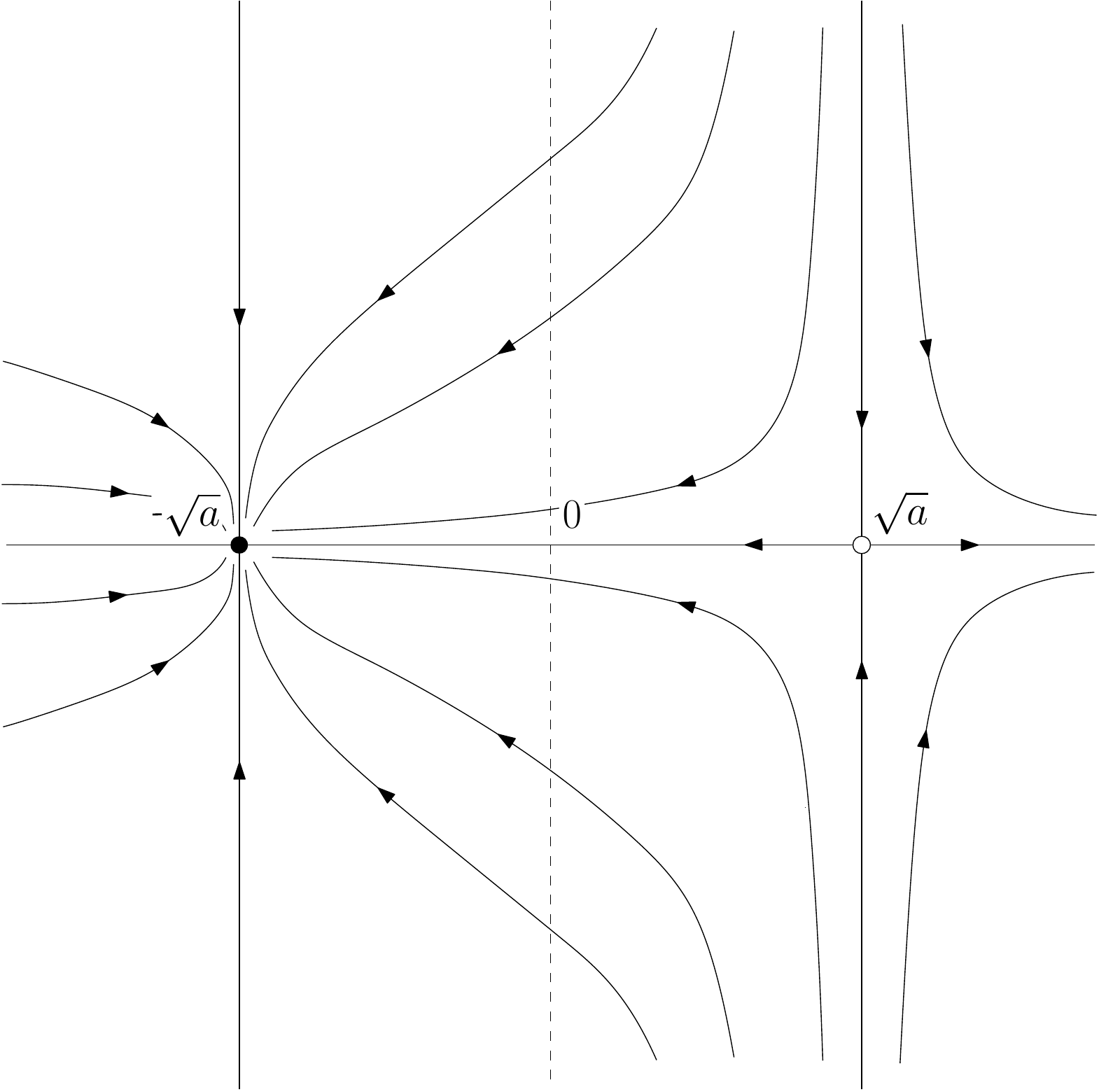}
\label{subfig:toy_model_apos}}
\quad\subfloat[$a<0$]{\includegraphics[width=0.35\textwidth]{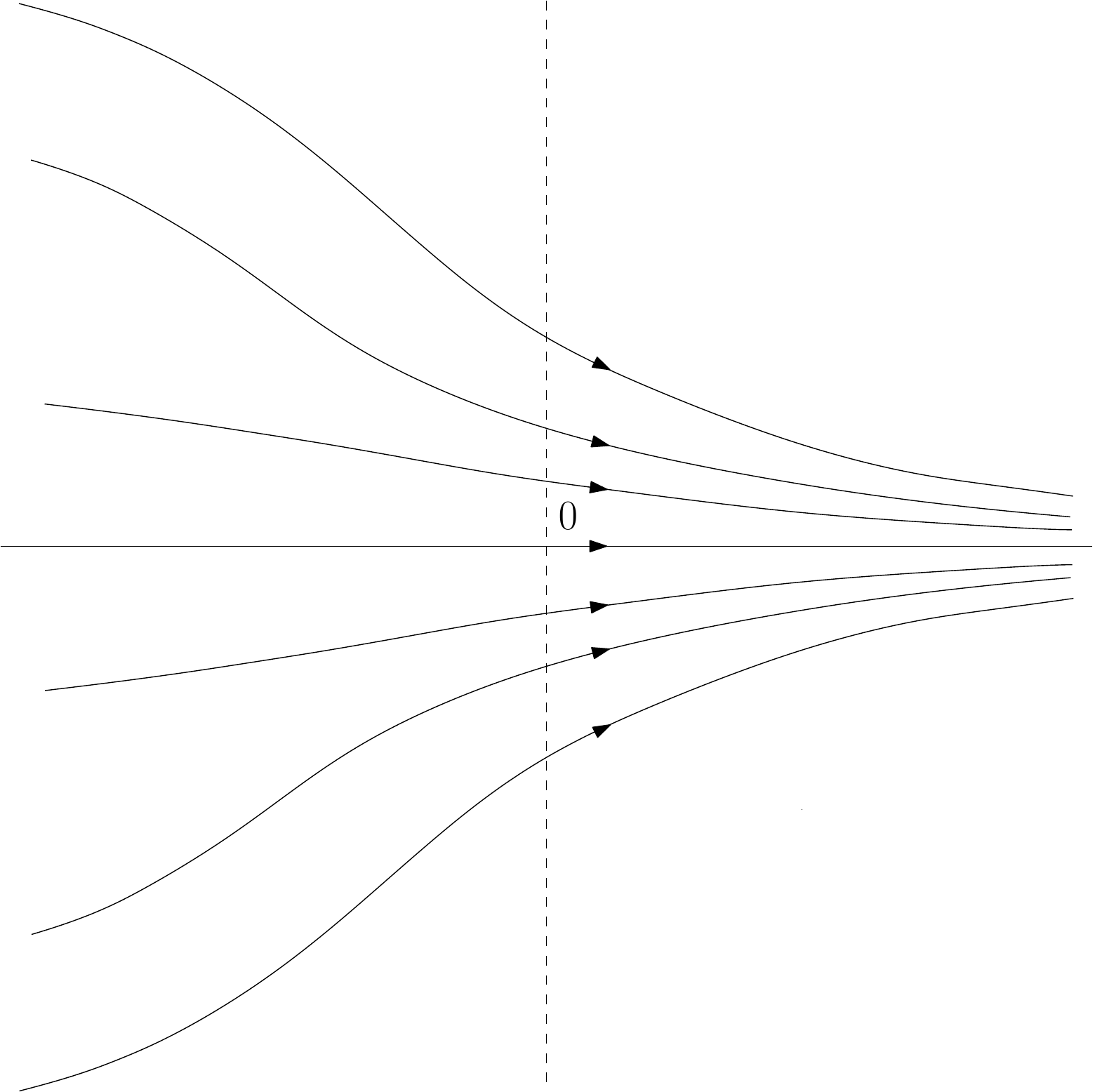}
\label{subfig:toy_model_aneg}}
\caption{Saddle-node bifurcation in $a=0$ for \eqref{eq:F_toymodel}.}%
\label{fig:toy_model}%
\end{figure}
In this particular model, the functional driving \eqref{eq:conv F of order epsilon} becomes
\begin{equation}
\label{eq:F_av_toymodel}
\int_{\bbR^2} F(u) q_{(x,y), \gs^2 K^{-1}}(u)\dd u\, =\, \left(x^2-a+ \varpi_1, -b y \right)\, .
\end{equation}
Here, the functional \eqref{eq:F_av_toymodel} is of the same nature as \eqref{eq:F_toymodel} with $a$ changed into $a- \varpi_1$. Hence, starting from $a>0$, the description of the bifurcation for \eqref{eq:F_av_toymodel} is now made in terms of $ \varpi_1<a$ or $ \varpi_1>a$: noise and interaction induce a drift to the right for the phase dynamics on $ \mathcal{ M}^\gd_L$ of \eqref{eq:PDE} as soon as $ \varpi_1= \frac{ \sigma_1^{ 2}}{ k_1}>a$, when $\gd$ is small enough. In that case, the combined effect of the noise and the interaction allows the interacting system \eqref{eq:part_syst} to collectively go over the difference of potential lying between the initial stable and unstable points.

One could imagine that farther from the origin, $F$ is such that the part of the line $\{y=0\}$ close to the origin belongs in fact to a stable closed loop for the IDS dynamics, and that this loop persists for the phase dynamics of \eqref{eq:PDE} (this could for example be the case for $a$ and $ \varpi_1$ small enough, so that the loop is only slightly perturbed away from the origin). Such a $F$ would provide a simple example of excitable system.
This is the spirit of the following concrete example, based on the Stuart Landau model \eqref{eq:stuart_landau}. 

\subsubsection{Modified Stuart Landau oscillators model:}
Let us now modify the system \eqref{eq:stuart_landau} in the following way:
\begin{equation}
\label{eq:SL_mod}
F(x,y)\, =\, \left( x(1-(x^2+y^2)) - (\omega-by) y, y(1-(x^2+y^2))+(\omega-by) x \right)\, .
\end{equation}
The IDS in this situation corresponds to the dynamics defined in polar coordinates by $\dot r=r(1-r^2)$, $\dot \theta=w-br\sin\theta$. The circle $\{r=1\}$ is invariant stable, and when $\omega>0$ and $b>\omega$ but close to $\omega$, this model is a simple example of excitable dynamics in $\bbR^2$: the point of polar coordinates $(r_0,\theta_0)=(1,\arcsin(\omega/b))$ is a stable fixed-point, and a perturbation of large enough amplitude may allow the system to go over the unstable fixed-point $(1,\pi-\arcsin(\omega/b))$ and to go back to the $(r_0,\theta_0)$ travelling along the circle $\{r=1\}$.

With this choice of $F$, \eqref{eq:PDE} can be seen as a $2$-dimensional generalization of the {\sl Active Rotators model} \cite{Sakaguchi1986,Shinomoto1986,Sakaguchi1988a} and the following is, in a sense, a generalization of the work made in \cite{doi:10.1137/110846452}, where a rigorous proof of the existence of noise induced periodic behaviors in this {\sl Active Rotators model} is given.

When it comes to \eqref{eq:conv F of order epsilon} in the case of \eqref{eq:SL_mod}, with $\varpi_1=\varpi_2=\varpi$, straightforward calculations lead to
\begin{multline}
\label{eq:SLmod_av}
\int_{\bbR^2} F(u) q_{(x,y), \gs^2 K^{-1}}(u)\dd u\, =\, \bigg(x\left(1-4 \varpi(x^2+y^2)\right) -(\omega-by) y +b \varpi ,\\ y\left(1-4 \varpi(x^2+y^2)\right) +(\omega-by) x\bigg)\, ,
\end{multline}
and the dynamics driven by \eqref{eq:SLmod_av} is given in polar coordinates by
\begin{equation}\label{eq:dyn polar pert}
\dot r\, =\, r\left(1-4 \varpi-r^2\right)-b \varpi\cos \theta\, ,\qquad  \dot \theta\, =\, \omega - b\left(r- \varpi\frac{1}{r}\right)\sin \theta\, .
\end{equation} 
Suppose now that $\omega=b(1-\boldsymbol{\zeta})$ with $\boldsymbol{\zeta}$ small.
If $\boldsymbol{\zeta}$ and $ \varpi$ are small enough the invariant manifold $\{r=1\}$ persists under perturbation, becoming an invariant curve $\{r_c(\theta)\}$ for the dynamics given by \eqref{eq:dyn polar pert}. It is easy to see that $r_c(\theta)\leq 1$ if $|b|\leq 4$, and we have $r_c(\theta)\geq 1/2$ for $ \varpi$ small enough, which means that
\begin{equation}
0 \, \leq\, r_c(\theta)- \varpi\frac{1}{r_c(\theta)}\, \leq\, 1-\frac12 \varpi\, .
\end{equation}
We deduce that \eqref{eq:dyn polar pert} admits a limit circle if $1-\frac12 \varpi<1-\boldsymbol{\zeta}$, i.e. $ \varpi$ is small enough and satisfies $ \varpi>2\boldsymbol{\zeta}$. Hence, with these choices of parameters, \eqref{eq:PDE} admits a noise-induced periodic behavior, for $\gd$ small enough. 

\begin{figure}[h]
\centering
\includegraphics[width=0.65\textwidth]{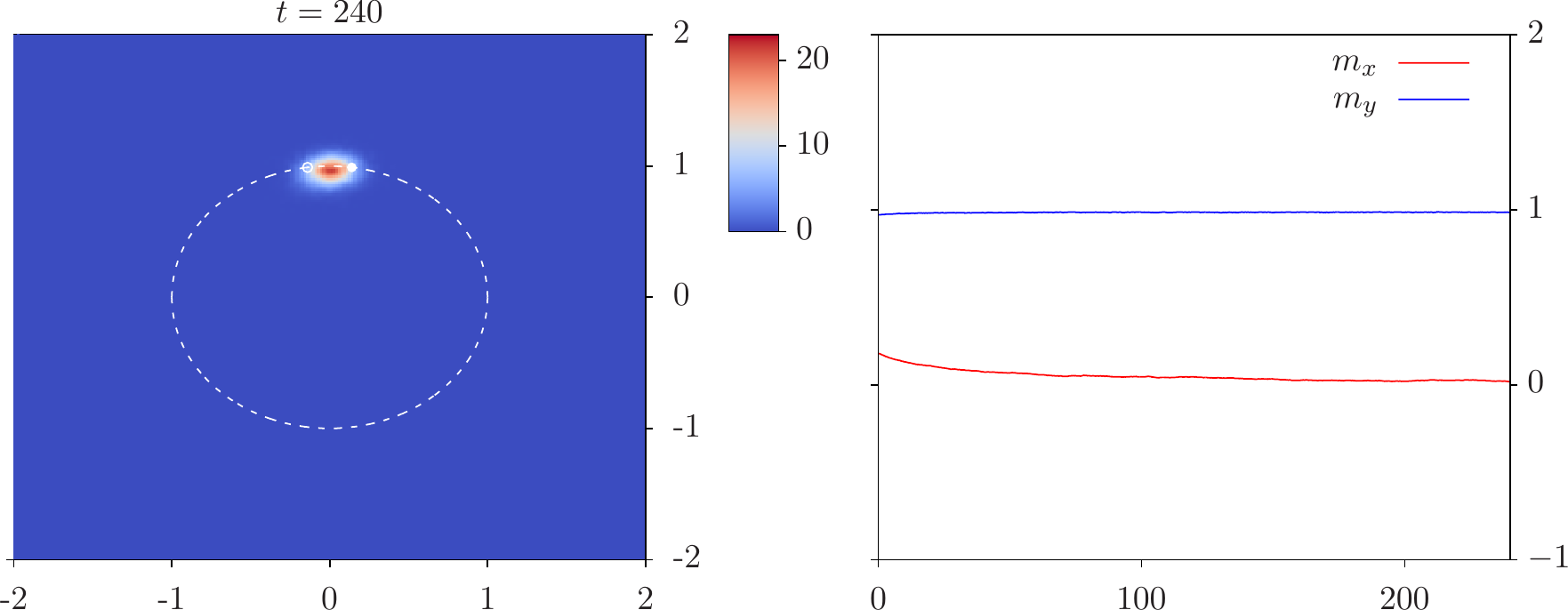}
\caption{Left: representation of the empirical measure of \eqref{eq:part_syst} for the modified Stuart-Landau model \eqref{eq:SL_mod} for $\sigma_1=\gs_2=0.1$. The limit-cycle of the IDS is represented in dotted-lines. Right: trajectories of the corresponding mean-values on $[0, 240]$. Simulations are made for $N=20000$ particles and parameters $\omega=1$, $k_1=k_2=1$, $b=1.01$ and $\delta=0.5$.}
\label{fig:SL_pt_stable}
\end{figure}
\begin{figure}[h]
\centering
\includegraphics[width=0.65\textwidth]{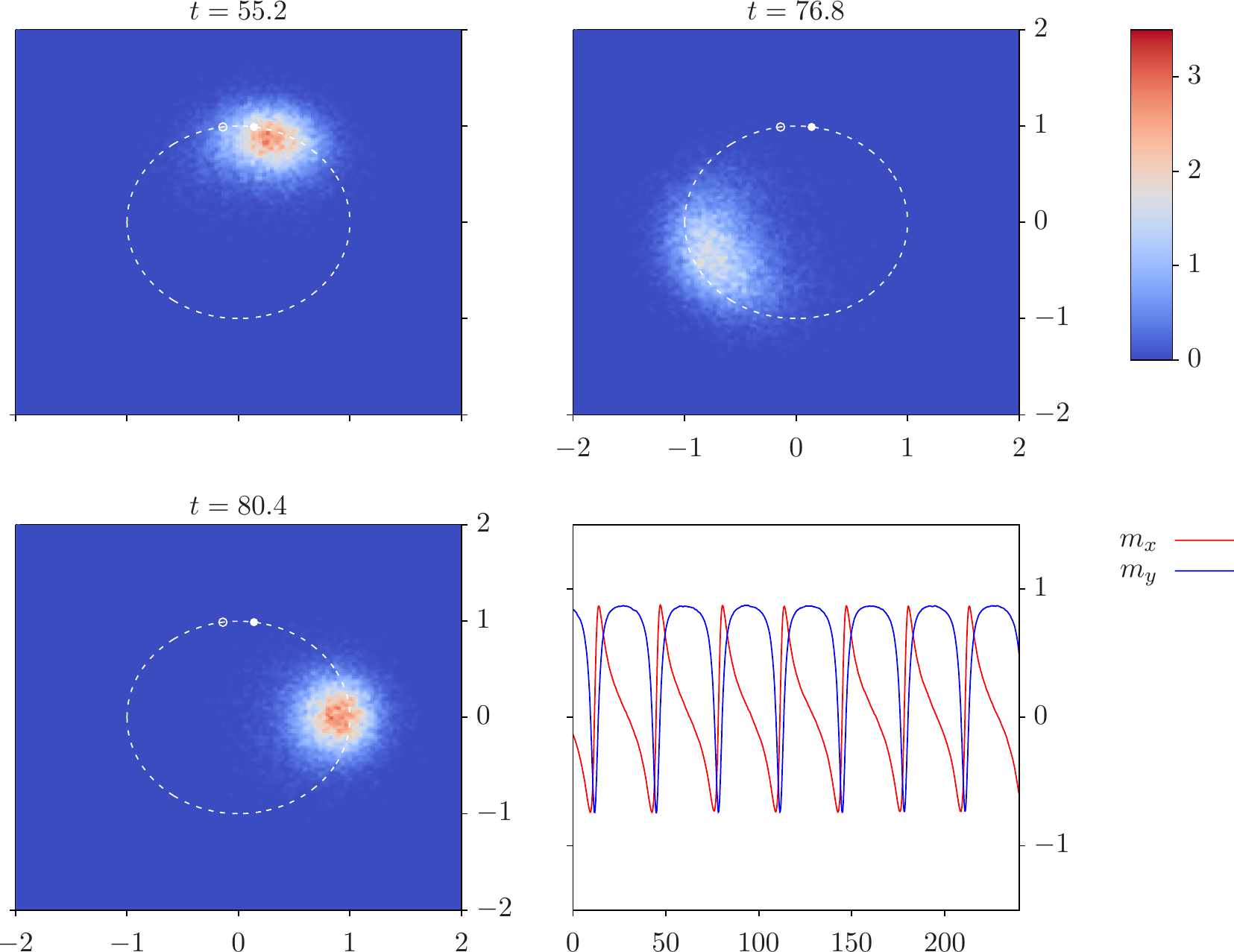}
\caption{From left to right: evolution of the empirical measure of the system \eqref{eq:part_syst}  for the modified Stuart-Landau model \eqref{eq:SL_mod} for $\sigma_1=\gs_2=0.3$. The limit-cycle of the IDS is represented in dotted-lines. Lower-right: trajectories of the corresponding mean-values on $[0, 240]$. Simulations are made for $N=20000$ particles and parameters $\omega=1$, $k_1=k_2=1$, $b=1.01$ and $\delta=0.5$.}
\label{fig:SL_cycle_limite}
\end{figure}

This transition, which corresponds to an infinite-period bifurcation (see \cite{strogatz2014nonlinear}, p.262) for the averaged system \eqref{eq:SLmod_av}, is made explicit in Figures~\ref{fig:SL_pt_stable} and~\ref{fig:SL_cycle_limite}: fixing the intensity of interactions $k_1=k_2=1$, the stationary point $(1,\arcsin(\omega/b))$ for \eqref{eq:SL_mod} remains stable for \eqref{eq:PDE} in the case of a small noise intensities (Figure~\ref{fig:SL_pt_stable}) whereas a periodic behavior appears when the noise is large enough (Figure~\ref{fig:SL_cycle_limite}).

\subsection{The case of the FitzHugh Nagumo model}
The bifurcation diagram for the dynamical system driven by \eqref{eq:FHN_u} is known to be complex (see \cite{MR1779040}). We study here two scenarios, the excitable and bistable cases, the first one being of particular interest in life sciences \cite{LINDNER2004321}.

\subsubsection{Disorder-induced limit cycles in the excitable case}

We are interested in this paragraph in the phase dynamics of \eqref{eq:PDE} for $F$ given by \eqref{eq:FHN_u} with parameters $a= \frac{ 1}{ 3}$, $b=1$, $\tau=10$. This choice of parameters corresponds to a situation where the IDS system \eqref{eq:FHN_u1} has a unique stationary point (see Figure~\ref{fig:FHN_bifurc_u}, case (a)), and is excitable, in the sense that if sufficiently perturbed while being initially at the stationary point, a trajectory of the IDS makes a whole excursion before coming back to the stationary point.

\begin{figure}[h]
\centering
\includegraphics[width=0.7\textwidth]{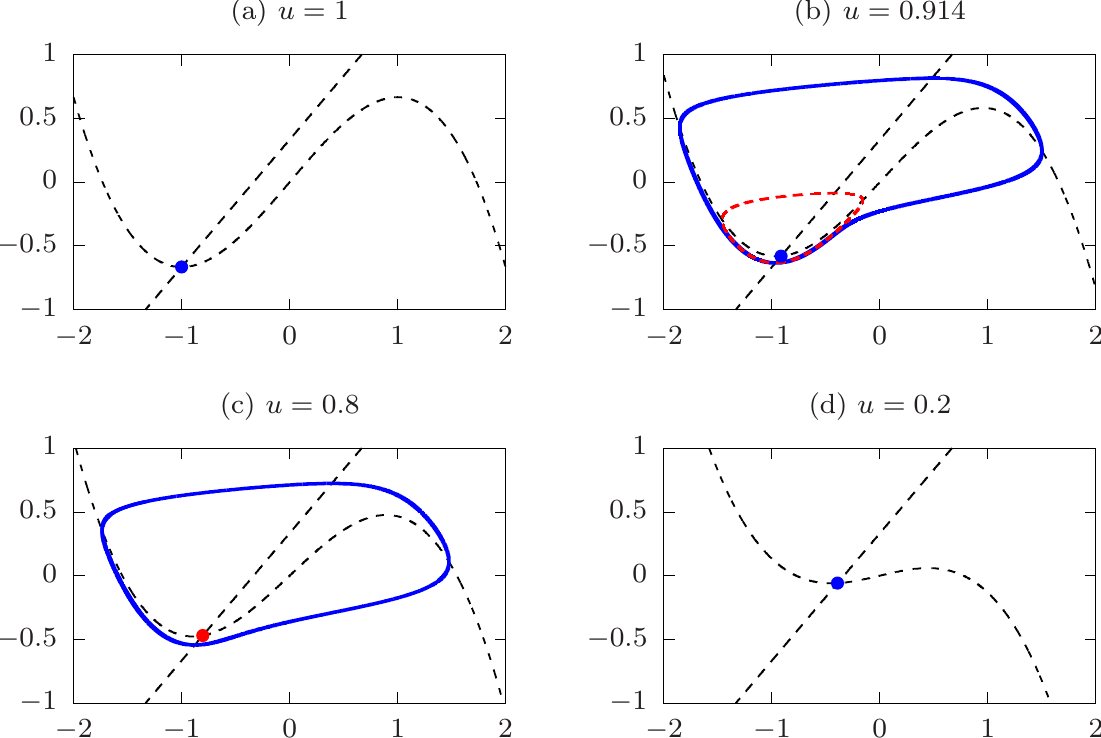}
\caption{Phase diagrams for the dynamics driven by \eqref{eq:FHN_u} for parameters $a= \frac{ 1}{ 3}$, $b=1$, $\tau=10$ and different choices of $u$. Stable (resp. unstable) points and limit cycles are represented in blue (resp. red). The nullclines $y=ux-\frac{x^3}{3}$ and $y=\frac{x+a}{b}$ are represented in black dashed lines.}
\label{fig:FHN_bifurc_u}
\end{figure}
\begin{figure}[h]
\centering
\includegraphics[width=0.6\textwidth]{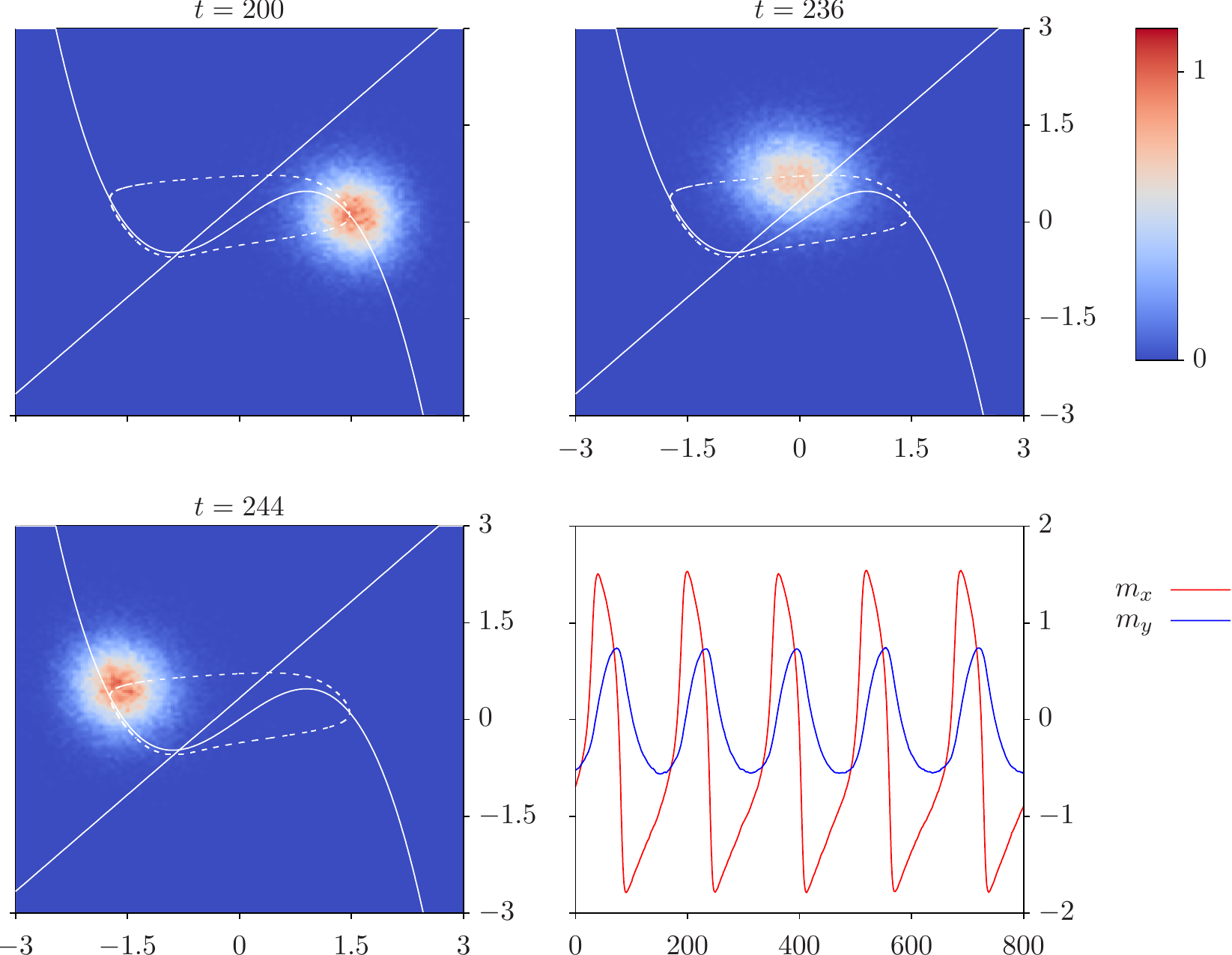}
\caption{From left to right: evolution of the empirical measure of the system \eqref{eq:part_syst} for the FitzHugh-Nagumo model \eqref{eq:FHN_u1}, with the IDS being in the excitable regime given by the parameter $a=\frac13$, $b=1$ and $\tau=10$. The theoretical limit-cycle (given by \eqref{eq:FHN_u}) is represented in dotted-lines. Lower-right: trajectories of the corresponding mean-values on $[0, 800]$. Simulations are made for $N=50000$ particles and parameters $k_1=k_2=1$, $\gs^2_1=\gs^2_2=0.2$ and $\delta=0.2$.}
\label{fig:u02}
\end{figure}

Recalling \eqref{eq:FHN_transf}, the dynamics of \eqref{eq:PDE} on the invariant manifold given by Theorem \ref{th:main1} depends, at first order in $\gd$, only on the parameter $\varpi_1$, and is given by a FitzHugh-Nagumo dynamics defined by $F_{u,a,b,\tau}$ (recall \eqref{eq:FHN_u}) with $u=1-\varpi_1$. We are thus interested in the dynamics of the FitzHugh-Nagumo model with varying parameter $u$, smaller values of $u$ corresponding to a larger ratio  intensity of the noise $\gs_1$ over interaction intensity $k_1$.

The different types of dynamics of \eqref{eq:FHN_u} obtained by tuning the parameter $u$ are represented in Figure~\ref{fig:FHN_bifurc_u}. Starting from the fixed-point dynamics of $u=1$ (case (a)), a saddle-node bifurcation of cycles then occurs (numerically estimated at $u\approx 0.91435$) after which a stable point and a stable cycle coexist, separated by an unstable cycle (case (b)). Then, at $u\approx 0.88604$ the stable point and the unstable cycle collide in a subcritical Andronov-Hopf bifurcation. The dynamics is then given by a limit cycle surrounding an unstable point (case (c)), until the supercritical Andronov-Hopf bifurcation at $u\approx 0.28383$, after which the dynamics is again given by a fixed-point (case (d)). While the saddle node bifurcation of cycles is estimated by simulating trajectories of \eqref{eq:FHN_transf}, the Andronov-Hopf bifurcations can be obtained by computing explicitly the fixed points and the eigenvalues of the linearized dynamics around these fixed points. More precisely a computation shows that the fixed point $(x_0(u),y_0(u))$ of \eqref{eq:FHN_transf} satisfies, for $u<1$,
\begin{equation}
x_0(u)\, =\, \left(\frac{\sqrt{5-12u+12u^2-4u^3}}{2}- \frac12 \right)^{\frac13}-\frac{1-u}{\left(\frac{\sqrt{5-12u+12u^2-4u^3}}{2}- \frac12 \right)^{\frac13}}\, .
\end{equation}
The Andronov-Hopf bifurcations occur at the values $u$ such that the matrix defining the linearized dynamics around $(x_0(u), y_0(u))$, which is given by $\left(\begin{array}{cc} u-x_0^2(u) & -1\\ \frac{1}{10} & -\frac{1}{10}\end{array}\right)$, has eigenvalues with real part equal to $0$. Here the eigenvalues are
\begin{equation}
\gl_{\pm}\, =\, \frac{-1+10(u-x_0^2)\pm i\sqrt{39-100(u-x_0^2)^2-29(u-x_0^2)}}{20}\, ,
\end{equation}
so this is exactly the case when $u-x^2_0(u)=\frac{1}{10}$, which occurs for the approximated values of $u$ given above (obtained by a computation on the software MAXIMA). At the bifurcation points we are precisely in the situation described in \cite{MR2263523}, page 213, exercise 16, with $\mu=\frac{1}{10}$, $F(x)=ux-\frac{x^3}{3}$ and $G(x)=x+\frac13$. For $u\approx 0.88604$ we obtain $a\approx 0.09334$ (see the definition of $a$ in \cite{MR2263523}), which indicates a subcritical Andronov-Hopf bifurcation, and for $u\approx 0.28383$ we get $a\approx -0.07394$, which indicates a supercritical Andronov-Hopf bifuration.

So in particular the cases (b) and (c) show that for an accurate choice of parameters $\gs_1$ and $k_1$ and for $\gd$ taken small enough, the PDE \eqref{eq:PDE} has a periodic behavior induced by the combined effect of noise and interaction. This is illustrated in Figure~\ref{fig:u02}.
Note that with this choice of parameters the IDS system \eqref{eq:IDS} is itself close to a saddle node bifurcation of cycles: keeping $b=1$ and $\tau=10$ this bifurcation occurs at $a\approx 0.29645$, and is then followed by a subcritical Andronov-Hopf bifurcation at $a\approx 0.28460$.
So in this sense this situation is somewhat close to the example given at the beginning of Section \ref{sec:emergence_dynamics_perturbation}.

This phenomenon of emergence of structured dynamics induced by noise and interaction is not observed for all values of parameters $a$, $b$ and $\tau$ putting \eqref{eq:IDS} close to a bifurcation point. For example for $a=1$, $b=0$ and $\tau=10$ the system is close to a supercritical Andronov-Hopf bifurcation, and no such phenomenon is observed for \eqref{eq:FHN_u} (but note that structured dynamics can be observed in this situation when no noise is present on the $x$ coordiante, and no interaction on the $y$ coordinate, see \cite{LINDNER2004321} page 383).

\subsubsection{Emergence of oscillatory behavior in the bistable case.}
We consider in this paragraph the phase dynamics of \eqref{eq:PDE} for $F$ given by \eqref{eq:FHN_u} for the parameters $a= 0$, $b=1.45$ and $\tau=10$. This choice of parameters corresponds to a situation where the system \eqref{eq:FHN_u1} has two symmetric attractive equilibria and a saddle point (see Figure~\ref{fig:traj_bistable}, case (a)).

\begin{figure}[h]
\centering
\includegraphics[width=0.7\textwidth]{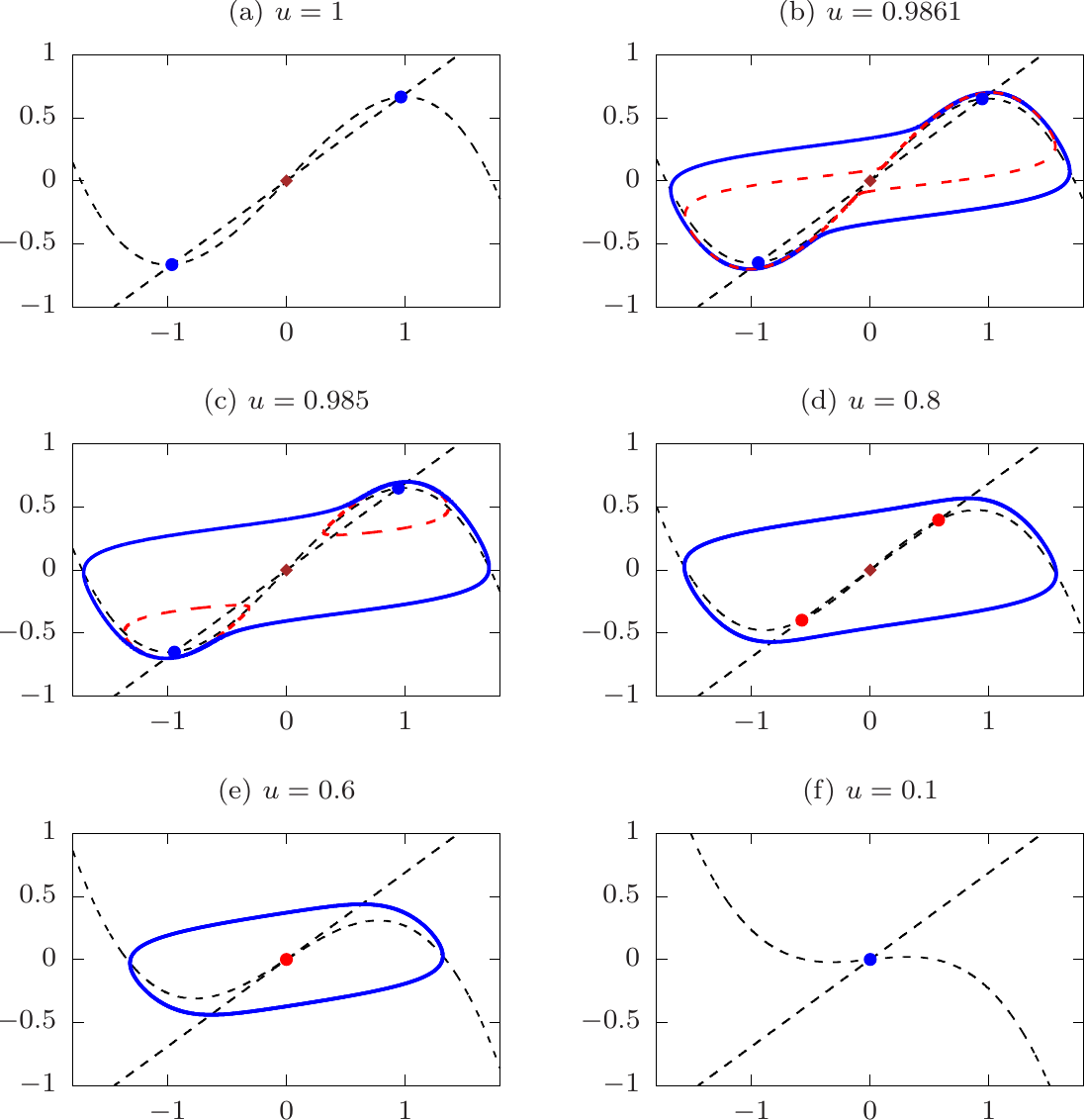}
\caption{Phase diagramms for the dynamics driven by \eqref{eq:FHN_u} for parameters $a= 0$, $b=1.45$ and $\tau=10$ and different choices of $u$. Stable (resp. unstable) points and limit cycles are represented in blue (resp. red), saddle points are represented in brown.}
\label{fig:traj_bistable}
\end{figure}

\begin{figure}[h]
\centering
\includegraphics[width=0.6\textwidth]{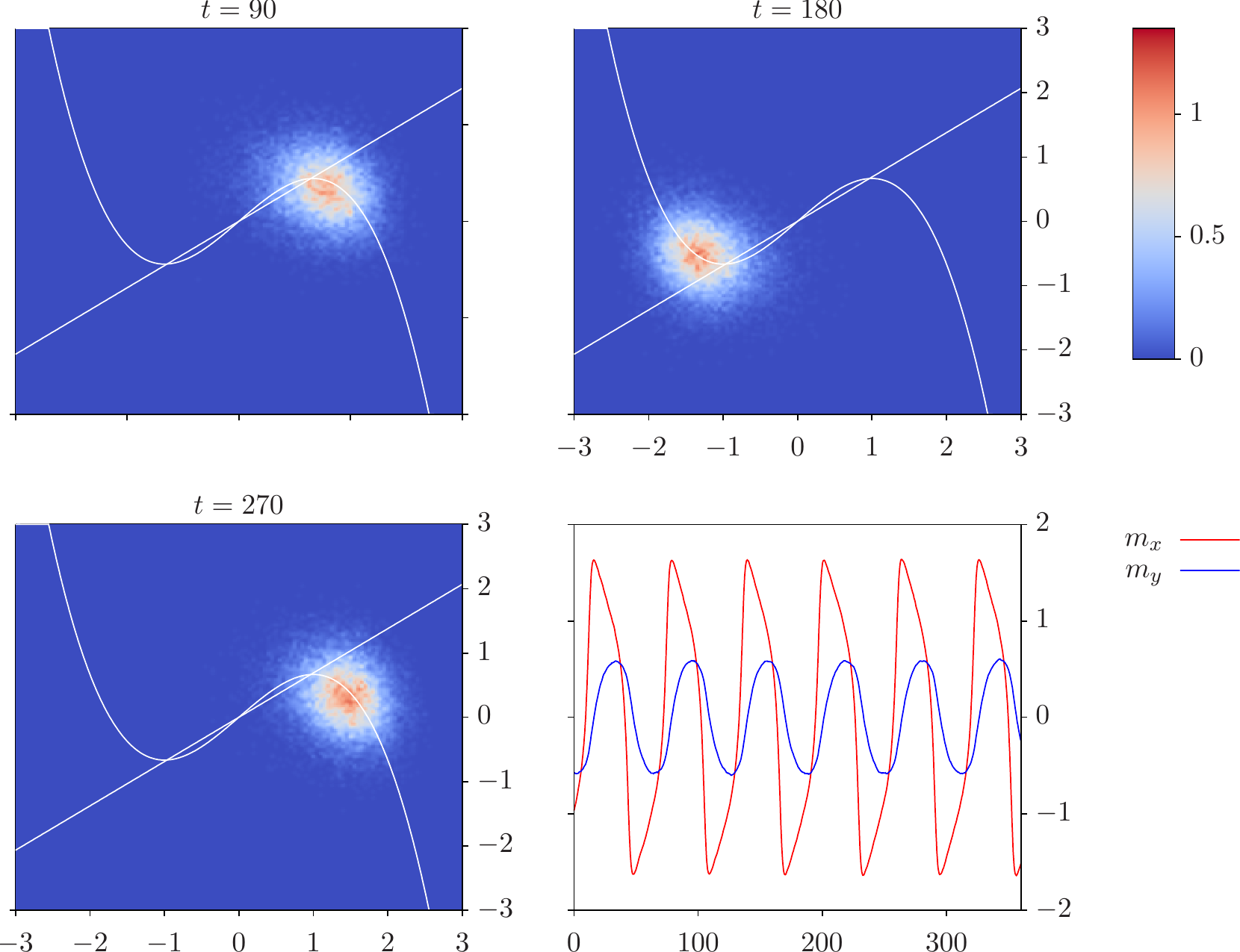}
\caption{From left to right: evolution of the empirical measure of the system \eqref{eq:part_syst} for the FitzHugh-Nagumo model \eqref{eq:FHN_u1}, when the IDS is in a bistable regime given by the parameters $a=0$, $b=1.45$ and $\tau=10$. Lower-right: trajectories of the corresponding mean-values on $[0, 360]$. Simulations are made for $N=50000$ particles and parameters $k_1=k_2=1$, $\gs_1=\gs_2=0.45$ and $\delta=0.5$.}
\label{fig:FHN_bistable_cycle_limite}
\end{figure}

The dynamics of \eqref{eq:FHN_u} obtained by tuning $u$ in this situation are represented in Figure~\ref{fig:traj_bistable}. Starting from the bistable case of $u=1$ (case a), the systems undergoes a saddle-node bifurcation of cycles at $u\approx 0.98614$, after which a stable and an unstable cycle surround the fixed-points (case (b)). Then the unstable cycle splits into two unstable cycles while colliding with the saddle point at a double homoclinic bifurcation point ($u\approx 0.98604$). For the values of $u$ following this bifurcation point, each unstable cycle surrounds a stable point (case (c)). At $u\approx 0.96198$, a double subcritical Adronov-Hopf bifurcation occurs, so that the stable cycle surrounds the saddle point and two unstable points (case (d)). At $u\approx 0.68966$, the three fixed-points collide in a pitchfork bifurcation, and only one unstable point remains (case (e)). Finally, a suppercritical Andronov-Hopf bifurcation occurs at $u = 0.145$, and then only one stable fixed-point remains (case (f)).

Once again, this shows that the addition of noise and interaction may induce a transition, from the bistable regime to an oscillatory behavior. This phenomenon is illustrated in Figure~\ref{fig:FHN_bistable_cycle_limite}.

\subsection{Another example of structured dynamics induced by noise and interaction: the Cucker-Smale model for collective dynamics}
A large number of models for collective alignment have been proposed in the literature (e.g. models of phase oscillators \cite{Kuramoto1975,Sakaguchi1986, Sakaguchi1988a,Shinomoto1986}, the Vicsek model \cite{MR2855983, Degond:2012uq}, etc.). We consider in this paragraph the Cucker-Smale model with self-propulsion, as proposed in \cite{MR3541988}. In the spatially-homogeneous case,  the point is to consider a mean-field model as \eqref{eq:PDE}, where the state variable (denoted as $v\in \mathbb{ R}^{ d}$ instead of $x$) represents the typical velocity $v$ of a particle. The intrinsic dynamics is here given by
\begin{equation}
\label{eq:F_CSM}
F(v)= v(1- \left\vert v \right\vert^{ 2}), \ v=(v_{ 1}, \ldots, v_{ d})\in \mathbb{ R}^{ d},\ .
\end{equation}
In absence of noise and interaction, the dynamics of the corresponding IDS \eqref{eq:IDS} is simple: each trajectory is exponentially attracted to the sphere $ \mathbb{ S}^{ d-1}$. 

The model \eqref{eq:PDE} driven by \eqref{eq:F_CSM} with noise and interaction is prototypical of mean field dynamics with non-convex potentials (in the local dynamics term and/or in the interaction), see in particular the example of granular media type equations (\cite{Bolley:2012fk, Durmus2018} and references therein).

As shown in \cite{MR3541988}, the stationary states $ \mu$ for the whole dynamics \eqref{eq:PDE} can be explicitly computed in this case in terms of the stationary mean velocity $\bar u:= \int v \mu({\rm d}v)$. Choosing the axis appropriately, one can always suppose that $\bar u$ points in the direction of the first vector $e_{ 1}$ of the canonical basis, so that $\bar u= u e_{ 1}$ may only be understood through its magnitude $u=u(\sigma)$, which solves an appropriate fixed-point relation (see \cite{MR3541988}, Eq. (7), p. 1067). The main results of \cite{MR3541988} concern the existence and characterization of such fixed-points, in the regime of small and large noise (considering the same noise intensity $\sigma$ and the same interaction parametre $k$ on each coordinate): for small noise, synchronization occurs (characterized by nontrivial $u(\sigma)$ such that $u(\sigma)\to1$ as $ \sigma\to 0$) whereas for large noise, $u(\sigma)=0$ is the only fixed-point. Numerical simulations in \cite{MR3541988} suggest that this phase transition occurs precisely at some $ \sigma_{ c}>0$, for which no explicit formula is known. 

Applying Theorem~\ref{th:main2} in this situation gives an explicit expression for $ \sigma_{ c}$ and $ \sigma \mapsto u( \sigma)$ in the regime where $ \delta$ (denoted as $ \alpha$ in \cite{MR3541988}) goes to $0$. Namely, choosing $k=1$, the Gaussian convolution of \eqref{eq:conv F of order epsilon} is here given by
\begin{equation}
\int_{\bbR^d} F(z) q_{v, \gs^2I_d}(z) \dd z= v \left( 1- \sigma^{ 2}(2+d) - \left\vert v \right\vert^{ 2}\right).
\end{equation}
We see here that the system \eqref{eq:conv F of order epsilon} exhibits two different behaviors. On one hand, when $ \sigma^{ 2}\geq \sigma_{ c}^{ 2}:= \frac{ 1}{ 2+d}$, its trajectories are attracted to $0$, and thus for $\gd$ small enough (depending on $\gs$) the PDE \eqref{eq:PDE} admits a stationnary solution, that is stable in the sense of Remark \ref{rem:stab} if $ \sigma^{ 2}\geq \sigma_{ c}^{ 2}$.
On the other hand, if $ \sigma^{ 2} < \sigma_{ c}^{ 2}$, the trajectories of \eqref{eq:conv F of order epsilon} concentrate on the attracting sphere of radius $ u(\sigma) := \sqrt{ 1- \sigma^{ 2}(2+d)}$, which means that for $\gd$ small enough \eqref{eq:dot m th} admits also an attracting manifold (which is again a sphere centered at the origin, by invariance by rotation of the problem), and thus in this case the PDE \eqref{eq:PDE} admits a sphere of stationnary solutions: there is synchronization. Note that, in the case $d=1$, this value of $ \sigma_{ c}^{ 2}$ coincides with the limit as $ \delta\to 0$ found in \cite{MR3180036}. 
\section{ Well-posedness results and a-priori estimates}
\label{sec:well_posed_apriori}
\subsection{Modification of the dynamics}
\label{sec:modification dynamics}

The proof we provide in the paper follows the classical steps of the proofs of persistence of normally hyperbolic manifold \cite{fenichel1971persistence,hirsch1977invariant,Bates1998,sell2013dynamics,wiggins2013normally}. To define the mapping for which the perturbed manifold is a fixed-point (see Section~\ref{sec:fixed_point_f}), we will need that the trajectories close to $\cM^0$ and with mean $m_t$ that escape some ball surrounding $\cV$ (see Lemma \ref{lem:bij}). In our case, the initial manifold $ \mathcal{ M}^{ 0}$ is static, so we modify artificially the dynamics given by \eqref{eq:slow fast PDE} to obtain this property, and in such a way that a trajectory with $m_t\in \cV$ is not affected by this artificial modification, so that the invariant manifold $ \mathcal{ M}^\gd$ that we obtain also remains unaffected.

Suppose that $\cV\subset B_L$ for a $L>0$ and consider a smooth mapping $h$ satisfying $h(x)=0$ for $\vert x\vert \leq L$ and $\vert x\vert \geq 3L$, and such that $ h(x)\cdot  x \geq c_h>0$ for $\vert x\vert =2L$ and $\vert h(x)\vert\leq C_h$ and $\vert h(x)-h(y)\vert \leq C_h \vert x-y\vert$ for some $C_h>0$. We tune $c_h$ such that
\begin{equation}\label{eq:hyp h}
c_h +\min_{\vert m\vert =2L} \left\{n(m) \cdot \int_{\bbR^d}F(x) q_{ m, \sigma^{ 2}K^{ -1}}(x) {\rm d}x\right\}\, \geq\, \xi\, ,
\end{equation}
for some $\xi>0$, where $n(m):= \frac{ m}{ \left\vert  m \right\vert}$.  For this choice of $h$, we consider the modified dynamics given by the McKean process
\begin{equation}\label{eq:modif mc kean}
\dd X_t\, =\, \Big(\delta F(X_t) + \delta h\left(\bbE [X_t]\right)-K\big(X_t-\bbE[X_t]\big)\Big)\dd t+\sqrt{2}\gs \dd B_t,\ t\geq0 \, ,
\end{equation}
whose distribution solves the PDE
\begin{multline}
\label{eq:PDE_mod_h}
\partial_t \mu_t\, =\, \nabla\cdot( \gs^2 \nabla \mu_t) + \nabla \cdot\left(K \mu_t\left(x-\int_{\bbR^d}z\mu_t({\rm d}z)\right)\right)\\-\delta \nabla\cdot \left(\mu_t \left(F(x) + h \left(\int_{ \mathbb{ R}^{ d}} z \mu_{ t}({\rm d}z)\right)\right)\right),\ t\geq0\, .
\end{multline}
The corresponding slow fast system is then given by
\begin{equation}\label{eq:slow fast PDE modified}
\left\{
\begin{array}{rl}
\partial_t p_t(x)& =\, \cL p_t(x)+ \nabla\cdot (p_t(x)(\dot m_t -\gd F_(x)-\gd h(m_t)))\\
\dot m_t& =\,  \gd \int_{\bbR^d} F(x+m_t) p_t(\dd x)+\gd h(m_t)
\end{array}
\right. \, , \,  t\geq 0\, ,
\end{equation}
where we have used the notations $ \mathcal{ L}$ for the Ornstein-Uhlenbeck operator defined by
\begin{equation}
\label{eq:cLtheta}
\cL f\, =\,  \nabla\cdot (\gs^2 \nabla f)+ \nabla\cdot (Kxf)\, ,
\end{equation}
and
\begin{equation}
\label{eq:F_t}
F_t(x)=F(x+m_t),\ t\geq0\,.
\end{equation}
A precise notion of solution for these equations is given in Definition~\ref{def:weak_sol} below.
\subsection{Existence and uniqueness for the nonlinear process}
We denote by $\cP_{2}= \mathcal{ P}_{ 2}(\mathbb{ R}^{ d})$ the space of probability measure on $\bbR^d$ endowed with the Wasserstein distance $W_{2}$:
\begin{equation}
\label{eq:W2}
W_{ 2}(\xi_{ 1}, \xi_{ 2}):= \left\lbrace \inf_{ X_{ 1}\sim \xi_{ 1}, X_{ 2}\sim \xi_{ 2}} \mathbb{ E}( \left\vert X_{ 1}-X_{ 2} \right\vert^{ 2})\right\rbrace^{ \frac{ 1}{ 2}},\ \xi_{ 1}, \xi_{ 2}\in \mathcal{ P}_{ 2}\,.
\end{equation}
\begin{lemma}\label{lem:pathwise existence}
For any $ \delta\geq 0$, $T>0$, any $X_0$ with distribution $\mu_0\in\cP_2$, equation \eqref{eq:modif mc kean} has unique pathwise solution $X_\cdot$ on $[0,T]$, the distribution $\mu_t$ of $X_t$ is element of $\cP_2$ for all $t\in[0,T]$ and the following bound on the mean-value $m_{t}^{ \delta}=m_{ t}:= \mathbb{ E} \left[X_{ t}\right], t\in[0, T]$, holds for any $\gd'>0$:
\begin{equation}
\label{eq:m_inf_delta}
\sup_{ \delta\leq \delta'}\sup_{ s\in[0, T]} \left\vert m_{ s}^{ \delta}\right\vert < +\infty.
\end{equation}
Secondly, for any $ \alpha\in \left(0, 1\right)$ such that $\bbE\left[w_{ \alpha}(X_0)\right]<\infty$ (recall \eqref{eq:w_alpha}), there exists a constant $ \kappa_{0}>0$  such that
\begin{align}
\sup_{ t\in[0, T]} \mathbb{ E} \left[ w_{ \alpha}(X_t)\right] &\leq \max \left( \kappa_{0}, \mathbb{ E} \left[ w_{ \alpha}(X_0)\right]\right).\label{eq:exp_control_Xt}
\end{align}
The constant $ \kappa_{0}$ in \eqref{eq:exp_control_Xt}, depends in particular on $\alpha$, $\sup_{ s\leq T} \left\vert m_{ s} \right\vert$ and $\gd$, but for a fixed $\gd'>0$ may be chosen independent of $ \delta \leq \gd'$. Moreover, we have
\begin{equation}
\label{eq:dot m_inf_delta}
\sup_{ \delta\leq \delta'}\sup_{ s\in[0, T]} \left\vert \dot m_{ s}^{ \delta} \right\vert < +\infty.
\end{equation}
\end{lemma}
\begin{proof}[Proof of Lemma~\ref{lem:pathwise existence}]
The proof of existence and uniqueness for \eqref{eq:modif mc kean} is similar to the ones given in \cite{22657695,SznitSflour}. We consider the space $\cM_T$ of probability measures on
$ \mathcal{ C}= \mathcal{ C}([0,T],\bbR^d)$, endowed with the Wasserstein distance
\begin{equation}
W_{2,T}(\pi_1,\pi_2)\, :=\, \left\{\inf_{\gl\in \Lambda( \lambda_1, \lambda_2)}\int_{\cC^2}\sup_{0\leq t\leq T}\left\vert x^1_t-x^2_t\right\vert^{2} \gl(\dd x^1,\dd x^2)\right\}^{\frac{1}{2}}\, ,
\end{equation}
where $\Lambda( \lambda_1,\lambda_2)$ denotes the set of couplings between $ \lambda_1$ and $ \lambda_2$. For any $ \lambda$ element of $\cM_T$, denote by $ \bar \lambda_t:=\int_{ \mathbb{ R}^{ d}} x \lambda_t( {\rm d}x)$ its mean-value and consider the equation
\begin{equation}\label{eq: eds mt parameter}
\dd Y_t\, =\, \Big(\delta F(Y_t) + \delta h( \bar \lambda_t)-K\big(Y_t- \bar \lambda_{ t}\big)\Big)\dd t+\sqrt{2}\gs \dd B_t\, ,
\end{equation}
with initial condition $Y_{ 0}=X_0$. Let us show that \eqref{eq: eds mt parameter} defines a process whose distribution belongs to $\cM_T$. To prove the non-explosion of $Y_\cdot$, let us consider the stopping times $\tau_n=\inf\{t>0,\vert Y_t\vert> n\}$ and the process $Y_\cdot^{(n)}=Y_{\cdot\wedge \tau_n}$. The process $Y^{(n)}_\cdot$ is well defined and satisfies
\begin{multline}
\vert Y^{(n)}_t \vert^2_{K\gs^{-2}}=\vert X_0\vert^2_{K\gs^{-2}}\\+\int_0^{t\wedge \tau_n}\Big( 2 K\gs^{-2}Y^{(n)}_s\cdot \big(\delta(F(Y^{(n)}_s)+h(\bar \lambda_s))-K(Y^{(n)}_s-\bar \lambda_s)\big)+   \Tr(K)\Big)\dd s\\
+\sum_{i=1}^d \frac{k_i}{\gs_i}\int_0^{t\wedge \tau_n} Y^{(n),i}_s\dd B^i_s\, .
\end{multline}
Since $ \lambda\in \mathcal{ M}_{ T}$, there exists $C_{ \lambda}>0$ such that $\sup_{ s\leq t}\vert \bar \lambda_s\vert \leq C_{ \lambda}$. Using this and the fact that $K\gs^ {-2} x\cdot F(x)\leq C_F$ (recall \eqref{hyp:F_dot_x_bound}), we get, using Burkholder-Davis-Gundy inequality,
\begin{equation}
\bbE\left[\sup_{s\leq t}\vert Y^{(n)}_s\vert^2_{K\gs^{-2}}\right]\, \leq\, \bbE[\vert X_0\vert^2_{K\gs^{-2}}]+C_{ \lambda,F}\int_0^{t}\bbE\left[\sup_{u\leq s}\vert Y^{(n)}_u\vert^2_{K\gs^{-2}}\right]\dd s\, ,
\end{equation}
which means that $\bbE\left[\sup_{s\leq t}\vert Y^{(n)}\vert^2_{K\gs^{-2}}\right]\leq (\bbE[\vert X_0\vert^2_{K\gs^{-2}}] + C_{ \lambda, F})e^{C_{ \lambda,F}t}$ for some $C_{ \lambda, F}>0$. The fact that this bound does not depend on $n$ implies that $\bbP(\tau_n\leq t)\rightarrow 0$ and so that it is valid for $Y_\cdot$, by Fatou Lemma. Hence, we can define the mapping $\Phi$ that associates $ \lambda$ to the distribution $ \Phi(\lambda)$ of $Y_\cdot$.

Considering now two distributions $ \lambda^1$ and $ \lambda^2$ satisfying $ \lambda^1_0= \lambda^2_0=\mu_0$, with respective means $\bar  \lambda^1_\cdot$ and $\bar  \lambda^2_\cdot$ and respective solutions $Y^1_\cdot$ and $Y^2_\cdot$ of \eqref{eq: eds mt parameter}, we obtain, for some constant $C_{ F, h, \delta}>0$ that remains bounded as $ \delta\to 0$, 
\begin{multline}
\sup_{s\leq t}\vert Y^1_t-Y^2_t\vert^2\, \leq \, 2\int_0^t\bigg|(Y^1_s-Y^2_s)\cdot \Big(\delta (F(Y^1_s)-F(Y^2_s))+\delta (h(\bar \lambda^1_s)-h(\bar \lambda^2_s))\\
-K(Y^1_s-Y^2_s)+K(\bar \lambda^1_s-\bar \lambda^2_s)\Big)\bigg|\dd s\\
\leq\, C_{F,h,\delta}\int_0^t \left(\sup_{u\leq s} \vert Y^1_u-Y^2_u\vert^2 +W^2_{2,s}(\lambda^1, \lambda^2)\right)\dd s\, ,
\end{multline}
where we have used \eqref{hyp:F_Lipschitz}, the Lipschitz continuity of $h$ and $\vert \bar \lambda^1_s- \bar \lambda^2_s\vert \leq W_{2,s}(\lambda^1,\lambda^2)$. By Gr\"onwall's inequality, this leads directly to the bound
\begin{equation}
W_{2,T}^2(\Phi( \lambda^1),\Phi( \lambda^2))\, \leq\,  e^{C_{F,h,\delta}T}\int_0^TW^2_{2,s}( \lambda^1, \lambda^2)\dd s\, .
\end{equation}
It implies directly uniqueness of solutions of \eqref{eq:modif mc kean} and existence of a unique fixed-point $ \mu= \Phi(\mu)$ follows by iteration, as done in \cite{SznitSflour}. We denote by $X_{ \cdot}$ the corresponding process with law $ \mu$, which solves \eqref{eq:modif mc kean}. A byproduct of the previous calculations is $\bbE\left[\sup_{s\leq T}\vert X\vert^2\right]\leq \bbE[\vert X_0\vert^2]e^{CT}$ for some constant $C>0$, independent of $ \delta$ and we deduce immediately \eqref{eq:m_inf_delta}.

To prove the existence of exponential moments, fix $ \alpha \in \left(0,1\right)$ and remark that
\begin{multline}
\frac{\dd}{\dd t} \bbE\left[w_{\alpha}(X^{(n)}_t)\right]\,=\, \alpha \bbE\Big[w_{\alpha}(X^{(n)}_t)\Big(-(1- \alpha)\left\vert K^{\frac12} X_{ t}^{ (n)} \right\vert_{K\gs^{-2}}^{ 2}\\ +\big(\delta F(X^{(n)}_t)+\delta h(m_t)+K m_t)\big)\cdot K\gs^{-2}X^{(n)}_t + \Tr(K)\Big)\Big]\, .
\end{multline}
Using \eqref{hyp:F_dot_x_bound}, we obtain
\begin{equation}
\frac{\dd}{\dd t} \bbE\left[w_{\alpha}(X^{(n)}_t)\right]\, \leq\,  \mathbb{ E} \left[ p \left( \left\vert X_{ t}^{ (n)} \right\vert_{K\gs^{-2}}\right)w_{ \alpha}(X^{(n)}_t)\right]
\end{equation}
for the second degree polynomial 
\begin{equation}
p(u)\, =\, - \alpha(1- \alpha) \underline{ k}u^{ 2} + \alpha(\delta \left\Vert h \right\Vert_{ \infty}\Tr(K \sigma^{ -2})^{ \frac{ 1}{ 2}} + \frac{\bar k^2}{\underline{\gs}^2}\sup_{ s\leq T} \left\vert m_{ s} \right\vert) u + \alpha\Tr(K)+ \ga\delta C_{ F}\, .
\end{equation}
Applying Lemma~\ref{lem:rough_bound_q}, \eqref{eq:rough_bound_q} to $q(u)=p(u)+1$, we obtain (recall that $1-\ga>0$),
\begin{equation}
\label{eq:bound_exp_Xt_C12}
\frac{\dd}{\dd t} \bbE\left[w_{ \alpha}(X^{(n)}_t)\right]\leq \kappa_{ 0} - \bbE\left[w_{ \alpha}(X^{(n)}_t)\right]\, ,
\end{equation}
for $ \kappa_{ 0}= \kappa_{ 0} \left(\sup_{ s\leq T} \left\vert m_{ s} \right\vert_{K\gs^{-2}}, \delta\right)>0$ given by the following equality (since we have $\Tr(K)\leq d\bar k$ and $\Tr(K \sigma^{ -2})^{ \frac{ 1}{ 2}}\leq \sqrt{ d \bar k} \underline{ \sigma}^{ -1}$)
\begin{multline}
\label{eq:kappa1}
\kappa_{ 0}(l, \delta):=\left(\frac{ \alpha\left(\delta \left\Vert h \right\Vert_{ \infty}\left(d \bar k\right)^{ \frac{ 1}{ 2}}\underline{ \sigma}^{ -1} + \frac{\bar k^2}{\underline{\gs}^2} l\right)^{ 2}}{ 4(1-\ga)\underline{ k}} +1+ \ga d\bar k+ \alpha\delta C_{ F}\right)\\ \exp\left(\frac{1}{ 2(1-\ga) \underline{ k}} \left( \frac{ \alpha \left(\delta \left\Vert h \right\Vert_{ \infty} \left(d \bar k\right)^{ \frac{ 1}{ 2}} \underline{ \sigma}^{ -1} + \frac{\bar k^2}{\underline{\gs}^2} l\right)^{ 2}}{(1-\ga) \underline{ k}}+2 (1+ \ga d\bar k+ \alpha\delta C_{ F})\right)\right)
\end{multline}
By \eqref{eq:bound_exp_Xt_C12},
\begin{equation}
\sup_{ t\in[0, T]} \bbE\left[w_{ \alpha}(X^{(n)}_t)\right] \leq \max \left(\kappa_{0}, \bbE\left[w_{ \alpha}(X_0)\right]\right)\, ,
\end{equation}
and we conclude by taking $n\to\infty$ using Fatou Lemma.
\end{proof}
\subsection{ Well-posedness of the McKean-Vlasov PDE}
The point of this section is to establish well-posedness results for \eqref{eq:PDE_mod_h}. The notion of solution to \eqref{eq:PDE_mod_h} is understood in the following way:
\begin{definition}
\label{def:weak_sol}
Let $ \mu_{ 0}\in \mathcal{ P}_{ 2}$. We say that $ t \mapsto \mu_{ t}$ is a weak solution to \eqref{eq:PDE_mod_h} if $ \mu\in \mathcal{ C}([0, +\infty), \mathcal{ P}_{ 2})$ and for any test function $ \varphi\in \mathcal{ C}^{ 1}([0, +\infty), \mathcal{ C}_{ c}^{ \infty}(\mathbb{ R}^{ d}))$ and any $t\geq 0$,
\begin{multline}
\int_{ \mathbb{ R}^{ d}} \varphi {\rm d}\mu_{ t} = \int_{ \mathbb{ R}^{ d}} \varphi {\rm d}\mu_{ 0}+ \int_{ 0}^{t} \int_{ \mathbb{ R}^{ d}}\left[\partial_{ s} \varphi +  \nabla\cdot(\gs^2\nabla \varphi) \right] {\rm d}\mu_{ s} {\rm d}s\\
- \int_{ 0}^{t} \int_{ \mathbb{ R}^{ d}} \left\lbrace K \left(x- \int_{ \mathbb{ R}^{ d}}z \mu_{ s}({\rm d}z)\right)- \delta \left(F+ h \left(\int_{ \mathbb{ R}^{ d}} z \mu_{ s}({\rm d}z)\right)\right) \right\rbrace\cdot \nabla \varphi {\rm d}\mu_{ s}{\rm d}s\,, t\geq0\,.
\end{multline}
\end{definition}
We first state the following uniqueness result, whose proof is postponed to Appendix~\ref{app: proof uniqueness}.
\begin{proposition}
\label{prop:PDE_well_posed}
For any $ \mu_{ 0}\in \mathcal{ P}_{ 2}$, there exists a unique weak solution $ \left\lbrace \mu_{ t}\right\rbrace_{ t\geq0}$ in $ \mathcal{ C}([0, \infty), \mathcal{ P}_{ 2})$ to \eqref{eq:PDE_mod_h} with initial condition $ \mu_{ 0}$ and this solution is the distribution of the nonlinear process $X_{ t}$ solution to \eqref{eq:modif mc kean}.
\end{proposition}

It is clear that Proposition \ref{prop:PDE_well_posed} is equivalent to
\begin{proposition}
\label{prop:PDE_well_posed p m}
For any $ p_{ 0}\in \mathcal{ P}_{ 2}$ satisfying $\int_{\bbR^d} x p_0(\dd x)=0$ and any $m_0\in \bbR^d$, there exists a unique couple $(p,m)\in \mathcal{ C}([0, \infty), \mathcal{ P}_{ 2})\times \mathcal{ C}^1([0,\infty),\bbR^d)$ solution to \eqref{eq:slow fast PDE modified} with initial condition $(p_0,m_0)$.
\end{proposition}
To emphasize the dependency in the initial condition $(p_{ 0}, m_{ 0})$, we will use in the following the notations $p_t^{p_0,m_0}$ and $m_t^{p_0,m_0}$. Note than an integration by parts shows that $y_{ t}^{ i}:=\partial_{x_i} (p_t^{p_0,m_0})$ is a weak solution of
\begin{equation}\label{eq:PDE partiali p}
\partial_t y^i_t\, =\, \mathcal{ U}_i y^i_t+\nabla\cdot \Big(\big(\dot m_t -\gd F_t -\gd h(m_t)\big) y^i_t\Big)+\gd \nabla \cdot \big( \partial_i F_t\, p_t\big),\ i=1, \ldots, d\,,
\end{equation}
where
\begin{equation}
\label{eq:def_U}
\mathcal{ U}_i u\, := \, \cL u + k_i u\, ,
\end{equation}
and $p_t=p_t^{p_0,m_0}$ and $m_t=m_t^{p_0,m_0}$.
\begin{rem}
\label{rem:uniqueness_frozen_PDE}
Along the same lines as for Proposition~\ref{prop:PDE_well_posed} (see Appendix~\ref{app: proof uniqueness}), one can also prove that there is a unique weak solution in $\mathcal{ C}([0, +\infty), \mathcal{ P}_{ 2})$ to \eqref{eq:slow fast PDE modified} where the nonlinearity $\{m_t\}_{t\geq0}$ has been frozen, equal to $\int_{ \mathbb{ R}^{ d}} z \mu_{ t}(z) {\rm d}z$, where $\mu_t$ denotes the distribution of $X_{ t}$.
The proof of uniqueness is actually simpler since the equation is now linear. This unique solution is obviously $p_{ t}$ the centered version of $\mu_t$.
\end{rem}

\subsection{ Regularity estimates}
\label{sec:regularity_estimates}
For any $ \theta\in \mathbb{ R}$, consider the weighted $L^{ 2}$ and $H^{ 1}$ spaces given by the norms and scalar products \eqref{eq:L2norm_w} for the choice of $w=w_{ \theta}$, where $w_{ \theta}$ is defined in \eqref{eq:w_alpha}. Note that in the case $ \theta=0$, we retrieve the usual $L^{ 2}$-norm $ \left\Vert \cdot \right\Vert_{ L^{ 2}}$ and scalar product $ \left\langle \cdot\, ,\, \cdot\right\rangle_{ L^{ 2}}$.
\medskip

The point of this section is to prove regularity estimates on the process $p$ given by \eqref{eq:slow fast PDE modified}. We first look more closely to the case $\gd=0$.

\subsubsection{The case $ \delta=0$: reversible dynamics and Ornstein-Uhlenbeck processes}\label{subsec:OU}
The case $\delta=0$ in \eqref{eq:slow fast PDE modified} corresponds to
\begin{equation}\label{eq:PDE rev}
\partial_t p^{\text{rev}}_t\, =\, \cL p^\text{rev}_t\, ,
\end{equation}
where the Ornstein-Uhlenbeck operator $ \mathcal{ L}$ is given by \eqref{eq:cLtheta}. Note that $\cL$ satisfies $\cL f =  \nabla\cdot (q_0 \, \gs^2\, \nabla(q_0^{-1} f))$. It will be useful in the following to consider, for $\theta>0$, the operator
\begin{equation}
\cL_\theta u\, =\, \nabla\cdot (\gs^2\nabla u) +\theta \nabla\cdot (Kx u)\, .
\end{equation}
It is well known (see for example \cite{MR3155209}) that its dual operator $\cL^*_\theta$, defined as
\begin{equation}
\cL^*_\theta f\, =\, \nabla\cdot( \gs^2\nabla f) -  \theta K x\cdot \nabla f\, ,
\end{equation}
is self-adjoint in $L^2(w_{-\theta})$ and admits the decomposition
\begin{multline}\label{decom L*}
 \cL^*_\theta \psi_l =-\lambda_l \psi_l\, , \quad \text{with} \qquad \lambda_l\, =\, \theta\sum_{i=1} k_i l_i \quad \text{and} \quad  \psi_l\, =\, \Pi_{i=1}^d H_{l_i}\left(\sqrt{\frac{ \theta k_i}{\gs_i^2}}x_i\right)\, , \\
l\, =\, (l_1,\ldots,l_d)\in \bbN^d\, , \quad \text{and}\quad x\, =\, (x_1,\ldots,x_d)\in \bbR^d\, ,
\end{multline}
where $(H_i)_{i\in \bbN}$ is the basis of the Hermite polynomials, which satisfy
\begin{equation}
e^{sx-\frac{s^2}{2}}\, =\, \sum_{k\in \bbN}\frac{s^k}{\sqrt{n!}}H_k(x),\ x\in\bbR\, .
\end{equation}
Moreover, $(\psi_l)_{l\in \bbN^d}$ is an orthonormal basis of $L^2(w_{-\theta})$.
In particular, if $\int_{\bbR^d}f w_{-\theta}=0$, then (recall that $ \underline{ k}$ denotes the smallest eigenvalue of $K$)
\begin{equation}\label{eq:Poincare q0}
-\langle f, \cL^*_\theta f\rangle_{L^2(w_{-\theta})}\, =\, \int_{\bbR^d} \vert \gs \nabla f\vert^2 w_{-\theta}\, \geq\,  \theta \underline{ k}\langle f,f\rangle_{L^2(w_{-\theta})}\, .
\end{equation}
Define the mapping $A: f\mapsto f w_{-\theta}$. $A$ is an isometry from $L^2(w_{-\theta})$ to $L^2(w_\theta)$ and we clearly have
\begin{equation}
\cL_\theta\, =\, A^{-1} \circ \cL^*_{ \theta} \circ A\, .
\end{equation}
So it is immediate that $\cL_\theta$ is self-adjoint in $L^2(w_\theta)$ and that it admits the decomposition
\begin{equation}\label{decomp cL}
\cL_\theta e_{l}\, =\, -\gl_l e_{l}\quad  \text{for all}\quad 
l= (l_1,\ldots,l_d)\in \bbN^d\, ,
\end{equation}
with $e_l= \psi_l w_{-\theta}$. In particular, for all $u\in L^{ 2}(w_\theta)$ such that $\int_{ \mathbb{ R}^{ d}} u(x) {\rm d} x=0$, we have
\begin{equation}
\label{eq:Lgamma_spectral_gap}
\langle \cL_\theta u, u\rangle_{L^2(w_\theta)} \, \leq - \theta \underline{ k} \left\Vert u \right\Vert_{L^2(w_\theta)}^{ 2}\, ,
\end{equation}
which implies that
\begin{equation}
\left\Vert p^{rev}_t-q_{0,\gs^2K^{-1}}\right\Vert_{L^2(w_1)}\, \leq\, e^{-\underline{ k}t}\left\Vert p^{rev}_0-q_{0,\gs^2K^{-1}}\right\Vert_{L^2(w_1)}\, .
\end{equation}

\subsubsection{ Technical estimates} 
We gather here some technical estimates that will be useful in the following:

\begin{lemma}
\label{lem:estimates}
Fix $\theta\in (0,1)$. For any $u\in H^{ 1}(w_{ \theta})$, then $ x u$ belongs to $L^{ 2}(w_{\theta})$ and we have
\begin{equation}
\label{eq:control_xu_alpha}
\left\Vert \gs^{-1} Kx u\right\Vert_{ L^{ 2}(w_{ \theta})} \, \leq\,  \frac{2}{\theta} \left\Vert \gs \nabla u\right\Vert_{L^2(w_{\theta})}\, ,
\end{equation}
and
\begin{equation}
\label{eq:IPP_xu_alpha}
\int_{\bbR^d}u\nabla u \cdot Kx  w_{\theta}=- \frac{\Tr(K)}{2}\Vert u\Vert^2_{L^2(w_\theta)}-\frac{ \theta}{2} \left\Vert \gs^{-1} Kx u\right\Vert_{L^2(w_\theta)}^2\, .
\end{equation}
\end{lemma}
\begin{proof}[Proof of Lemma~\ref{lem:estimates}]
Note that we have $\nabla w_{\theta}(x)= \theta K\gs^{-2}x w_{\theta}(x)$ and $\nabla\cdot (\gs^2 \nabla w_{\theta})(x)=\theta^2 \vert \gs^{-1} K x\vert^2 w_{\theta}(x)+\theta\Tr(K)  w_{\theta}(x)$, so that
\begin{multline}
\theta^2\int_{\bbR^d} u^2 \vert \gs^{-1} K x\vert^2 w_{\theta}\, \leq\, \int_{\bbR^d}u^2 \nabla\cdot (\gs^2\nabla w_{\theta})\, =\, -2\int_{\bbR^d}u \nabla u\cdot \gs^2 \nabla w_{\theta}\\ \leq\, 2\theta \left( \int_{\bbR^d} u^2 \vert \gs^{-1} Kx\vert^2 w_{\theta}\right)^\frac12\left(\int_{\bbR^d}\vert \gs \nabla u\vert^2 w_{\theta}\right)^\frac12\, , 
\end{multline}
which implies \eqref{eq:control_xu_alpha}. Concerning \eqref{eq:IPP_xu_alpha}, it is a consequence of the identity
\begin{equation}
\int_{\bbR^d}u\nabla u \cdot K x  w_{\theta}=-\int_{\bbR^d}u\nabla u \cdot K x w_{\theta}- \Tr(K)\int_{\bbR^d}u^2 w_{\theta}- \theta \int_{\bbR^d}u^2 \vert \sigma^{-1} K x\vert^2 w_{\theta}\, .
\end{equation}
Lemma~\ref{lem:estimates} is proven.
\end{proof}

\begin{lemma}
\label{lem:bound L OU weighted L2}
Fix $\theta\in (0,1)$. For all $u\in L^{ 2}(w_{ \theta})$ such that $\int_{\bbR^d} u(x) {\rm d} x=0$ we have
\begin{equation}\label{eq poinc w-ga}
\left\Vert \gs \nabla u\right\Vert^2_{ L^{ 2}(w_{ \theta})}\, \geq\, \theta (\Tr(K)+ \underline{ k}) \left\Vert u \right\Vert_{ L^{ 2}(w_{\theta})}^{ 2}\, ,
\end{equation}
and
\begin{equation}
\label{eq:Dir_Lgamma_alpha}
\langle \cL u,u\rangle_{L^2(w_{\theta})}\, \leq\, -\frac{k_\theta}{ \theta (\Tr(K)+ \underline{ k}) }\Vert \gs\nabla u\Vert^2_{L^2(w_{\theta})}\,
\leq\, -k_\theta \Vert  u\Vert^2_{L^2(w_{ \theta})}\, ,
\end{equation}
where 
\begin{equation}
\label{eq:k_alpha_gamma}
k_\theta\, =\,\theta \underline{ k}-\frac{1-\theta}{2} \Tr(K)\, .
\end{equation}
\end{lemma}
Note that this result is consistent with the case $ \theta= 1$, where the optimal rate is $ \underline{ k}$, and that $k_\theta>0$ when $\theta > \frac{\Tr(K)}{\Tr(K)+2 \underline{ k}}$.

\begin{proof}[Proof of Lemma~\ref{lem:bound L OU weighted L2}]
Consider $\theta_1,\theta_2\in (0,1)$ and $u\in L^{ 2}(w_{\theta_1})$ such that $ \int_{ \mathbb{ R}^{ d}}u(x) {\rm d} x=0$. A straightforward calculation gives
\begin{multline}
\label{eq:Dir_Lgamma_alpha_1}
\langle \cL_{\theta_2} u,u\rangle_{L^2(w_{\theta_1})}\, =\, - \left\Vert \gs \nabla u\right\Vert^2_{L^2(w_{\theta_1})}+( \theta_2- \theta_1)\int_{\bbR^d}u \nabla u\cdot K x\,w_{ \theta_1} \\+\theta_2\Tr(K)\Vert u\Vert^2_{L^2(w_{\theta_1})}\, .
\end{multline}
In the particular case of $ \theta_2=\theta_1=\theta$, this boils down to
\begin{equation}
\label{eq:Dir_Lgamma_gamma}
\langle \cL_{\theta} u, u\rangle_{L^2(w_{\theta})}\, =\, -\left\Vert \gs\nabla u\right\Vert^2_{ L^{ 2}(w_{ \theta})} + \theta\Tr(K)\left\Vert u \right\Vert_{ L^{ 2}(w_{ \theta})}^{ 2}\, ,
\end{equation}
and combining this with \eqref{eq:Lgamma_spectral_gap}, we obtain the Poincar\'e inequality \eqref{eq poinc w-ga}.
Now going back to \eqref{eq:Dir_Lgamma_alpha_1}, with $\theta_2=1$ and $\theta_1=\theta$ and using \eqref{eq:IPP_xu_alpha},
\begin{multline}
\langle \cL_{ \theta} u,u\rangle_{L^2(w_{ \theta})}\, =\, -\left\Vert\gs \nabla u\right\Vert^2_{L^2(w_{ \theta})}-\frac{ \theta}{2}(1- \theta)\left\Vert \gs^{-1} Kx u\right\Vert^2_{L^2(w_{ \theta})}\\ +\frac{ 1+ \theta}{2}\Tr(K)\Vert u\Vert^2_{L^2(w_{ \theta})}\, .\label{eq:Dirich_L_alpha}
\end{multline}
It remains to remark from \eqref{eq poinc w-ga} that
\begin{multline}
- \left\Vert \gs\nabla u\right\Vert^2_{L^2(w_{\theta})}+\frac{ 1+ \theta}{2}\Tr(K)\Vert u\Vert^2_{L^2(w_{ \theta})}
\leq -\left(1-\frac{1+\theta}{2\theta}\frac{\Tr(K)}{\Tr(K)+ \underline{ k}}\right)\Vert \gs\nabla u\Vert^2_{L^2(w_{\theta})}\\
=\, -\frac{2\theta \underline{ k}-(1-\theta) \Tr(K)}{2 \theta (\Tr(K)+ \underline{ k})}\Vert \gs\nabla u\Vert^2_{L^2(w_{ \theta})}\, .
\end{multline}
The second inequality in \eqref{eq:Dir_Lgamma_alpha} is again a consequence of \eqref{eq poinc w-ga}.
\end{proof}

\begin{lemma}
\label{lem:control_U}
Fix $\theta\in (0,1)$. For all $v_j\in L^{ 2}(w_{ \theta})$ satisfying the equalities  $\int_{ \mathbb{ R}^{ d}} v_j(x) {\rm d} x =0$, $\int_{ \mathbb{ R}^{ d}} x_j v_j(x) {\rm d} x=0$ and $\int_{ \mathbb{ R}^{ d}} x_i^n v_j(x) {\rm d} x=0$ for $i\neq j$ and $n\geq 1$,
\begin{equation}
\label{eq:poinc_2}
\Vert \gs \nabla v_j\Vert^2_{ L^{ 2}(w_{\theta})}\, \geq\, \theta (\Tr(K)+k_j) \left\Vert v_j \right\Vert_{ L^{ 2}(w_{ \theta})}^{ 2} \, ,
\end{equation}
and
\begin{multline}
\left\langle \mathcal{ U}_{ j} v_j\, ,\, v_j\right\rangle_{ L^{ 2}(w_{ \theta})}\, \leq\, - \frac{ k^{ \prime}_{ \theta}}{ \theta (\Tr(K)+2 \underline{ k})}\Vert \gs \nabla  v_j\Vert^2_{L^2(w_{ \theta})}-\frac{ \underline{ k}\theta(1- \theta)}{2 \bar \sigma}\Vert  x v_j\Vert^2_{L^2(w_{\theta})}\\
\leq - k^{ \prime}_{ \theta}\Vert v_j\Vert^2_{L^2(w_{ \theta})}-\frac{ \underline{ k}\theta(1- \theta)}{2\bar \sigma}\Vert  x v_j\Vert^2_{L^2(w_{\theta})}\label{eq:Dir_Ugamma_alpha}
\end{multline}
where $k^{ \prime}_{ \theta}$ is given by
\begin{equation}
\label{eq:k_prime}
k^{ \prime}_{ \theta} = (2\theta-1) \underline{ k} -\frac{1-\theta}{2}\Tr(K)\, .
\end{equation}
\end{lemma}

Note that again this result is consistent with the case $ \theta= 1$, where the optimal rate (for $j$ such that $k_j=k$) is $ \underline{ k}$, and that $k'_{\theta}>0$ when $\theta > \frac{\Tr(K)+2 \underline{ k}}{\Tr(K)+4 \underline{ k}}$.

\begin{proof}[Proof of Lemma~\ref{lem:control_U}]
The Poincar\'e inequality \eqref{eq:poinc_2} is a consequence of \eqref{eq:Dir_Lgamma_gamma} and the fact that for $v_j$ satisfying the hypotheses of the Lemma, we have (recall \eqref{decomp cL} and the fact that the eigenvectors are constructed with Hermite polynomials)
\begin{equation}
 \langle \cL_\theta v_j, v_j\rangle_{L^2(w_\theta)}\, \leq\, -2\theta k_j \Vert v_j\Vert^2_{L^2(w_\theta)}\, .
\end{equation}
 From \eqref{eq:Dirich_L_alpha}, we obtain
\begin{multline}
\left\langle \mathcal{ U}_{j} v_j\, ,\, v_j\right\rangle_{ L^{ 2}(w_{ \theta})}\, =\, \left\langle \cL v_j\, ,\, v_j\right\rangle_{ L^{ 2}(w_{ \theta})} + k_j\left\Vert v_j \right\Vert_{ L^{ 2}(w_{  \theta})}\\
=-\Vert \gs\nabla v_j\Vert^2_{L^2(w_{ \theta})}-\frac{ \theta}{2}( 1- \theta)\Vert \gs^{-1} K x v_j\Vert^2_{L^2(w_{\theta})} \\+\left(\frac{1+\theta}{2}\Tr(K)+k_j\right) \Vert v_j\Vert^2_{L^2(w_{ \theta})}\, .
\end{multline}
Hence, we deduce \eqref{eq:Dir_Ugamma_alpha} by two successive applications of \eqref{eq:poinc_2} to the previous inequality.
\end{proof}

\subsubsection{Regularity of the solution of \eqref{eq:slow fast PDE modified}}
Recall the definition of $ \varepsilon$ in \eqref{hyp:bound_F_exp}.
\begin{proposition}
\label{lem:regularity_p}
Fix $ \alpha$ and $\beta$ such that $0<\beta<\beta+2\gep<\ga<1$. Suppose that \eqref{eq:slow fast PDE modified} is endowed with an initial condition $ (p_0,m_0)\in \cP_2\times \bbR^d$ such that $p_0$ admits a density w.r.t. Lebesgue measure (that is renamed as $ p_{ 0}$ with a slight abuse of notations) satisfying $ p_0\in L^2(w_{\ga})$. Then, for all $ \delta \geq 0$, for all $t>0$, $ p_{ t}({\rm d}x)$ admits a density $ p_{ t}(x) {\rm d} x$ such that $ p\in L^{ 2}((0, T), H^1(w_{\ga}))$ for all $T>0$. If moreover $p_0\in H^1(w_{\beta})$, then for all $i=1, \ldots, d$, $ \partial_{x_i} p\in L^{ 2}((0, T), H^1(w_{\beta}))$ for all $T>0$.
\end{proposition}

\begin{proof}[Proof of Lemma~\ref{lem:regularity_p}]
Anticipating on the end of the proof, the point of the following is to construct a regular probability solution $ t \mapsto \nu_{ t}$ to the first equation of \eqref{eq:slow fast PDE modified} with the nonlinearity frozen equal to $m_t=m_t^{p_0,m_0}$. By Remark~\ref{rem:uniqueness_frozen_PDE}, it is equal to $ p_t^{p_0,m_0}$. Hence, we will have proven that $ p_t^{p_0,m_0}$ is regular.

We rely on classical ideas, following in particular \cite{Mischler2016}. Consider a regular plateau function $\chi(\cdot)$ on $[0, +\infty)$, namely $ \chi(x)=1$ for $x\in[0, 1]$, $ \chi(x)=0$ for $x\geq 2$ and decreasing on $[1, 2]$. Define finally for $R>0$, $\chi_R(x)=\chi(\vert x\vert^2/R)$ and $F_R(x)=F_t(x) \chi_R(x+m_t)$. For every $t\geq0$, $R\geq1$, introduce the linear operator $Q_{t,R}$
\begin{equation}
-Q_{t,R} u \, =\,  \cL u+  \dot m_t\cdot \nabla u-\delta \nabla\cdot(u (F_R+ h(m_{ t})) \, .
\end{equation}
A straightforward calculation leads to
\begin{multline}\label{eq:decomp Quv}
\langle Q_{t,R} u,v\rangle_{L^2(w_{\ga})}\, =\, \int_{\bbR^d}\nabla u\cdot\gs^2\nabla v\,  w_{\ga}-(1-\ga)\int_{\bbR^d}\nabla u\cdot K x v \,w_{\ga}
\\-\Tr(K)\int_{\bbR^d}uv\, w_{\ga}
- \dot m_{ t} \cdot\int_{ \mathbb{ R}^{ d}} \nabla u v w_{  \alpha} +\delta \int_{\bbR^d}\nabla u\cdot (F_R+h(m_{ t})) v\,w_{\ga}\\+\delta \int_{\bbR^d}uv \nabla\cdot (F_R+h(m_{ t})) w_{\ga}.
\end{multline}
Fix a finite time horizon $T>0$. Using \eqref{eq:control_xu_alpha}, the fact that $ h$ is bounded (see Section \ref{sec:modification dynamics}) and $\sup_{ t\in[0, T]}\vert \dot m_t\vert< +\infty$ (recall \eqref{eq:dot m_inf_delta}), it is easy to see that, for $t\in [0,T]$,
\begin{equation}\label{eq:bound Q L2 H1}
\left| \langle Q_{t,R} u,v\rangle_{L^2(w_{\ga})} \right| \, \leq\,  C_{T,R} \Vert u\Vert_{H^1(w_{\ga})}\Vert v\Vert_{H^1(w_{\ga})}\, .
\end{equation}
Let us now show that $Q_{t,R}$ is coercive, by finding a lower bound for $\langle Q_{t,R}u,u\rangle_{L^2(w_{\ga})}$. The constants appearing in this lower bound will not depend on $R$, which will be crucial to take the limit $R\to \infty$. Note that the following integration by parts formula is true:
\begin{equation}
\int_{\bbR^d} \nabla u\cdot F_R uw_{\ga}=-\int_{\bbR^d}u F_R\cdot \nabla u w_{\ga}-\int_{\bbR^d} u^2 \nabla\cdot F_R w_{\ga}-\ga \int_{\bbR^d}u^2 F_R\cdot K\gs^{-2} x w_{\ga},
\end{equation}
so that
\begin{equation}\label{eq:simplify F term}
\int_{\bbR^d} \nabla u\cdot F_R uw_{\ga}=-\frac12\int_{\bbR^d} u^2 (\nabla\cdot F_R w_{\ga}+\ga F_R\cdot K\gs^{-2}x) w_{\ga}.
\end{equation}
Using also \eqref{eq:IPP_xu_alpha}, we obtain
\begin{multline}
\langle Q_{t,R} u,u\rangle_{L^2(w_{\ga})} \, =\, \Vert \gs \nabla u\Vert^2_{L^2(w_{\ga})}+\frac{\ga}{2}(1-\ga)\Vert \gs^{-1} Kx u\Vert^2_{L^2(w_{\ga})}-\frac{1+\ga}{2}\Tr(K)\Vert u\Vert^2_{L^2(w_{\ga})}\\
- \frac{ \alpha}{ 2}\left( \delta h(m_{ t})-\dot m_t\right)\cdot \int_{\bbR^d} K\gs^{-2}x u^2 w_{\ga} -\frac{\delta}{2}\int_{\bbR^d}u^2 (\ga F_R\cdot K\gs^{-2} x-\nabla\cdot F_R) w_{\ga}\, .\label{eq:QtR_u2}
\end{multline}
Note that, by \eqref{eq:m_inf_delta} and the assumptions on $h$, we have
\begin{equation}
\frac{ \alpha}{ 2}\left\vert \delta h(m_{ t})-\dot m_t \right\vert \,\leq\, C_{ 1}\, ,
\end{equation}
for some positive constant $C_{ 1}=C_{ 1}(\delta, T)$ independent of $R$.
Moreover we have
\begin{multline}
\ga F_R(x)\cdot K\gs^{-2} x-\nabla\cdot F_R(x)\, =\, \chi_R(x+m_t)(\ga F(x+m_t)\cdot K\gs^{-2}(x+m_t)\\-\nabla\cdot F(x+m_t)
-F(x+m_t)\cdot K\gs^{-2}m_t)
-\frac{2}{R}\chi' \left( \frac{ \vert x\vert^2}{ R}\right)x \cdot F(x)\, , 
\end{multline}
so that, using \eqref{hyp:F_dot_x_bound}, \eqref{hyp:F_dot_x_larger}, \eqref{eq:m_inf_delta} and the fact that $ \chi^{ \prime}\leq 0$, there exists a constant $C_{2}>0$ (independent of $R$) such that for all $R\geq1$,
\begin{align*}
\ga F_R(x)\cdot K\gs^{-2} x-\nabla\cdot F_R(x)&\, \leq\, C_{ 2} \left(\chi_R(x) -\frac{2}{R}\chi' \left( \frac{ \vert x\vert^2}{ R}\right)\right)\leq C_{2} \left(1 + 2 \left\Vert \chi^{ \prime} \right\Vert_{ \infty}\right).
\end{align*}
Consequently,
\begin{multline}\label{eq:lower_bound_Qtuu}
\langle Q_{t,R} u,u\rangle_{L^2(w_{\ga})} \, \geq\, \Vert \gs\nabla u\Vert^2_{L^2(w_{\ga})}+\int_{\bbR^d}u^2 \left( \frac{\ga}{2}(1-\ga)\vert \gs^{-1} K x\vert^2 - C_{ 1} \left\vert x \right\vert\right) w_{\ga}\\
 - \frac{ 1}{ 2} \left( (1+\ga)\Tr(K) + \delta C_{2} \left(1 + 2 \left\Vert \chi^{ \prime} \right\Vert_{ \infty}\right)\right)\int_{\bbR^d}u^2 w_{\ga}\, .
\end{multline}
Since $ \alpha< 1$ by assumption, we deduce that there exist $C_{3}>0$ and $C_{4}>0$ that depend on $\ga$, $ \delta$ but not on $R$ such that
\begin{equation}\label{eq:Q coerc}
\langle Q_{t,R} u,u\rangle_{L^2(w_{\ga})}\,\geq \, C_{3} \Vert \nabla u\Vert_{L^2(w_{\ga})}^{ 2}- C_{4}\Vert u\Vert^2_{L^2(w_{\ga})}\, .
\end{equation}
Then the Lions Theorem (see \cite{MR697382}, Theorem X.9) implies that there exists a unique $ \nu^{ R}$ such that, for all $T>0$,
\[
\nu^R\in L^2((0,T),H^1(w_{\alpha}))\cap \mathcal{ C}([0,T],L^2(w_{ \alpha}))\, \quad\text{with} \quad \frac{\dd}{\dd t} 
\nu^R\in L^2((0,T), H^{-1}(w_{ \alpha})) \, ,
\]
such that $ \nu^R_0= p_0$ and for any $v\in H^1(w_{\ga})$ and almost every $t\in [0,T]$,
\begin{equation}
\label{eq:nuR_weak_form_L2}
\left\langle \frac{\dd}{\dd t} \nu^R_t, v\right\rangle_{L^2(w_{\ga})}+\left\langle Q_{t,R}\, \nu^R_t,v\right\rangle_{L^2(w_{\ga})}\, =\, 0\, .
\end{equation}
The fact that the positivity of the solution $ \nu^{ R}$ is conserved for $t>0$ can be obtained by regularization arguments (note in particular that since $ \nu_{ t}^{ R} \in H^{ 1}(w_{  \alpha})$, so does $\left(\nu^{ R}_{ t}\right)_{ -}:= \min \left( \nu^{ R}_{ t}, 0\right)$ and we have $ \nabla  \left(\nu_{ t}^{ R}\right)_{ -}= \nabla \nu_{ t}^{ R} \mathbf{ 1}_{ \nu_{ t}^{ R}<0}$, see for example \cite{MR2254750}, X,~Theorem~8).

\medskip

Our next aim is to get bounds on $\nu^R$ and $ \frac{\dd}{\dd t} \nu^R$ independent from $R$.
Taking $v= \nu^R_t$ we get for almost every $t\in[0,T]$,
\begin{equation}\label{eq:ddt norm 2 nuR}
\frac12 \frac{\dd}{\dd t} \Vert \nu^R_t\Vert^2_{L^2(w_{\ga})}+\left\langle Q_{t,R}\, \nu^R_t, \nu^R_t\right\rangle_{L^2(w_{\ga})} \, =\, 0\, , 
\end{equation}
which implies, using \eqref{eq:Q coerc} and Gr\"onwall's inequality that $\Vert \nu^R_t\Vert^2_{L^2(w_{\ga})} \leq  e^{Ct}\Vert p_0\Vert^2_{L^2(w_{\ga})}$ and
\begin{equation}\label{eq:bound nuR ind R} \int_0^T \Vert \nu^R_t\Vert_{H^1(w_{\ga})}^2\dd t\, \leq C(T)\, \Vert p_0\Vert^2_{L^2(w_{\ga})}\, ,
\end{equation}
where these constants do not depend on $R$.
Consider now a $v\in H^1(w_{\ga'})$ with $\ga'+2 \gep<\ga$ and $\Vert v\Vert_{H^1(w_{\ga'})}\leq 1$. Our first aim is to prove that
\begin{equation}\label{eq:bound Q H1}
\left|\langle Q_{t,R} \nu^R_t, v\rangle_{L^2(w_{\ga'})}\right|\leq C \Vert \nu^R_t\Vert_{H^1(w_{\ga})}\Vert v\Vert_{H^1(w_{\ga'})}\, ,
\end{equation}
where the constant $C$ does not depend on $R$. We have, by Lemma \ref{lem:estimates},
\begin{equation}
\left|\int_{\bbR^d} \nabla \nu^R_t \cdot Kx  v \, w_{\ga'} \right|\, \leq \, \frac{2}{\ga'}\Vert \gs \nu^R_t\Vert_{H^1(w_{\ga'})}\Vert \gs v\Vert_{H^1(w_{\ga'})}\, ,
\end{equation}
and by \eqref{hyp:bound_F_exp},
\begin{equation}
\left|\int_{\bbR^d} \nabla \nu^R_t \cdot F_R v \, w_{\ga'} \right|\, \leq \, \Vert \nabla \nu^R_t F_R\Vert_{L^2(w_{\ga'})}\Vert v\Vert_{L^2(w_{\ga'})}\, 
\leq\, C_F \Vert \nu^R_t\Vert_{H^1(w_{\ga})} \Vert v\Vert_{L^2(w_{\ga'})}\, ,
\end{equation}
and
\begin{equation}
\left|\int_{\bbR^d} \nu^R_t v \nabla \cdot F_R  \, w_{\ga'} \right|\, \leq \, \Vert  \nu^R_t \nabla \cdot  F_R\Vert_{L^2(w_{\ga'})}\Vert v\Vert_{L^2(w_{\ga'})}\, 
\leq\, C_F \Vert \nu^R_t\Vert_{H^1(w_{\ga})} \Vert v\Vert_{H^2(w_{\ga'})}\, .
\end{equation}
So, recalling the decomposition \eqref{eq:decomp Quv}, \eqref{eq:bound Q H1} is indeed satisfied.
We deduce that
\begin{equation}\label{eq:bound ddt nu H-1}
\left\Vert  \frac{\dd}{\dd t} \nu^R_t\right\Vert_{L^2((0,T),H^{-1}(w_{\ga'}))}
\, \leq\, C\Vert \nu^R_t\Vert_{L^2((0,T),H^1(w_{\ga}))}\, \leq\, C'(T) \Vert p_0\Vert_{L^2(w_{\ga})}\, .
\end{equation} 

We are now ready to apply the Banach-Alaoglu Theorem: from \eqref{eq:bound nuR ind R} and \eqref{eq:bound ddt nu H-1} we deduce that there exists a sequence $(R_p)_{p\geq 0}$ going to infinity and a $\nu \in  L^2((0,T),H^1(w_{\ga}))$ with $\frac{\dd}{\dd t} \nu \in L^2((0,T),H^{-1}(w_{\ga'}))$ such that $\nu^{R_{ p}}$ and $\frac{\dd}{\dd t}  \nu^{ R_{ p}}$ converge weakly respectively to $\nu$ and $\frac{\dd }{\dd t} \nu$. 
In particular $\nu \in \mathcal{ C}([0,T],L^2(w_{\ga'}))$ (see for example  Theorem 7.2 of \cite{MR1881888}), and $\nu \in \mathcal{ C}([0,T],\cP_{ 2})$ by continuous inclusion. Moreover for any $\phi$ smooth with compact support $ \nu$ satisfies thus for all $t\leq T$
\begin{equation}
\int_{\bbR^d} \nu_t \phi_t\, =\, \int_{\bbR^d}p_0 \phi_0 +\int_0^t \nu_s \Big(\partial_t \phi_s+\nabla\cdot(\gs^2 \nabla\phi_s)-\big(Kx+\dot m_t-\gd(F_t+h(m_t))\big)\cdot \nabla \phi_s\Big)\dd s\, , 
\end{equation}
and thus it is also satisfied for any $\phi$, $ \mathcal{ C}^2$ with bounded derivatives, by a density argument. So $\nu$ is a weak solution of \eqref{eq:slow fast PDE modified}, and $\nu=p$ by uniqueness.

\medskip

Suppose now that $p_0 \in H^1(w_{\beta})$. Our first aim is to prove that we have  $\frac{\dd}{\dd t} \nu^R\in L^2((0,T), L^2(w_{\beta}))$, which means that $\frac{\dd}{\dd t} \partial_i \nu^R\in  L^2((0,T), H^{-1}(w_{\beta}))$ for all $i$. For the orthonormal basis $(e_k)$ of $L^2(w_{\beta})$ defined in \eqref{decomp cL}, we denote for $n\geq 1$
\begin{equation}
\nu^{R,n}_t := \sum_{|j|\leq n} \langle \nu^R_t,e_j\rangle_{L^2(w_{\beta})}e_j\, .
\end{equation}
For all $|j|\leq n$ we have
\begin{multline}
\left\langle \frac{\dd}{\dd t} \nu^{R,n}_t,e_j\right\rangle_{L^2(w_{\beta})}\, 
=\,\left\langle \frac{\dd}{\dd t} \nu^{R}_t, e_j\right\rangle_{L^2(w_{\beta})}
\, =\,  -\left\langle Q_{t,R}\nu^{R}_t, e_j\right\rangle_{L^2(w_{\beta})}\\
 =\,  \left\langle \mathcal{ L}_{\beta}\nu^{R}_t,e_j\right\rangle_{L^2(w_{\beta})} -\left\langle A_{t,R}\nu^{R}_t, e_j\right\rangle_{L^2(w_{\beta})}\\
 =\,  \left\langle \mathcal{ L}_{\beta}\nu^{{R,n}}_t, e_j\right\rangle_{L^2(w_{\beta})} -\left\langle A_{t,R}\nu^{R}_t, e_j\right\rangle_{L^2(w_{\beta})} \, ,
\end{multline}
where $A_{t,R}=Q_{t,R}+  \mathcal{ L}_\beta$ and we have used the fact that the $e_j$'s are eigenvalues of $ \mathcal{ L}_\beta$. So, since $ \frac{\dd}{\dd t} \nu^{R,n}_t\in {\rm Span}(e_j)_{|j|\leq n}$, we have
\begin{equation}
\left\langle \frac{\dd}{\dd t} \nu^{R,n}_t,\frac{\dd }{\dd t} \nu^{R,n}_t\right\rangle_{L^2(w_{\beta})}
\, =\,  \left\langle \mathcal{ L}_\beta\nu^{R,n}_t, \frac{\dd }{\dd t}\nu^{R,n}_t\right\rangle_{L^2(w_{\beta})}-\left\langle A_{t,R}\nu^{R}_t, \frac{\dd }{\dd t}\nu^{R,n}_t\right\rangle_{L^2(w_{\beta})}\, .
\end{equation}
But on one hand, we have
\begin{equation}
\left\langle \mathcal{ L}_\beta\nu^{R,n}_t, \frac{\dd }{\dd t}\nu^{R,n}_t\right\rangle_{L^2(w_{\beta})}
\, =\, -\int_{\bbR^d}\nabla \nu^{R,n}_t \cdot \gs^2\nabla \left( \frac{\dd}{\dd t}\nu^{R,n}_t\right) w_{\beta}\, =\, -\frac12\frac{\dd}{\dd t}\Vert \gs \nabla\nu^{R,n}_t\Vert_{L^2(w_{\beta})}^2\, ,
\end{equation}
and on the other hand
\begin{multline}\label{estimate AtR}
\left| \left\langle A_{t,R}\nu^{R}_t, \frac{\dd }{\dd t}\nu^{R,n}_t\right\rangle_{L^2(w_{\beta})}\right|\, \leq\, C_T\Big(\Vert  \nu^R_t\Vert_{L^2(w_{\beta})}+\Vert \nabla \nu^R_t\Vert_{L^2(w_{\beta})}+\Vert x\nabla \nu^R_t\Vert_{L^2(w_{\beta})}\\
+\Vert F_R\cdot \nabla \nu^R_t\Vert_{L^2(w_{\beta})}+\Vert \nabla\cdot F_R  \nu^R_t\Vert_{L^2(w_{\beta})} \Big)\left\Vert \frac{\dd}{\dd t}\nu^{R,n}_t\right\Vert_{L^2(w_{\beta})}\\
\leq\, C_{F,T} \Vert \nu^R_t\Vert_{H^1(w_{\ga})}\left\Vert \frac{\dd}{\dd t}\nu^{R,n}_t\right\Vert_{L^2(w_{\beta})}\, ,
\end{multline}
where we have used in particular \eqref{hyp:bound_F_exp} and the fact that $\beta+2\gep<\ga$. This means that 
\begin{equation}
\left\Vert  \frac{\dd}{\dd t} \nu^{R,n}_t\right\Vert^2_{L^2(w_{\beta})}
\, \leq\, -\frac{1}{2} \frac{\dd}{\dd t}\Vert \gs \nabla\nu^{R,n}_t\Vert_{L^2(w_{\beta})}^2+ C'_{F,T}\Vert \nu^R_t\Vert_{H^1(w_{\ga})}^2+\frac12 \left\Vert \frac{\dd}{\dd t}\nu^{R,n}_t\right\Vert_{L^2(w_{\beta})}^2\, ,
\end{equation}
and after a time integration we obtain
\begin{equation}
\int_0^T\left\Vert  \frac{\dd}{\dd t} \nu^{R,n}_t\right\Vert^2_{L^2(w_{\beta})}\dd t
\, \leq\, \Vert p_0\Vert_{H^1(w_{\beta})}^2 +2 C'_{F,T}\int_0^T \Vert \nu^R_t\Vert_{H^1(w_{\ga})}^2\dd t\, ,
\end{equation}
and, recalling \eqref{eq:bound nuR ind R} and taking $n$ going to infinity, that $\frac{\dd}{\dd t} \nu^R\in L^2((0,T), L^2(w_{\beta}))$, with
\begin{equation}\label{eq:bound ddt nuR L2}
\left\Vert \frac{\dd}{\dd t}\nu^R\right\Vert_{L^2((0,T),L^2(w_{\beta}))}\leq C'_{F,T}\left(\Vert p_0\Vert_{H^1(w_{\beta})}+\Vert p_0\Vert_{L^2(w_{\ga})}\right)\, ,
\end{equation}
where $C'_{F,T}$ does not depend on $R$.

\medskip

Let us now prove that $\nu^R\in L^2((0,T), H_0^2(w_{\beta}))$. Here we define $H_0^2(w_{\beta})$ as
\begin{equation}
H_0^2(w_{\beta})\, :=\, \left\{ u\in L^2(w_{\beta}): \, \int_{\bbR^d}u =0\text{ and } \langle \mathcal{ L}_\beta u, \mathcal{ L}_\beta u\rangle_{L^2(w_{\beta})} <\infty \right\}\, .
\end{equation}
Remark in particular that there exists $C_\beta>0$ such that if $u\in H^2_0(w_{\beta})$, then
\begin{equation}
\Vert \partial_i u\Vert_{H^1(w_{\beta})}\, \leq\, C_\beta\Vert u\Vert_{H^2_0(w_{\beta})}\, .
\end{equation}
Indeed, if $u$ admits the decomposition $u=\sum_{j} u_j e_j$, then $\partial_i u=\sum_j u_j \partial_i e_k$, and we have
\begin{equation}
\partial_i e_j \, =\, \partial_i (\psi_j w_{-\beta})\, =\, -\beta \frac{k_j}{\gs_j^2}x_j e_j+\sqrt{j_i}\sqrt{\beta \frac{k_j}{\gs_j^2}}e_{\breve j_i}\, =\, -\sqrt{j_i+1}\sqrt{\beta \frac{k_j}{\gs_j^2}} e_{\hat j_i}\, , 
\end{equation}
where we used the notation $\breve j_i=(j_1,\ldots,j_{i-1}, j_i-1,j_{i+1},\ldots, j_d)$, $\hat j_i = (j_1,\ldots,j_{i-1}, j_i+1,j_{i+1},\ldots, j_d)$ and relied on \eqref{decom L*} and the fact that for the Hermite polynomials $H_n$ we have $H'_n(x)=\sqrt{n}H_n(x)$ and $xH_n(x)=\sqrt{n+1}H_{n+1}(x)+\sqrt{n}H_{n-1}(x)$. Then we have, since $\int_{\bbR^d} \partial_i u=0$, 
\begin{equation}
\Vert \partial_i u\Vert_{H^1(w_{\beta})}^2\, \leq\, C_\beta \sum_{j} \lambda_j\Vert e_j\Vert_{L^2(w_\beta)} u_k^2\, \leq\,C_{\beta}' \sum_j \lambda_j^2  u_k^2 \, =\, C_\beta'\Vert u\Vert_{H^2_0(w_{\beta})}^2\, .
\end{equation}
Now for all $v_t$ smooth, \eqref{eq:nuR_weak_form_L2} can be rewritten as
\begin{equation}
\langle \mathcal{ L}\nu^R_t, v_t\rangle_{L^2(w_{\beta})}\, = \left\langle \frac{\dd}{\dd t} \nu^R_t +A_{t,R} \nu^R_t, v_t\right\rangle_{L^2(w_{\beta})}\, ,
\end{equation}
so that
\begin{equation}
\left|\int_0^T \langle \mathcal{ L} \nu^R_t, v_t\rangle_{L^2(w_{\beta})}\right|\, \leq\, \left\Vert \frac{\dd}{\dd t} \nu^R_t +A_{t,R} \nu^R_t \right\Vert_{L^2((0,T),L^2(w_{\beta}))}
\left\Vert v \right\Vert_{L^2((0,T),L^2(w_{\beta}))}\, .
\end{equation}
But similar arguments as the ones used for example to obtain \eqref{estimate AtR}, we get
\begin{equation}
\left\Vert A_{t,R} \nu^R_t\right\Vert_{L^2((0,T), L^2(w_\beta))}\, \leq\, C_{F,T}\left\Vert \nu^R_t\right\Vert_{L^2((0,T),H^1(w_\ga))}\, \leq\, C'_{F,T}\Vert p_0\Vert_{L^2(w_\ga)}\, ,
\end{equation}
where for the last inequality we relied on \eqref{eq:bound nuR ind R}. So recalling also \eqref{eq:bound ddt nuR L2}, we get
\begin{equation}
\left|\int_0^T \langle \mathcal{ L} \nu^R_t, v_t\rangle_{L^2(w_{\beta})}\right|\, \leq\, C''_{F,T} \left(\Vert p_0\Vert_{H^1(w_\beta)}+\Vert p_0\Vert_{L^2(w_\ga)}\right)\left\Vert v \right\Vert_{L^2((0,T),L^2(w_{\beta}))}\, ,
\end{equation}
and we deduce that $\nu^R\in L^2((0,T), H^2_0(w_{\beta}))$, with
\begin{equation}
\Vert \nu^R\Vert_{L^2((0,T), H^2_0(w_{\beta}))}\leq C''_{F,T}\left(\Vert p_0\Vert_{H^1(w_\beta)}+\Vert p_0\Vert_{L^2(w_\ga)}\right)\, .
\end{equation}
Here $C''_{F,T}$ does not depend on $R$, and taking $R$ going to infinity along sub-sequences shows that this last bound is also valid for $p$. Remark that using integration by parts, we obtain for all $v$ smooth with compact support,
\begin{multline}
\left\langle \frac{\dd}{\dd t} \partial_i \nu^R_t, v\right\rangle_{L^2((0,T),L^2(w_{\beta}))}
+ \left\langle B_{t,R,i}\, \partial_i \nu^R_t, v\right\rangle_{L^2((0,T),L^2(w_{\beta}))}\\ =\, 
-\gd \left\langle \nabla \cdot \big( \partial_i F_t\, p_t\big), v\right\rangle_{L^2((0,T),L^2(w_{\beta}))}\, ,
\end{multline}
where
\begin{equation}
B_{t,R,i} u\,  =\, - \mathcal{ U}_i u-\nabla\cdot \Big(\big(\dot m_t -\gd F_t -\gd h(m_t)\big) u\Big)\, . 
\end{equation}
Since $\langle B_{R,t,i} u,v\rangle_{L^2((0,T),L^2(w_{\beta}))}\leq C_{R,T} \Vert u\Vert_{L^2((0,T),H^1(w_{\beta}))}\Vert v \Vert_{L^2((0,T),H^1(w_{\beta}))}$ and $\frac{\dd}{\dd t} \partial_i \nu^R$ and $ \nabla \cdot ( \partial_i F_t\, p_t)$ are elements of $L^2((0,T), H^{-1}(w_{\beta}))$, by a density argument we obtain for almost all $t\in [0,T]$, 
\begin{equation}\label{eq:ddt norm L2 partiali nuR}
\frac{\dd}{\dd t}\Vert \partial_i \nu^R_t\Vert_{L^2(w_{\beta})}^2+ \langle B_{t,R}\, \partial_ i\nu^R_t,\partial_i\nu^R_t\rangle_{L^2(w_{\beta})}\, =\, \langle  \nabla \cdot ( \partial_i F_t\, p_t),\partial_i \nu^R_t\rangle_{L^2(w_{\beta})}\, .
\end{equation}
\end{proof}

\begin{rem}\label{rem:non rig}
We make in the following an abuse of notations, i.e. write terms of the type $\frac{\dd}{\dd t}\Vert p_t\Vert_{L^2(w_{-\ga})}^2$ even if we have not define them properly as in Lemma \ref{lem:regularity_p}. These terms are in fact well defined for the process $\nu^R$ for all $R$ (see for example \eqref{eq:ddt norm 2 nuR} and \eqref{eq:ddt norm L2 partiali nuR}), and the bounds we will obtain will be independent of $R$, so will be satisfied by $p$ at the limit $R\rightarrow\infty$. So for simplicity we will drop the dependency in $R$ in our notations and write all these expressions with respect to $p$.
\end{rem}

\section{Persistence of the invariant manifold}
\label{sec:persistence}
\subsection{Uniform estimates}

Recall the definition of the exponent $ \varepsilon$ appearing in the exponential bound on $F$ in \eqref{hyp:bound_F_exp} and the hypothesis \eqref{hyp:epsilon_small}. We now define
\begin{equation}
\label{hyp:alpha_beta_gamma}
\alpha:= 1- \varepsilon,\ \beta:=1 - 4 \varepsilon,\ \gamma:= 1- 7 \varepsilon, \text{ and } \gamma^{ \prime}:= 1- 10 \varepsilon\,.
\end{equation}
The assumption \eqref{hyp:epsilon_small} ensures that the following conditions are satisfied:
\begin{multline}
\label{eq:def_alpha_beta}
\frac{ \Tr(K)+2 \underline{ k}}{ \Tr(K)+4 \underline{ k}}\, <\, \gamma'\, <\, \gamma'+2\gep\, <\, \gamma\, <\, \gamma+2\gep \, <\,   \beta \, < \, \beta+ 2\varepsilon \, <\,  \alpha\, <\,  1\, . 
\end{multline}
\begin{rem}
\label{rem:meaning_constants}
We will use the constants $\ga$ and $\gb$ in the present section: $L^{ 2}(w_{ \alpha})$ for the control of the proximity of the solution $p$ to the first equation of \eqref{eq:slow fast PDE modified} to the Gaussian measure $ q_{0,\gs^2K^{-1}}$ (see Lemma~\ref{lem:mu bounded L2} below) and $L^{ 2}(w_{ \beta})$ for the control of the spatial derivative of $p$ (see Lemma~\ref{lem:nablamu_bounded_L2}). These two spaces are the core of the construction of the fixed-point procedure exposed in Section~\ref{sec:fixed_point_f}. The constants $\gamma$ and $\gamma'$ are the counterparts of $ \alpha$ and $ \beta$ in order to establish the $ \mathcal{ C}^{ 1}$-regularity of the invariant manifold constructed in Section~\ref{sec:persistence} (see Section~\ref{sec:regularity_manifold}).

The lower bound $ \frac{ \Tr(K)+2 \underline{ k}}{ \Tr(K)+4 \underline{ k}}$ in \eqref{eq:def_alpha_beta} is here to ensure that both $k_{ \iota}$ and $k^{ \prime}_{ \iota}$ (defined respectively in \eqref{eq:k_alpha_gamma} and \eqref{eq:k_prime}) are strictly positive for any $ \iota\in \left\lbrace \alpha, \beta, \gamma, \gamma^{ \prime}\right\rbrace$. Note also that, under the hypothesis \eqref{hyp:epsilon_small} on $ \varepsilon$ together with \eqref{hyp:alpha_beta_gamma}, we have $ 2 \varepsilon < \iota$ for any $ \iota\in \left\lbrace \alpha, \beta, \gamma, \gamma^{ \prime}\right\rbrace$ (which entails, by \eqref{hyp:bound_F_exp}, that $F\in L^{ 2}(w_{-\iota})$) as well as $ \alpha + 2 \varepsilon< 2$ (which entails that $q_{ 0}F$ and $ \nabla(q_{ 0}F)$ belong to $L^{ 2}(w_{ \alpha})$). These properties will be used continuously in the following.
\end{rem}
The point of the following results is to prove proximity estimates on $p$, solution to the centered PDE \eqref{eq:slow fast PDE modified}.
\begin{lemma}
\label{lem:apriori_bound_mt}
Let $ \alpha$ given by \eqref{hyp:alpha_beta_gamma}. For all $T>0$, there exist $ \delta_{1}= \delta_{1}(T)>0$ and $ \kappa_{1}=\kappa_1(T)>0$ such that if we define for all $m\in\bbR^d$:
\begin{equation}\label{eq:def cE}
\cE(m)\, =\, \cE(m,\kappa_1)\, :=\, \left\{ p\in \cP_2 :\, \int_{\bbR^d} xp({\rm d}x)=0\quad \text{and}\quad  \int_{\bbR^d} w_{\ga}(x-m)p({\rm d}x)\leq \kappa_1\right\}\, ,
\end{equation}
then for all $ 0\leq \delta \leq \delta_{1}$ the following holds:
\begin{equation}
\sup_{|m_0|\leq 3L}
\quad
\sup_{p_0\in \cE(m_0)}
\quad
 \sup_{ t\in[0, 2T]}\quad
 \left\vert m_t^{p_0,m_0} \right\vert \, \leq 4L\, ,\label{eq:apriori_bound_mt}
\end{equation}
and
\begin{equation}
\sup_{|m_0|\leq 3L}
\quad
\sup_{p_0\in \cE(m_0)}
\quad
 \sup_{ t\in[0, 2T]}\quad
\int_{ \mathbb{ R}^{ d}} w_{ \alpha}(x-m^{p_0,m_0}_t) p_t^{p_0,m_0}({\rm d}x)\, \leq\,  \kappa_{1}\, .\label{eq:apriori_bound_exp}
\end{equation}
\end{lemma}
\begin{proof}[Proof of Lemma~\ref{lem:apriori_bound_mt}]
We restrict ourselves to $ \delta\in[0, \delta_{ 0}]$ where $ \delta_{ 0}>0$ is a fixed arbitrary constant. Let $ \mu_{ 0}(\cdot)=p_0(\cdot-m_0)$, $X_{ 0}$ with distribution $\mu_{ 0}$ and consider $\{X_{ t}\}_{ t\in[0, 2T]}$ solution to \eqref{eq:modif mc kean}. Let $t_1$ be defined as
\begin{equation}
t_1\, =\, \inf\{t\in(0, 2T]:\,  \vert m_t\vert > 4L\}\, .
\end{equation}
By continuity we have $t_1>0$. Moreover for any $t\leq t_1$ we obtain, by \eqref{eq:exp_control_Xt}, that $\bbE[w_{\alpha}(X_t)]\leq \max(\bbE[w_{ \alpha}(X_0)], \kappa_{ 0}(4L, \delta_{ 0}))$, where $ \kappa_{0}$ is given in \eqref{eq:kappa1}. Choosing $ \kappa_{1}:= \kappa_{0}(4L, \delta_{ 0})$, by the exponential bound \eqref{hyp:bound_F_exp} on $F$ and since $ \alpha \geq \gep$ (recall \eqref{eq:def_alpha_beta}),
\begin{equation}
\label{eq:apriori_bound_mut_F}
\left\vert \left\langle \mu_{ t}\, ,\, F\right\rangle \right\vert \,=\, \mathbb{ E}\left[F(X_{ t})\right]\,\leq\, C_F\mathbb{ E} \left[w_{  \alpha}(X_{ t})\right] \, \leq\,C_F \kappa_{1}\, ,
\end{equation}
for all $t\leq t_1$.
Hence, since $m_t$ is solution of \eqref{eq:slow fast PDE modified},
we obtain that $\vert  \dot m_t\vert  \leq \delta \left( \left\Vert h \right\Vert_{ \infty} + C_F\kappa_{1}\right) $ for all $t\leq t_1$, so that $ \left\vert m_{ t} \right\vert \leq \left\vert m_{ 0} \right\vert +2T \delta \left( \left\Vert h \right\Vert_{ \infty} + C_F\kappa_{1}\right)$. For the choice of $ \delta_{1}:= \min \left(\frac{ L}{ 2T( \left\Vert h \right\Vert_{ \infty} +C_F \kappa_{1})}, \delta_{ 0}\right)$, $ \left\vert m_{ t} \right\vert \leq 4L$ for all $t\leq t_{ 1}$, so that $t_{ 1}=2T$ and Lemma~\ref{lem:apriori_bound_mt} is proven.
\end{proof}
\begin{lemma}
\label{lem:mu bounded L2}
Let $ \alpha$ given by \eqref{hyp:alpha_beta_gamma}. For all $T>0$, there exist $\delta_2= \delta_{ 2}(T)>0$ and $ \kappa_{2}= \kappa_{ 2}(T)>0$ such that if
we define for all $m\in \bbR^d$ (recall \eqref{eq:def cE})
\begin{equation}\label{eq:def cF}
\cF(m)\, =\, \cF(m,\gd,\kappa_1,\kappa_2)\, :=\, \cE(m,\kappa_1)\cap \bigg\{p\in \cP_2 :\, \Vert p-q_0\Vert_{L^2(w_{\ga})}\leq\, \kappa_2\gd\bigg\}\, ,
\end{equation}
then for all $\delta\in[0, \delta_{ 2}]$, the following holds:
\begin{equation}
\sup_{|m_0|\leq 3L}
\quad
\sup_{p_0\in \cF(m_0)}
\quad
 \sup_{ t\in[0, 2T]}\quad
 \left\Vert p_t^{p_0,m_0}-q_0\right\Vert_{L^2(w_{\ga})}\, \leq\,  \kappa_{2} \delta.
\end{equation}
\end{lemma}
\begin{proof}[Proof of Lemma~\ref{lem:mu bounded L2}]
We place ourselves in the framework of Lemma~\ref{lem:apriori_bound_mt} and consider $ \delta\leq \delta_{1}$. For simplicity, we denote $p_t=p_t^{p_0,m_0}$ and $m_t=m_t^{p_0,m_0}$. The process $p_t-q_0$ satisfies
\begin{equation}
\partial_t (p_t-q_0)\, =\,\mathcal{ L}(p_t-q_0)+\nabla\cdot (p_t \dot m_t)-\delta \nabla\cdot \big(p_t F_t\big)-\delta \nabla\cdot\left(p_t h(m_t)\right)\, ,
\end{equation}
so that (recall Remark \ref{rem:non rig}),
\begin{multline}
\frac12\frac{\dd}{\dd t}\Vert p_t-q_0\Vert^2_{L^2(w_{\ga})}\, =\, \langle \mathcal{ L}(p_t-q_0) ,(p_t-q_0)\rangle_{L^2(w_{\ga})}\\
+(\dot m_t -\gd h(m_t))\cdot \left(\int_{\bbR^d}  (p_t-q_0)\nabla(p_t-q_0) \,w_{\ga}+ \int_{\bbR^d}  (p_t-q_0)\nabla q_0 \,w_{\ga}\right)\\
-\delta\int_{\bbR^d}(p_t-q_0)\nabla\cdot (F_t (p_t-q_0))w_{\ga}\\
-\delta\int_{\bbR^d}(p_t-q_0)\nabla\cdot(q_0 F_t)w_{\ga}\, .
\end{multline}
By integration by parts, we obtain
\begin{equation}\label{eq:by part Ft}
-\delta\int_{\bbR^d}(p_t-q_0)\nabla\cdot (F_t (p_t-q_0))w_{\ga} \,=\, \frac{\delta}{2}\int_{\bbR^d} (p_t-q_0)^2\left(\ga F_t\cdot K\gs^{-2} x- \nabla \cdot F_t\right)w_{\ga}\, ,
\end{equation}
and since from Lemma \ref{lem:apriori_bound_mt} we have $\vert m_t \vert \leq 4L$ on $[0, 2T]$, by hypotheses \eqref{hyp:F_dot_x_bound} and \eqref{hyp:F_dot_x_larger}, there exists a constant $C=C_{ \alpha, T}>0$, independent of $ \delta$ such that
\begin{align*}
-\delta\int_{\bbR^d}(p_t-q_0)\nabla\cdot (F_t (p_t-q_0))w_{\ga} \, \leq\,  C\delta \left\Vert p_{ t}-q_{ 0} \right\Vert_{ L^{ 2}(w_{\ga})}^{ 2} \, .
\end{align*}
Using in particular \eqref{eq poinc w-ga}, we have moreover the bounds, for a constant $C_{\ga, \gs}>0$,
\begin{align}
\left\vert \int_{\bbR^d} (p_t-q_0)\nabla (p_t-q_0) w_{\ga}\right\vert\, &\leq\,C_{\ga,\gs}\Vert \nabla(p_t-q_0)\Vert^2_{L^2(w_{\ga})}\, ,\\
\left\vert \int_{\bbR^d}  (p_t-q_0)\nabla q_0 w_{\ga}\right\vert\, &\leq\, \Vert p_t-q_0\Vert_{L^2(w_{\ga})}\Vert \nabla q_0\Vert_{L^2(w_{\ga})}\, ,
\end{align}
and
\begin{equation}
\left| \int_{\bbR^d}  (p_t-q_0)\nabla\cdot( q_0 F_t) w_{\ga}\right|\, \leq\, \Vert p_t-q_0\Vert_{L^2(w_{\ga})}\Vert \nabla \cdot(q_0 F_{ t})\Vert_{L^2(w_{\ga})}\, .
\end{equation}
The exponential control \eqref{hyp:bound_F_exp} on $F$, Lemma~\ref{lem:bound L OU weighted L2}, Lemma~\ref{lem:apriori_bound_mt} and the definition of $ \alpha$ in \eqref{hyp:alpha_beta_gamma} imply that $\nabla q_0$ and $\nabla\cdot (q_0 F_{ t})$ are elements of $L^2(w_{\ga})$, uniformly in $t\in[0, 2T]$. By \eqref{eq:slow fast PDE modified} and the definition of $h$ in Section~\ref{sec:modification dynamics}, we have by the same arguments that $\dot m_t$ is uniformly bounded by $ \gd(C_F \kappa_{1}+C_h)$. Putting all these estimates together, we obtain, for constants $c,c',C>0$ (depending in particular on $T$ and $F$), 
\begin{multline}\label{eq:bound ddt p-q}
\frac12\frac{\dd}{\dd t}\Vert p_t-q_0\Vert^2_{L^2(w_{\ga})}\,\leq\, -\left(\frac{k_{\ga}}{\ga(\Tr(K)+ \underline{ k})}-c\delta\right)\Vert \gs \nabla(p_t-q_0)\Vert^2_{L^2(w_{ \ga})} +  C\delta \Vert p_t-q_0\Vert_{L^2(w_{ \ga})}\\
\leq\, -(k_{\ga}-c'\delta )\Vert p_t-q_0\Vert^2_{L^2(w_{ \ga})}+  C\delta \Vert p_t-q_0\Vert_{L^2(w_{\ga})}\,,
\end{multline}
where $k_{ \alpha}$ is defined in \eqref{eq:k_alpha_gamma} (recall Remark~\ref{rem:meaning_constants} so that in particular $k_{\ga}>0$).
This concludes the proof of Lemma~\ref{lem:mu bounded L2}, with $\delta_2= \min \left(\frac{k_{\ga}}{2 c'}, \delta_{1}\right)$ and $ \kappa_{2}=\frac{2 C}{k_{\ga}}$.
\end{proof}
The following lemma is the equivalent of Lemma~\ref{lem:mu bounded L2} for the control of the gradient of~$p$:
\begin{lemma}\label{lem:nablamu_bounded_L2}
Fix $ \alpha, \beta$ as in \eqref{hyp:alpha_beta_gamma}. For every $T>0$, there exist $\delta_3= \delta_{ 3}(T)>0$ and $ \kappa_{3}= \kappa_{ 3}(T)>0$ such that if we define for all $m\in \bbR^d$ (recall \eqref{eq:def cF})
\begin{multline}\label{eq:def cG}
\cG(m)\, =\, \cG(m,\gd,\kappa_1,\kappa_2,\kappa_3)\\:=\, \cF(m,\gd,\kappa_1,\kappa_2)\cap \bigg\{p\in \cP_2:\, \Vert \nabla (p-q_0)\Vert^2_{L^2(w_{\gb})}\leq \kappa_3\gd\bigg\}\, ,
\end{multline}
then, for all $0\leq\delta\leq \delta_{ 3}$, the following holds:
\begin{equation}
\label{eq:nablamu_bounded_L2}
\sup_{|m_0|\leq 3L}\quad 
\sup_{p_0\in \cG(m_0)}\quad \sup_{t\in[0,2T]} \quad 
 \left\Vert \nabla\left(p_t^{p_0,m_0}-q_0\right)\right\Vert_{L^2(w_{\beta})}\, \leq\, \kappa_3\gd\, .
\end{equation}
\end{lemma}

\begin{proof}[Proof of Lemma~\ref{lem:nablamu_bounded_L2}]
We place ourselves in the framework of Lemma~\ref{lem:mu bounded L2} and suppose that $ \delta\leq \delta_{ 2}$. We denote $p_t=p^{p_0,m_0}_t$ and $m_t=m^{p_0,m_0}_t$ and fix $j=1, \ldots, d$. Then $\ell_{ t}^{ (j)}:=\partial_{ x_{ j}} \left( p_{ t} - q_{ 0}\right)$ is a weak solution to
\begin{align}
\partial_t \ell_{ t}^{ (j)}\, &=\, \mathcal{ U}_{j} \ell_{ t}^{ (j)} +\nabla\cdot (\partial_{ x_{ j}}p_t (\dot m_t - \delta h(m_{ t})))-\delta \nabla\cdot \big(\partial_{ x_{ j}}(p_t F_t)\big)\, ,
\end{align}
where we recall the definition of the operator $ \mathcal{ U}_j$ in \eqref{eq:def_U}. Then (Remark \ref{rem:non rig}),
\begin{multline}
 \frac{ 1}{ 2} \frac{ {\rm d}}{ {\rm d} t} \left\Vert \ell_{ t}^{ (j)} \right\Vert_{ L^{ 2}(w_{  \beta})}^{ 2} = \left\langle  \mathcal{ U}_{ j} \ell_{ t}^{ (j)}\, ,\, \ell_{ t}^{ (j)}\right\rangle_{ L^{ 2}(w_{ \beta})} + \int \nabla\cdot (\partial_{ x_{ j}}p_t (\dot m_t - \delta h(m_{ t}))) \ell_{ t}^{ (j)} w_{  \beta}\\
 - \delta \int \nabla\cdot \big(\partial_{ x_{ j}}(p_t F_t)\big) \ell_{ t}^{ (j)} w_{  \beta},
\end{multline}
which gives by integration by parts,
\begin{align}
 \frac{ 1}{ 2} \frac{ {\rm d}}{ {\rm d} t} \left\Vert \ell_{ t}^{ (j)} \right\Vert_{ L^{ 2}(w_{ \beta})}^{ 2} &= \,\left\langle  \mathcal{ U}_{j} \ell_{ t}^{ (j)}\, ,\, \ell_{ t}^{ (j)}\right\rangle_{ L^{ 2}(w_{\beta})} \nonumber\\
&\quad - (\dot m_t - \delta h(m_{ t})) \cdot \left(\frac{ \beta}{ 2}\int \left(\ell_{ t}^{ (j)}\right)^{ 2} K\gs^{-2} x w_{ \beta}- \int \ell_{ t}^{ (j)} \nabla (\partial_{ x_{ j}} q_{ 0} )w_{ \beta} \right) \nonumber\\
& \quad  +\frac{ \delta}{ 2} \int \left(\ell_{ t}^{ (j)}\right)^{ 2} \left(\beta F_{ t}\cdot K\gs^{-2} x - \nabla \cdot F_{ t}\right) w_{ \beta} - \delta \sum_{ l=1}^{ d} \int \ell_{ t}^{ (j)} \ell_{ t}^{ (l)} \partial_{ x_{ j}} F_{ t}^{ (l)} w_{ \beta} \nonumber\\
&\quad  - \delta \int \ell_{ t}^{ (j)} (p_{ t}- q_{ 0}) \nabla \cdot  \left(\partial_{ x_{ j}} F_{ t}\right) w_{ \beta} - \delta \int \nabla \cdot \left[\partial_{ x_{ j}} \left(q_{ 0} F_{ t}\right)\right] \ell_{ t}^{ (j)} w_{ \beta}, \nonumber\\
&=:\,  L^{ (1)}_{ t} + L^{ (2)}_{ t} + L^{ (3)}_{ t} + L^{ (4)}_{ t}. \label{eq:sum_Lt_i}
\end{align}
Let us treat the different terms in \eqref{eq:sum_Lt_i} apart: since $ \int_{ \mathbb{ R}^{ d}} \ell_{ t}^{ (j)}(x) {\rm d} x =0$, $ \int x_j \ell_{ t}^{ (j)}(x)=0$ and $\int_{\bbR^d }x_i^n \ell^{(j)}_t=0$ for $n\geq 1$, one has by Lemma~\ref{lem:control_U} 
\begin{align}
L^{ (1)}_{ t}\, =\, \left\langle  \mathcal{ U}_{ j} \ell_{ t}^{ (j)}\, ,\, \ell_{ t}^{ (j)}\right\rangle_{ L^{ 2}(w_{ \beta})} \leq - k^{ \prime}_\gb\Vert \ell_{ t}^{ (j)}\Vert^2_{L^2(w_{\beta})}-\frac{ \underline{ k}\beta(1-\beta)}{2\bar \sigma}\Vert  x \ell_{ t}^{ (j)}\Vert^2_{L^2(w_{ \beta})}\, ,\label{eq:Lt_1}
\end{align}
Moreover by \eqref{eq:slow fast PDE modified},  \eqref{eq:exp_control_Xt}, 
\begin{equation}
L^{ (2)}_{ t}\, \leq\, C_{ 1} \delta\left(\left\Vert x\ell_{ t}^{ (j)} \right\Vert_{ L^{ 2}(w_{  \beta})}^{ 2} + \left\Vert \ell_{ t}^{ (j)} \right\Vert_{ L^{ 2}(w_{  \beta})} \left\Vert \nabla \partial_{ x_{ j}} q_{ 0} \right\Vert_{ L^{ 2}(w_{  \beta})}\right)\, ,\label{eq:Lt_2}
\end{equation}
for some $C_{ 1}>0$ and where $\left\Vert \nabla \partial_{ x_{ j}} q_{ 0} \right\Vert_{ L^{ 2}(w_{ \beta})}$ is bounded since $ \beta < 1$ (recall \eqref{eq:def_alpha_beta}). Concerning $L_{ t}^{ (3)}$, we have
\begin{align}
L_{ t}^{ (3)} &\leq \frac{ \delta}{ 2} \int \left(\ell_{ t}^{ (j)}\right)^{ 2} \left( \beta F_{ t}\cdot K\gs^{-2} x - \nabla \cdot F_{ t}\right) w_{ \beta} + \frac{ \delta}{ 2} \sum_{ l=1}^{ d} \int \left(\ell_{ t}^{ (j)}\right)^{ 2}\left\vert \partial_{ x_{ j}} F_{ t}^{ (l)} \right\vert w_{  \beta}\nonumber\\
&\qquad \qquad +\frac{ \delta}{ 2} \sum_{ l=1}^{ d} \int \left(\ell_{ t}^{ (l)}\right)^{ 2} \left\vert \partial_{ x_{ j}} F_{ t}^{ (l)} \right\vert w_{  \beta},\nonumber\\
&=\frac{ \delta}{ 2} \int \left(\ell_{ t}^{ (j)}\right)^{ 2} \left( \beta F_{ t}\cdot K\gs^{-2} x - \nabla \cdot F_{ t} + \left\vert \partial_{ x_{ j}}F_{ t} \right\vert_{ 1}\right) w_{  \beta} +\frac{ \delta}{ 2} \sum_{ l=1}^{ d} \int \left(\ell_{ t}^{ (l)}\right)^{ 2} \left\vert \partial_{ x_{ j}} F_{ t}^{ (l)} \right\vert w_{  \beta},\label{eq:Lt_3}
\end{align}
for $ \left\vert u \right\vert_{ 1}:= \left\vert u_{ 1} \right\vert+ \ldots+ \left\vert u_{ d} \right\vert$. Lastly,
\begin{align}
L_{ t}^{ (4)} &\leq \delta \int \left\vert \ell_{ t}^{ (j)} (p_{ t}- q_{ 0}) \nabla \cdot  \left(\partial_{ x_{ j}} F_{ t}\right) \right\vert w_{ \beta} + \delta \int  \left\vert \ell_{ t}^{ (j)} \right\vert \left\vert \nabla \cdot \left[\partial_{ x_{ j}} \left(q_{ 0} F_{ t}\right)\right]\right\vert  w_{  \beta},\nonumber\\
&\leq \delta\left\Vert \ell_{ t}^{ (j)} \right\Vert_{ L^{ 2}(w_{  \beta})} \left(\left\Vert (p_{ t}- q_{ 0}) \nabla \cdot  \left(\partial_{ x_{ j}} F_{ t}\right) \right\Vert_{ L^{ 2}(w_{  \beta})} + \left\Vert \nabla \cdot \left[ \partial_{ x_{ j}} \left(q_{ 0}F_{ t}\right)\right]\right\Vert_{ L^{ 2}(w_{  \beta})}\right).\label{eq:Lt_4}
\end{align}
By \eqref{hyp:bound_F_exp}, we have
\begin{multline}\label{eq:bound p-q nabla F}
\left\Vert (p_{ t}- q_{ 0}) \nabla \cdot  \left(\partial_{ x_{ j}} F_{ t}\right) \right\Vert_{ L^{ 2}(w_{  \beta})}^{ 2}\, =\, \sum_{ l=1}^{ d}\int (p_{ t}(x) - q_{ 0}(x))^{ 2}  \left\vert \partial_{ x_{ j}x_{ l}}^{ 2} F^{ (l)}(x+m_{ t}) \right\vert^{ 2} e^{ \frac{\beta}{2}x\cdot K\gs^{-2} x } {\rm d} x\\
\leq\, d C_{ F}^{ 2}\int (p_{ t}(x) - q_{ 0}(x))^{ 2}  \exp \left(\frac{ 1}{ 2}\left( 2\gep (x-m_t)\cdot K\gs^{-2} (x-m_t)  + \beta x\cdot K\gs^{-2} x\right)\right) {\rm d} x\\
\leq\, d C_{ F}^{ 2} e^{ \Tr(K\gs^{-2})\left(1+\frac{1}{\zeta}\gep\right) (4L)^{ 2}}\int (p_{ t}(x) - q_{ 0}(x))^{ 2}  w_{ \ga}(x) {\rm d} x< \infty\, ,
\end{multline}
where we have used Lemma~\ref{lem:apriori_bound_mt}, Lemma \ref{lem:mu bounded L2}, and where $\zeta>0$ has been chosen such that $\beta+2\gep(1+\zeta)\leq \alpha$ (possible since $ \beta + 2 \varepsilon < \alpha$, by \eqref{eq:def_alpha_beta}). In a similar way, since $\gb+2\gep<1$, using again \eqref{hyp:bound_F_exp},
\begin{equation}
\left\Vert \nabla \cdot \left[ \partial_{ x_{ j}} \left(q_{ 0}F_{ t}\right)\right]\right\Vert_{ L^{ 2}(w_{  \beta})}^{ 2} \,=\, \sum_{ l=1}^{ d} \int \left\vert \partial_{ x_{ l}x_{ j}} (q_{ 0} F_{ t}^{ (l)})\right\vert^{ 2} e^{ \frac{\beta}{2} x\cdot K\gs^{-2} x} {\rm d} x\ <\, \infty\, .
\end{equation}
Gathering the previous estimates \eqref{eq:Lt_1}, \eqref{eq:Lt_2}, \eqref{eq:Lt_3} and \eqref{eq:Lt_4} into \eqref{eq:sum_Lt_i}, we obtain, for some constant $C>0$ (that does not depend on $j$),
\begin{multline}
 \frac{ 1}{ 2} \frac{ {\rm d}}{ {\rm d} t} \left\Vert \ell_{ t}^{ (j)} \right\Vert_{ L^{ 2}(w_{ \beta})}^{ 2}\leq -  \left(k^{ \prime}_\beta - C\gd \right)\Vert \ell_{ t}^{ (j)}\Vert^2_{L^2(w_{ \beta})}-\left(\frac{ \underline{ k}\beta(1-\beta)}{2\bar \sigma} - C\gd\right)\Vert x \ell_{ t}^{ (j)}\Vert^2_{L^2(w_{ \beta})}\\
+ \delta C \left\Vert \ell_{ t}^{ (j)} \right\Vert_{ L^{ 2}(w_{  \beta})}+\frac{ \delta}{ 2} \int \left(\ell_{ t}^{ (j)}\right)^{ 2} \left( \beta F_{ t}\cdot x - \nabla \cdot F_{ t} + \left\vert \partial_{ x_{ j}}F_{ t} \right\vert_{ 1}\right) w_{ \beta} 
\\ +\frac{ \delta}{ 2} \sum_{ l=1}^{ d} \int \left(\ell_{ t}^{ (l)}\right)^{ 2} \left\vert \partial_{ x_{ j}} F_{ t}^{ (l)} \right\vert w_{ \beta}\, .
\end{multline}
Let $S_{ t}:= \sum_{ j=1}^{ d}  \left\Vert \ell_{ t}^{ (j)} \right\Vert_{ L^{ 2}(w_{\beta})}^{ 2}$. Choosing $ \delta>0$ small enough such that $\frac{ \underline{ k}\beta(1-\beta)}{2\bar \sigma} - C\gd >0$ and summing over $j=1, \ldots, d$ in the previous inequality leads to
\begin{multline}
 \frac{ 1}{ 2} \frac{ {\rm d}}{ {\rm d} t} S_{ t}\, \leq\,  -  \left(k^{ \prime}_\beta - C\gd\right) S_{ t} + \delta C \sum_{ j=1}^{ d}\left\Vert \ell_{ t}^{ (j)} \right\Vert_{ L^{ 2}(w_{  \beta})}\\
 +\frac{ \delta}{ 2} \int \sum_{ j=1}^{ d}\left(\ell_{ t}^{ (j)}\right)^{ 2} \left( \beta F_{ t}\cdot K\gs^{-2}x - \nabla \cdot F_{ t} + \left\vert \partial_{ x_{ j}}F_{ t} \right\vert_{ 1} + \sum_{ l=1}^{ d}\left\vert \partial_{ x_{ l}} F_{ t}^{ (j)} \right\vert\right) w_{ \beta}\,.
\end{multline}
Recalling the hypothesis \eqref{hyp:F_dot_x_larger}, the term in the second line is in fact bounded by $C'S_t$ for some $C'>0$, and we have
\begin{equation}
 \frac{ 1}{ 2} \frac{ {\rm d}}{ {\rm d} t} S_{ t}\, \leq\,  -  \left(k^{ \prime}_\beta - (C+C')\gd \right) S_{ t} + \delta C \sqrt{ d} \sqrt{ S_{ t}},
\end{equation}
so that the result holds for $ \delta_{3}:= \min \left( \delta_{ 2}, \frac{ \underline{ k}\beta(1-\beta)}{ 4\bar \sigma C},  \frac{ k^{ \prime}_\beta}{2(C+C')}\right)$ and $ \kappa_{ 3}:= \frac{ 2 C \sqrt{ d}}{ k^{ \prime}_\beta}$, by Lemma~\ref{lem:gronwall_sqrt}.
\end{proof}

\subsection{Lipschitz-continuity close to $\cM^0$}
\begin{lemma}
\label{lem:lip_cont_pi}
Fix $ \alpha, \beta$ as in \eqref{hyp:alpha_beta_gamma}. For every $T>0$, there exist $ \delta_{4}= \delta_{ 4}(T)>0$ and $\kappa_{4}= \kappa_{ 4}(T)>0$ such that for all $ 0\leq\delta\leq \delta_{4}$, $m^i_0\in \bbR^d$ with $\vert m^i_0\vert \leq 3L$ ($i=1,2$) and all $p^i_0\in   \cF(m_i,\gd,\kappa_1,\kappa_2)$, we have for all $t\leq 2T$, denoting $p^i_t=p^{p^i_0,m^i_0}_t$:
\begin{equation}\label{eq:lip_cont_pi 1}
\left\Vert p_{ t}^{ 1} - p_{ t}^{ 2} \right\Vert^{ 2}_{ L^{ 2}(w_{  \beta})}\, \leq\, \max\left( \left\Vert p_{ 0}^{ 1}- p_{ 0}^{ 2} \right\Vert^{ 2}_{ L^{ 2}(w_{  \beta})}\, , \quad \kappa_{4} \sup_{ s\leq t} \left\vert m_{ s}^{ 1}- m_{ s}^{ 2} \right\vert^{ 2} \gd \right)\, ,
\end{equation}
and
\begin{equation}\label{eq:lip_cont_pi 2}
\left\Vert p_{ t}^{ 1} - p_{ t}^{ 2} \right\Vert^{ 2}_{ L^{ 2}(w_{ \beta})}\, \leq\, e^{ - \chi(\gd) t}\left\Vert p_{ 0}^{ 1}-p_{ 0}^{ 2} \right\Vert_{ L^{ 2}(w_{ \beta})}^{ 2} +   \kappa_{ 4}\sup_{ s\leq t} \left\vert m_{ s}^{ 1} - m_{ s}^{ 2}  \right\vert^{ 2}  \gd \, ,
\end{equation}
where, for some constant $c_1>0$ depending on $F$, $L$, $h$ and $T$,
\begin{equation}\label{eq:def chi}
 \chi(\gd)\, :=\, k_{\beta}-c_1 \gd\, .
\end{equation}
\end{lemma}
\begin{proof}[Proof of Lemma~\ref{lem:lip_cont_pi}]
We suppose in this proof that $ \delta\leq \delta_{2}$ (recall Lemma \ref{lem:mu bounded L2}). Define $ \pi_{ t}:= p_{ t}^{ 1}- p_{ t}^{ 2}$. Recall that $p_{ t}^{ i}$, $i=1, 2$, solves
\begin{multline}
\partial_t p_t^{ i}\, =\, \gs^ 2 \mathcal{ L}p_t^{ i}+\delta \nabla\cdot \big(p_t^{ i} \left(\dot m^i_t - h(m^i_t)-F_t^{ i}\right)\big)\\
=\, \gs^ 2 \mathcal{ L}p_t^{ i}+\delta \nabla\cdot \big(p_t^{ i} \left(\langle p^i_t\, ,\, F^i_t\rangle -F_t^{ i}\right)\big)\, ,
\end{multline}
for $F_{ t}^{ i}(x):= F(x+ m_{ t}^{ i})$. Hence, noting that $\langle p^i_t\, ,\, F^i_t\rangle=\langle \mu^i_t\, ,\, F\rangle $ where $\mu^i_t$ is the solution to \eqref{eq:PDE_mod_h} with initial condition $p^i_0(\cdot-m^i_0)$, and using the notation $F^{12}_t:= F^1_t-F^2_t$,
we have
\begin{multline}
\partial_t \pi_{ t}\, =\,  \mathcal{ L}\pi_{ t}- \delta \nabla\cdot \big( \pi_{ t} \left( F_t^{ 1}- \left\langle \mu_{ t}^{ 1}\, ,\, F\right\rangle \right) \big)- \delta \nabla\cdot \big(p_t^{ 2} \left(F_t^{12}- \left\langle \mu_{ t}^{ 1} - \mu_{ t}^{ 2}\, ,\, F\right\rangle\right)\big)\, ,
\end{multline}
which leads to
\begin{multline}\label{eq:ddt Vert pi_t 2}
\frac{ 1}{ 2}\frac{ {\rm d}}{ {\rm d} t} \left\Vert \pi_{ t} \right\Vert_{ L^{ 2}(w_{  \beta})}^{ 2} =\,  \left\langle \mathcal{ L} \pi_{ t}\, ,\, \pi_{ t}\right\rangle_{ L^{ 2}(w_{  \beta})}
 - \delta \left\langle \nabla\cdot \left( \pi_{ t} (F_{ t}^{ 1} - \left\langle \mu_{ t}^{ 1}\, ,\, F\right\rangle)\right)\, ,\, \pi_{ t}\right\rangle_{ L^{ 2}(w_{  \beta})}\\
- \delta \left\langle \nabla \cdot \left(p_{ t}^{ 2} \left( F_{ t}^{12}  - \left\langle \mu_{ t}^{ 1} - \mu_{ t}^{ 2}\, ,\, F\right\rangle\right)\right)\, ,\, \pi_{ t}\right\rangle_{ L^{ 2}(w_{  \beta})}.
\end{multline}
Let us treat the different terms appearing in \eqref{eq:ddt Vert pi_t 2}.
Using Lemma~\ref{lem:bound L OU weighted L2}, we have
\begin{equation}
\left\langle \mathcal{ L} \pi_{ t}\, ,\, \pi_{ t}\right\rangle_{ L^{ 2}(w_{ \beta})}\, \leq\, -  \frac{ k_{ \beta}}{ \beta(\Tr(K)+ \underline{ k})}\Vert \gs \nabla\pi_{ t}\Vert^2_{L^2(w_{\beta})}\, .
\end{equation}
Moreover, by a integration by parts (see for example \eqref{eq:simplify F term}), hypotheses \eqref{hyp:F_dot_x_bound} and \eqref{hyp:F_dot_x_larger}, Lemma \ref{lem:apriori_bound_mt} and Lemma \ref{lem:bound L OU weighted L2}, we obtain for some $C>0$
\begin{multline}
-\gd\left\langle \nabla\cdot \left( \pi_{ t} (F_{ t}^{ 1} - \left\langle \mu_{ t}^{ 1}\, ,\, F\right\rangle)\right)\, ,\, \pi_{ t}\right\rangle_{ L^{ 2}(w_{\beta})}
\\ =\, \frac{ \delta}{ 2} \int_{ \mathbb{ R}^{ d}} \left\vert \pi_{ t} \right\vert^{ 2}\left( \beta F_{ t}^{ 1}\cdot K\gs^{-2}x - \nabla \cdot F_{ t}^{ 1} - \beta\left\langle \mu_{ t}^{ 1}\, ,\, F\right\rangle\cdot  K\gs^{-2} x\right)w_{\beta} \\ \leq\, C\gd \Vert \nabla\pi_t\Vert_{L^2(w_{\beta})}^2\, .
\end{multline}

Lastly, we have the decomposition
\begin{multline}
- \delta \left\langle \nabla \cdot \left(p_{ t}^{ 2} \left(F_{ t}^{ 12} - \left\langle \mu_{ t}^{ 1} - \mu_{ t}^{ 2}\, ,\, F\right\rangle\right)\right)\, ,\, \pi_{ t}\right\rangle_{ L^{ 2}(w_{\beta})}\\
=\,
\delta\int_{ \mathbb{ R}^{ d}} p_{ t}^{ 2}  \left\lbrace F_{ t}^{ 12}  - \left\langle \mu_{ t}^{ 1} - \mu_{ t}^{ 2}\, ,\, F\right\rangle\right\rbrace \cdot \nabla \pi_{ t} w_{  \beta}\\
+ \delta \beta \int_{ \mathbb{ R}^{ d}} p_{ t}^{ 2} \pi_{ t} \left\lbrace F_{ t}^{ 12} - \left\langle \mu_{ t}^{ 1} - \mu_{ t}^{ 2}\, ,\, F\right\rangle\right\rbrace \cdot K\gs^{-2} x w_{  \beta}\, ,
\end{multline}
and the following estimates :
\begin{multline}
\int_{ \mathbb{ R}^{ d}} p_{ t}^{ 2}  \left(F_{ t}^{ 12} - \left\langle \mu_{ t}^{ 1} - \mu_{ t}^{ 2}\, ,\, F\right\rangle\right) \cdot \nabla \pi_{ t} w_{\beta}\\ \leq\,  \left\Vert \nabla \pi_{ t} \right\Vert_{ L^{ 2}(w_{  \beta})} \left( \left\Vert p_{ t}^{ 2}  F_{ t}^{ 12} \right\Vert_{ L^{ 2}(w_{\beta})} + \left\vert \left\langle \mu_{ t}^{ 1} - \mu_{ t}^{ 2}\, ,\, F\right\rangle \right\vert \left\Vert p_{ t}^{ 2} \right\Vert_{ L^{ 2}(w_{ \beta})}\right)\, ,
\end{multline}
and
\begin{multline}
\int_{ \mathbb{ R}^{ d}} p_{ t}^{ 2} \pi_{ t} \left\lbrace F_{ t}^{ 1} - F_{ t}^{ 2} - \left\langle \mu_{ t}^{ 1} - \mu_{ t}^{ 2}\, ,\, F\right\rangle\right\rbrace \cdot x w_{ \beta} \\
\leq\, \left\Vert \pi_{ t} \right\Vert_{ L^{ 2}(w_{  \beta})}  \left(  \left\Vert  p_{ t}^{ 2}  F_{ t}^{ 12}\cdot K\gs^{-2}x \right\Vert_{ L^{ 2}(w_{  \beta})} + \left\vert \left\langle \mu_{ t}^{ 1} - \mu_{ t}^{ 2}\, ,\, F\right\rangle \right\vert  \left\Vert K\gs^{-2}x  p_{ t}^{ 2} \right\Vert_{ L^{ 2}(w_{  \beta})}\right)\, .
\end{multline}
Let us now treat the different terms appearing in the right hand side of these last two estimates. We have, by Lemma \ref{lem:apriori_bound_mt}, Lemma \ref{lem:mu bounded L2} and the fact that $\beta+2\gep<\ga$ (as it has been done in \eqref{eq:bound p-q nabla F}):
\begin{multline}
\left\Vert  p_{ t}^{ 2}  F_{ t}^{ 12}\cdot x \right\Vert_{ L^{ 2}(w_{  \beta})}^2
\, \leq\,  \int_{ \mathbb{ R}^{ d}} \left\vert p_{ t}^{ 2}(x) \right\vert^{ 2} \left\vert F(x+m_{ t}^{ 1}) - F(x+m_{ t}^{ 2}) \right\vert^{ 2} \left\vert x \right\vert^{ 2} w_\beta(x) {\rm d} x,\\
\leq\,  C_{ F}^{ 2} \left\vert m_{ t}^{ 1}- m_{ t}^{ 2} \right\vert^{ 2}\int_{ \mathbb{ R}^{ d}} \left\vert p_{ t}^{ 2}(x) \right\vert^{ 2} \left\vert w_\gep(x-m^1_t)+w_\gep(x-m^2_t)\right\vert^{ 2} \left\vert x \right\vert^{ 2} w_\beta(x) {\rm d} x,\\
\leq\,  C_{F,r}  \left\vert m_{ t}^{ 1}- m_{ t}^{ 2} \right\vert^{ 2} \left\Vert p_{ t}^{ 2} \right\Vert_{ L^{ 2}(w_{\alpha})}^{ 2}\, \leq\, C'_{F,r} \left\vert m_{ t}^{ 1}- m_{ t}^{ 2} \right\vert^{ 2}\, .
\end{multline}
The same type of estimate can be obtained for $\left\Vert  p_{ t}^{ 2}  F^{12}_t \right\Vert_{ L^{ 2}(w_{\beta})}$. Moreover, we have
\begin{equation}\label{eq:first boun mui F}
\left\vert \left\langle \mu_{ t}^{ 1} - \mu_{ t}^{ 2}\, ,\, F\right\rangle \right\vert
\, \leq\, \int_{ \mathbb{ R}^{ d}} \left\vert F(x+m_{ t}^{ 1}) \right\vert \left\vert \pi_{ t}(x) \right\vert {\rm d}x  + \int_{ \mathbb{ R}^{ d}} \left\vert F^{12}_t(x) \right\vert p_{ t}^{ 2}(x) {\rm d}x\, .
\end{equation}
On one hand, by \eqref{eq:def_alpha_beta}, we have $\beta>2\gep$ so that
\begin{multline}
\int_{ \mathbb{ R}^{ d}} \left\vert F(x+m_{ t}^{ 1}) \right\vert \left\vert \pi_{ t}(x) \right\vert {\rm d}x\, \leq\, \left( \int_{ \mathbb{ R}^{ d}} \left\vert F(x+m_{ t}^{ 1}) \right\vert^{ 2} w_{-\beta}(x) {\rm d}x\right)^{ \frac{ 1}{ 2}} \left\Vert \pi_{ t} \right\Vert_{ L^{ 2}(w_{  \beta})}\\ \leq\, C\left\Vert \pi_{ t} \right\Vert_{ L^{ 2}(w_{  \beta})}\, ,
\end{multline}
and on the other hand, by Hypothesis \eqref{hyp:bound_F_exp}, Lemma \ref{lem:apriori_bound_mt}, Lemma \ref{lem:mu bounded L2} and the fact that $2\gep <\ga$,
\begin{multline}
\int_{ \mathbb{ R}^{ d}} \left\vert F^{12}_t(x) \right\vert p_{ t}^{ 2}(x) {\rm d}x\, \leq\, \left\vert m_{ t}^{ 1} - m_{ t}^{ 2} \right\vert\int_{ \mathbb{ R}^{ d}} \sup_{ z\in [x+m_{ t}^{ 1}, x+m_{ t}^{ 2}]}\left\vert DF(z)\right\vert p_{ t}^{ 2}(x) {\rm d}x\\
\leq\,  C_{ F}\left\vert m_{ t}^{ 1} - m_{ t}^{ 2} \right\vert\int_{ \mathbb{ R}^{ d}} \left( w_\gep(x+m^1_t)+w_\gep(x+m^2_t)\right) p_{ t}^{ 2}(x) {\rm d}x\\
\leq\, C_{F,r} \left\vert m_{ t}^{ 1} - m_{ t}^{ 2} \right\vert \left\Vert p^2_t\right\Vert_{L^2(w_{\ga})}\, \leq\, C'_{F,r} \left\vert m_{ t}^{ 1} - m_{ t}^{ 2} \right\vert\, ,
\end{multline}
so that \eqref{eq:first boun mui F} becomes
\begin{equation}\label{eq:second bound mui F}
\left\vert \left\langle \mu_{ t}^{ 1} - \mu_{ t}^{ 2}\, ,\, F\right\rangle \right\vert
\, \leq\, C''_{F} \left(\left\Vert \pi_{ t} \right\Vert_{ L^{ 2}(w_{\beta})}+\left\vert m_{ t}^{ 1} - m_{ t}^{ 2} \right\vert \right)\, .
\end{equation}
Gathering all the preceding estimates, we deduce that there exists $C>0$ such that
\begin{equation}
\frac{ 1}{ 2}\frac{ {\rm d}}{ {\rm d} t} \left\Vert \pi_{ t} \right\Vert_{ L^{ 2}(w_{  \beta})}^{ 2}\, \leq\, - \left( \frac{ k_{ \beta}}{ \beta(\Tr(K)+ \underline{ k})}-C\gd\right)\Vert  \nabla\pi_{ t}\Vert^2_{L^2(w_{\beta})}+C\gd \sup_{s\leq t}\left\vert m^1_t-m^2_t\right\vert^2\, ,
\end{equation}
which leads to, by Lemma \ref{lem:bound L OU weighted L2}, for some $C'>0$,
\begin{equation}
\frac{ 1}{ 2}\frac{ {\rm d}}{ {\rm d} t} \left\Vert \pi_{ t} \right\Vert_{ L^{ 2}(w_{  \beta})}^{ 2}\, \leq\, - \left( k_{ \beta}-C'\gd\right)\Vert  \pi_{ t}\Vert^2_{L^2(w_{ \beta})}+C\gd \sup_{s\leq t}\left\vert m^1_t-m^2_t\right\vert^2\, .
\end{equation}
We deduce \eqref{eq:lip_cont_pi 1} and \eqref{eq:lip_cont_pi 2} provided that $ \delta\leq \delta_{4}:= \min \left(\delta_{2}, \frac{ k_{ \beta}}{2 C'}\right)$, for the choice of $\kappa_{4}:= \frac{2C}{k_{\beta}}$.
\end{proof}
\begin{lemma}\label{lem:encadre m}
Fix $ \alpha, \beta$ as in \eqref{hyp:alpha_beta_gamma}. For any $T>0$, there exists $\delta_5= \delta_{ 5}(T)>0$ such that for all $m^i_0\in \bbR^d$ ($i=1,2$) with $\vert m^i_0\vert \leq 3L$ and all $p^i_0 \in \mathcal{ F}(m^i_0,\gd,\kappa_1,\kappa_2)$ satisfying
\begin{equation}
\label{eq:control_deltap_deltam}
\left\Vert p_{ 0}^{ 1} - p_{ 0}^{ 2} \right\Vert_{ L^{ 2}(w_{ \beta})} \leq  \left\vert m_{ 0}^{ 1} - m_{ 0}^{ 2} \right\vert,
\end{equation}
we have for all $0\leq\delta\leq \delta_5$ and $t\in [0,2T]$, denoting $m^i_t=m^{p^i_0,m^i_0}_t$:
\begin{equation}
\label{eq:encadre m}
\frac12 \vert m^1_0-m^2_0\vert \, \leq\, \vert m^1_t-m^2_t\vert\, \leq\, 2\vert m^1_0-m^2_0\vert\, . 
\end{equation}
\end{lemma}
\begin{proof}[Proof of Lemma~\ref{lem:encadre m}]
We suppose here that $ \delta\leq \delta_{4}$ (recall Lemma~\ref{lem:lip_cont_pi}). We have
\begin{multline}
\left\vert \left\vert m_{ t}^{ 1} - m_{ t}^{ 2} \right\vert - \left\vert m_{ 0}^{ 1}- m_{ 0}^{ 2} \right\vert \right\vert\, \leq\, \left\vert m_{ t}^{ 1}- m_{ t}^{ 2}- (m_{ 0}^{ 1}- m_{ 0}^{ 2}) \right\vert\\
 \leq\,  \delta \int_{0}^{t} \left( \left\vert h(m_{ s}^{ 1}) - h(m_{ s}^{ 2}) \right\vert + \left\vert \left\langle \mu_{ s}^{ 1} - \mu_{ s}^{ 2}\, ,\, F\right\rangle \right\vert \right){\rm d}s\, ,
\end{multline}
and by \eqref{eq:second bound mui F}, we obtain for some $C>0$
\begin{equation}
\left\langle \mu_{ s}^{ 1} - \mu_{ s}^{ 2}\, ,\, F\right\rangle \, \leq\, C\left( \left\Vert p^1_s-p^2_s \right\Vert_{ L^{ 2}(w_{ \beta})} + \left\vert m_{ s}^{ 1} - m_{ s}^{ 2} \right\vert \right),\ 0\leq s\leq t\, ,
\end{equation}
so that, by the Lipschitz-continuity of $h$ and Lemma \ref{lem:lip_cont_pi}, for some $C'>0$,
\begin{equation}\label{eq:bound diff m_1}
\left\vert \left\vert m_{ t}^{ 1} - m_{ t}^{ 2} \right\vert  - \left\vert m_{ 0}^{ 1}- m_{ 0}^{ 2} \right\vert \right\vert\, \leq\, C'\gd \int_0^t \left( \left\Vert p^1_0-p^2_0 \right\Vert_{ L^{ 2}(w_{ \beta})} + \sup_{u\leq s}\left\vert m_{ u}^{ 1} - m_{ u}^{ 2} \right\vert \right)\dd s\, .
\end{equation}
Recalling \eqref{eq:control_deltap_deltam}, we get
\begin{equation}
\left\vert \left\vert m_{ t}^{ 1} - m_{ t}^{ 2} \right\vert - \left\vert m_{ 0}^{ 1}- m_{ 0}^{ 2} \right\vert \right\vert\, \leq\, 4 C' T\gd \sup_{s\leq t} \left\vert m_{ s}^{ 1} - m_{ s}^{ 2} \right\vert \, .
\end{equation}
Introducing the time $t_5:= \inf \left\lbrace t\in[0, 2T],\ \left\vert m_{ t}^{ 1}- m_{ t}^{ 2} \right\vert > 2\left\vert m_{ 0}^{ 1} - m_{ 0}^{ 2}\right\vert\right\rbrace$ we deduce that for all $t\leq t_5$,
\begin{equation}
\left\vert \left\vert m_{ t}^{ 1} - m_{ t}^{ 2} \right\vert - \left\vert m_{ 0}^{ 1}- m_{ 0}^{ 2} \right\vert \right\vert
\, \leq \, 8 C' T\gd \left\vert m_{ 0}^{ 1}- m_{ 0}^{ 2} \right\vert\, .
\end{equation}
Choosing $ \delta_5=\min(\gd_4,\frac{C' T}{16})$, we obtain that $t_5=2T$ and the result \eqref{eq:encadre m}.
\end{proof}

\subsection{ The fixed-point problem}
\label{sec:fixed_point_f}
Let $ \mathcal{ C}(B_{2L}, L^{ 2}(w_{\beta}))$ be the set of continuous functions $m \mapsto f(m)$ with values in $L^{ 2}(w_{\beta})$, endowed with the norm
\begin{equation}
\label{eq:norm_inf_minus_beta}
\left\Vert f \right\Vert_{ \infty, \beta}:= \sup_{ m\in B_{2L}} \left\Vert f(m) \right\Vert_{ L^{ 2}(w_{\beta})}.
\end{equation}

\begin{definition}
\label{def:space_F}
Fix $ \delta>0$ and let $\alpha, \beta$ defined as in \eqref{hyp:alpha_beta_gamma}. We denote $\cH=\cH(\gd,\kappa_1,\kappa_2,\kappa_3)$ the subset of $ \mathcal{ C}(B_{2L}, L^{ 2}(w_\beta))$ composed of the elements $f$ satisfying the following conditions:
\begin{enumerate}
\item for all $m\in B_{2L}$, 
\begin{equation}\label{eq:def H 1}
f(m)(\cdot)\, \geq\,  0 \text{ a.e.}\, ,   \qquad  \int_{ \mathbb{ R}^{ d}} f(m)(x) {\rm d}x \, =\, 1 \, , \qquad \text{and} \quad \int_{ \mathbb{ R}^{ d}} x f(m)(x) {\rm d}x\, =\, 0 \, ,
\end{equation}
\item for the constants $\kappa_1, \kappa_2$ and $\kappa_3$ defined respectively in Lemma \ref{lem:apriori_bound_mt}, Lemma \ref{lem:mu bounded L2} and Lemma \ref{lem:nablamu_bounded_L2},
\begin{multline}\label{eq:def H 2}
\sup_{ m\in B_{ 2L}}\int_{ \mathbb{ R}^{ d}} w_{\ga}(x+m) f(m)(x) {\rm d}x \, \leq\,  \kappa_{ 1}\, , \qquad \sup_{m\in B_{2L}} \left\Vert f(m)-q_0\right\Vert_{L^2(w_{\ga})} \,  \leq\,  \kappa_{2} \gd\, ,\\
\text{and} \quad 
  \sup_{m\in B_{2L}} 
 \Vert \nabla(f(m)-q_0)\Vert^2_{L^2(w_{\gb})}\, \leq\, \kappa_3\gd\, ,
\end{multline}
\item for all $m_{ 1}, m_{ 2}\in B_{ 2L}$, 
\begin{equation}
\label{eq:tilde_f_Lip}
 \left\Vert f(m_{ 1})- f(m_{ 2}) \right\Vert_{ L^{ 2}(w_{ \beta})}\leq  \left\vert m_{ 1}- m_{ 2} \right\vert.
\end{equation}
\end{enumerate}
\end{definition}

\begin{rem}
$\cH$ is a complete subset of $\cC(B_{ 2L},L^2(w_{\beta}))$.
\end{rem}

\bigskip

For $t\geq0$, consider the mapping
\begin{equation}
\label{eq:g_ft}
m'\, \mapsto \, g_{t,f}(m')\, :=m_t^{f(m'),m'} \, .
\end{equation}

\begin{lemma}
\label{lem:bij}
For any $T>0$, there exists $\delta_6= \delta_{ 6}(T)>0$ such that if $0\leq\delta\leq\delta_6$ and $f\in\mathcal{ H}$, then for all $m'\in B_{2L}$ and $t\in[0,2T]$ there exists a unique $m\in B_{2L}$ such that $m'=g_{t,f}(m)$. 
\end{lemma}

\begin{proof}[Proof of Lemma~\ref{lem:bij}]
We use ideas from \cite{wiggins2013normally}, relying in particular on approximation results given in \cite{MR0120319}, page 261. Suppose here that $\gd\leq \gd_5$ and consider some $f\in \cH$. Let us extend $f$ artificially to $B_{3L}$ by stating, for any $2L \leq \vert m\vert \leq 3L$, $f(m)=f \left(\frac{2L }{\vert m\vert}m\right)$. With this definition, \eqref{eq:tilde_f_Lip} is still satisfied for $m_1,m_2\in B_{3L}$.

Consider now for $t\in [0,2T]$, the mapping $w_t$ defined for all $m_0\in B_{3L}$ by $w_t(m_0)=g_{t,f}(m_0)-m_0$. Now recalling Lemma \ref{lem:lip_cont_pi}, \eqref{eq:second bound mui F} and Lemma \ref{lem:encadre m}, we have for a positive constant $C$:
\begin{equation}
\vert w_t(m^1_0)-w_t(m^2_0)\vert \leq 
  \delta \int_{0}^{t} \left( \left\vert h(m_{ s}^{ 1}) - h(m_{ s}^{ 2}) \right\vert + \left\vert \left\langle \mu_{ s}^{ 1} - \mu_{ s}^{ 2}\, ,\, F\right\rangle \right\vert \right){\rm d}s\\
\leq\, C \delta \vert m^1_0-m^2_0\vert\, ,
\end{equation}
and by Lemma \ref{lem:apriori_bound_mt},
\begin{equation}
\vert w_t(m_0)\vert \, \leq\, \delta \int_{0}^{t} \left( \left\vert h(m_{ s}) \right\vert + \left\vert \left\langle \mu_{ s}\, ,\, F\right\rangle \right\vert \right){\rm d}s\\
\leq\, C \delta\, .
\end{equation}
We can thus apply the approximation results given in \cite{MR0120319}: for all $t\in [0,2T]$ there exists an homeomorphism $b_t$ defined on the ball $B_{3L-c\delta}$ for some constant $c>0$ such that $b_t(m)=m-w_t(b_t(m))$ for all $m\in B_{3L-c\delta}$. This means that its inverse, i.e. the mapping $m\mapsto m+w_t(m)=g_{t,f}(m)$ is an homeomorphism from $f(B_{3L-c\delta})\supset B_{2L}$ to $B_{3L-c\delta}$, where the inclusion $f(B_{3L-c\delta})\supset B_{2L}$ holds for $\delta$ small enough, since $m\in f(B_{3L-c\delta})$ is equivalent to $g_{t,f}(m)\in B_{3L-c\delta}$, and if $\vert m\vert \leq 2L$ we have $\vert g_{t,f}(m)\vert \leq \vert m\vert +\vert w_t(m)\vert \leq 2L+C\delta\leq 3L-c\delta$ for $\delta\leq \frac{2L}{c+C}$.

It remains to prove that the elements $m'\in B_{2L}$ are images of elements $m\in B_{2L}$ by $g_{t,f}$. But for each trajectory $m_\cdot$ such that $\vert m_t\vert =2L$ we have, recalling \eqref{eq:hyp h},
\begin{multline}
n(m_t)\cdot \dot m_t =\, \gd n(m_t)\cdot \left(h(m_t)+\int_{\bbR^d}F(x)\dd \mu_t(x) \right)\\ \geq \, \gd\xi+\gd n(m_t)\cdot \left(\int_{\bbR^d}F(x)\dd \mu_{t}(x)-\int_{\bbR^d}F(x) q_{m_t}(x)\dd x \right)\, ,
\end{multline}
and, relying on Lemma \ref{lem:mu bounded L2} and the fact that $2\gep < \ga$,
\begin{multline}
\left|\int_{\bbR^d}F(x)\dd \mu_{t}(x)-\int_{\bbR^d}F(x) q_{m_t}(x)\dd x \right|\, \leq\,C_F \left( \int_{\bbR^d} w_\gep(x-m_t) (p_t(x)-q_0(x))^2\dd x\right)^{\frac12}\\ \leq\, C'\gd\, .
\end{multline}
We deduce that if $\gd$ is small enough, then for all trajectory $m_\cdot$ such that $\vert m_t\vert=2L$, we have $n(m_t)\cdot \dot m_t>0$, which means that the trajectories defined by \eqref{eq:PDE_mod_h} can not enter $B_{2L}$. So the elements $m'\in B_{2L}$ are indeed images of elements $m\in B_{2L}$ by $g_{t,f}$, and this concludes the proof.
\end{proof}

\bigskip

From Lemma~\ref{lem:bij} we deduce that for all $T>0$, $t\in[0, 2T]$, we can define a mapping $\eta_{ t}$ as follows:
\begin{equation}
\label{eq:eta_f}
\begin{array}{rcc}
\eta_{t,f}: B_{2L} & \rightarrow & L^2(w_{\beta})\\
           m  &  \mapsto &  p_t^{f( g_{t,f}^{-1}(m)),\, g_{t,f}^{-1}(m)} 
\end{array}\, .
\end{equation}
Form now on, we fix $T:=T_{ 0}$ such that:
\begin{equation}\label{eq:def T}
e^{-k_{\beta} T_{ 0}}\, \leq\, \frac{1}{16}\, .
\end{equation}

\begin{lemma}
\label{lem:etat_Lip}
There exists $\delta_7>0$ such that if $\delta\leq\delta_7$, then for any $f\in  \mathcal{ H}$ and any $t\in[T_{ 0},2T_{ 0}]$, we have $ \eta_{ t, f}\in \mathcal{ H}$.
\end{lemma}
\begin{proof}[Proof of Lemma~\ref{lem:etat_Lip}]
We suppose here that $ \delta\leq \delta_{ 6}$ (recall Lemma~\ref{lem:bij}), we fix $t\in [T_{ 0}, 2T_{ 0}]$ and consider a $f\in \cH$. Then $\eta_{t,f}$ satisfies \eqref{eq:def H 1} by construction.
Moreover it satisfies \eqref{eq:def H 2} as a consequence of Lemma \ref{lem:apriori_bound_mt}, Lemma \ref{lem:mu bounded L2} and Lemma \ref{lem:nablamu_bounded_L2}. By Lemma~\ref{lem:bij}, consider now $m_{ 0}^{ i}= g_{ t, f}^{ -1}(m_{ i})$ for $i=1, 2$, so that, by definition, $m^i_t:=m^{f(m^i_0),m^i_0}_t=m_i$ and for $p^{ i}_t=p^{f(m^i_0),m^i_0}_t$,
\begin{equation}\label{eq:first bound etatf}
\left\Vert{ \eta}_{ t, f}(m_{ 1})- { \eta}_{ t, f}(m_{ 2}) \right\Vert_{ L^{ 2}(w_{  \beta})} \,  =\,  \left\Vert p_{ t}^{ 1} - p_{ t}^{ 2} \right\Vert_{ L^{ 2}(w_{  \beta})}\, .
\end{equation}
Note that in particular $p_{ 0}^{ i}(\cdot) = { f}(m_{ 0}^{ i})$, so that, by hypothesis \eqref{eq:tilde_f_Lip} on $f$, assumption \eqref{eq:control_deltap_deltam} of Lemma \ref{lem:encadre m} is satisfied. Hence, applying Lemma \ref{lem:lip_cont_pi} and Lemma \ref{lem:encadre m}, we deduce
\begin{equation}
\left\Vert p_{ s}^{ 1} - p_{ s}^{ 2} \right\Vert_{ L^{ 2}(w_{  \beta})}^{ 2} \, \leq\, \left(\left\Vert p_{ 0}^{ 1} - p_{ 0}^{ 2} \right\Vert_{ L^{ 2}(w_{  \beta})}^{ 2}e^{ - \chi(\gd)s} + 4\kappa_4\delta  \left\vert m_{ 0}^{ 1} - m_{ 0}^{ 2} \right\vert^{ 2}\right)\, ,
\end{equation}
and applying once again \eqref{eq:tilde_f_Lip}, we obtain
\begin{align}
\left\Vert p_{ s}^{ 1} - p_{ s}^{ 2} \right\Vert_{ L^{ 2}(w_{  \beta})}^{ 2} \, \leq\,  \left\vert m_{ 0}^{ 1} - m_{ 0}^{ 2} \right\vert^{ 2}\left(e^{ - \chi(\gd) s} +  4\kappa_{ 4}\delta\right).
\end{align}
Hence, using the lower bound in \eqref{eq:encadre m}, and choosing the value of $\gd_7$ small enough so that $e^{-\chi(\gd_{7})T_{ 0}}\leq \frac18$ (recall \eqref{eq:def chi} and \eqref{eq:def T}) and $4\kappa_4\gd_7\leq \frac{1}{8}$, we get:
\begin{equation}
\left\Vert{ \eta}_{ t, f}(m_{ 1}) - { \eta}_{ t, f}(m_{ 2}) \right\Vert_{ L^{ 2}(w_{  \beta})}^2\, \leq\,  4 \left\vert m_{ t}^{ 1} - m_{ t}^{ 2} \right\vert^{ 2}\left(e^{ - \chi(\gd)T_{ 0}} + 4\kappa_{ 4}\delta\right)
\,\leq\,  \left\vert m_{ 1} - m_{ 2} \right\vert^{ 2}\, ,
\end{equation}
so that Lemma~\ref{lem:etat_Lip} is proven.
\end{proof}
\begin{lemma}
\label{lem:eta_contraction}
There exists $\delta_{ 8}>0$ and a constant $c_2>0$ such that if $0\leq \delta\leq\delta_{8}$, then for $f_1,f_2\in  \mathcal{ H}$ and all $t\in [T_{ 0},2T_{ 0}]$ we have
\begin{equation}\label{eq:lem contract}
 \left\Vert \eta_{t,f_1}-\eta_{t,f_2} \right\Vert_{ \infty, \beta}\,\leq\, e^{-\left( k_{\beta}-c_2 \gd\right) t}  \left\Vert f_1- f_2 \right\Vert_{ \infty, \beta}\, \leq\, \frac12 \left\Vert f_1- f_2 \right\Vert_{ \infty,\beta}\, .
\end{equation}
\end{lemma}
\begin{proof}[Proof of Lemma~\ref{lem:eta_contraction}]
We suppose that $0\leq \delta\leq \delta_{7}$. Fix a $ m^{ \prime}\in B_{ r}$ and a $t\in [T_{ 0},2T_{ 0}]$. According to Lemma~\ref{lem:bij}, for $i=1, 2$, there exists a unique $m_{ 0}^{ i}\in B_{ r}$ such that $ m^{ \prime} = g_{ t, f_{ i}}(m_{ 0}^{ i})$. For $i=1, 2$, denote by $p^i:= \left(p^{f_{ i}(m^i_0),m^i_0}_s\right)_{ s\geq0}$ and $m^i:= \left(m^{f_{ i}(m^i_0),m^i_0}_s\right)_{ s\geq0}$. Note that by construction $m^i_t=m'$ (see Figure~\ref{fig:traj_mi}). We consider also the solutions $p^3:= \left(p^{f_1(m^2_0),m^2_0}_s\right)_{ s\geq0}$ and $m^3:= \left(m^{f_1(m^2_0),m^2_0}_s\right)_{ s\geq0}$.
\begin{figure}[h]
\centering
\includegraphics[width=0.6\textwidth]{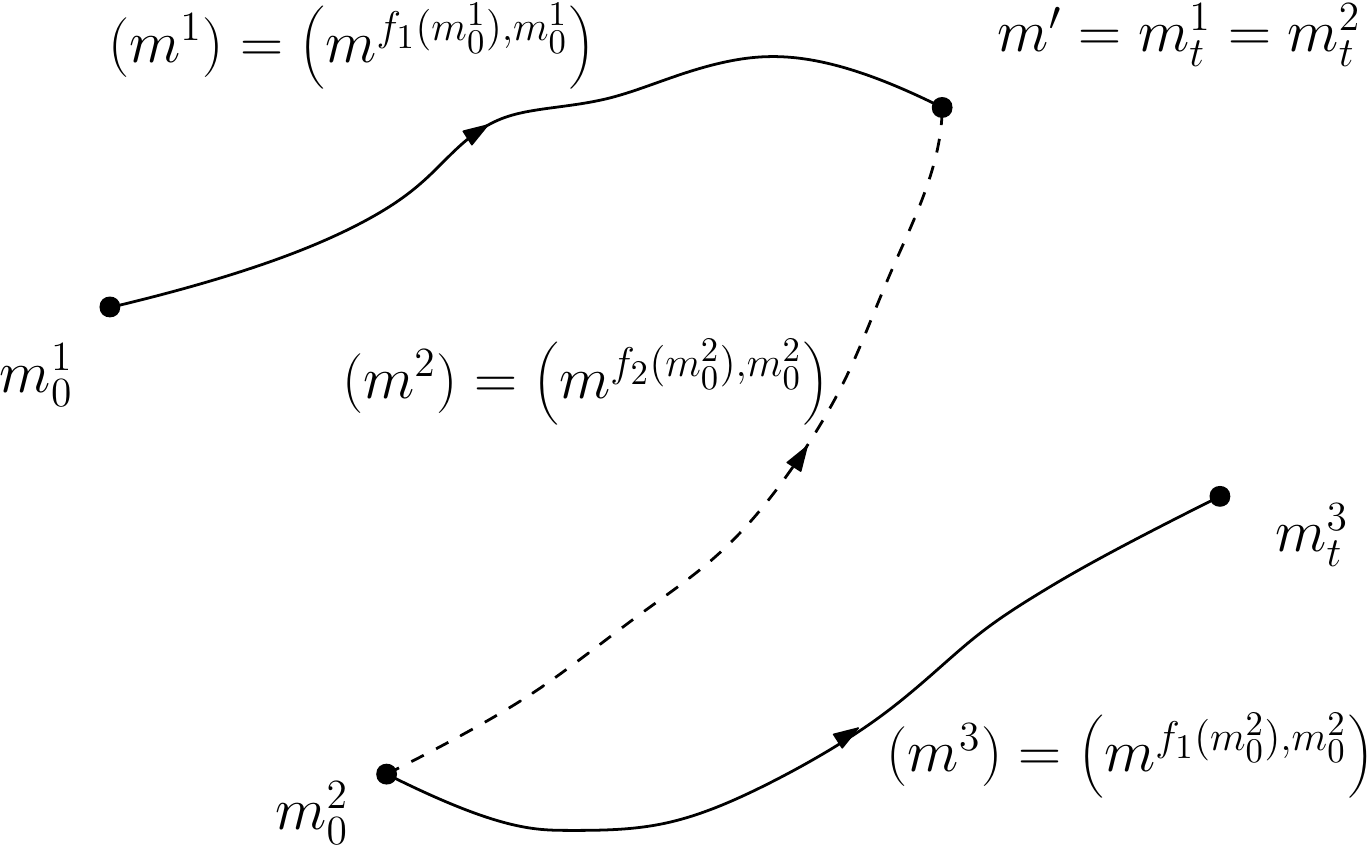}
\caption{Trajectories of the mean values $m^{ i}$ for $i=1, 2, 3$. The solid lines (resp. dashed line) correspond to trajectories of the mean value $(m_{ s}^{ f_{ 1}(\cdot), \cdot})_{ s\geq0}$ (resp. $(m_{ s}^{ f_{ 2}(\cdot), \cdot})_{ s\geq0}$) with initial conditions provided by $f_{ 1}$ (resp. $f_{ 2}$).}
\label{fig:traj_mi}
\end{figure}
With these notations at hand,
\begin{multline}
\left\Vert { \eta_{ t, f_{ 1}}}(m^{ \prime})-{ \eta_{ t, f_{ 2}}}(m^{ \prime})\right\Vert_{ L^{ 2}(w_{  \beta})} \, =\, \left\Vert p_{ t}^{ 1}- p_{ t}^{ 2} \right\Vert_{ L^{ 2}(w_{ \beta})}
\\ \leq \, \left\Vert p_{t}^{ 1} - p_{ t}^{ 3} \right\Vert_{ L^{ 2}(w_{  \beta})} + \left\Vert p_{t}^{ 3}-p_{t}^{ 2} \right\Vert_{ L^{ 2}(w_{ \beta})}\, .
\end{multline}
Firstly we have
\begin{equation}
\left\Vert p_{t}^{ 1}- p_{t}^{ 3} \right\Vert_{ L^{ 2}(w_{  \beta})}\, =\,  \left\Vert { \eta_{ t, f_{ 1}}}(m_{ t}^{ 1}) -{ \eta_{ t, f_{ 1}}}(m_{ t}^{ 3})\right\Vert_{ L^{ 2}(w_{ \beta})}\, ,
\end{equation}
and so, by Lemma~\ref{lem:etat_Lip}, we obtain
\begin{equation}
\left\Vert p_{ t}^{ 1}- p_{ t}^{ 3} \right\Vert_{ L^{ 2}(w_{ \beta})} \, \leq\ \left\vert m_{t}^{ 1}-m_{t}^{ 3} \right\vert\, = \, \left\vert m_{t}^{ 2} - m_{t}^{ 3} \right\vert\, ,
\end{equation}
where we have also used the identity $m_{t}^{ 1}= m^{ \prime} = m_{t}^{ 2}$. Note that, by construction, $m_{ 0}^{ 3}= m_{ 0}^{ 2}$. In particular, the same calculations leading to \eqref{eq:bound diff m_1} give the estimate, for $s\leq t$ and some constant $C>0$,
\begin{align*}
\left\vert m_{ s}^{ 3} - m_{ s}^{ 2} \right\vert &\leq \delta C\int_{0}^{s} \left( \left\Vert p_{ 0}^{ 3}- p_{ 0}^{ 2} \right\Vert_{ L^{ 2}(w_{  \beta})} + \sup_{ v\leq u} \left\vert m_{ v}^{ 3} - m_{ v}^{ 2} \right\vert\right){\rm d}u,
\end{align*}
so that Gr\"onwall's lemma implies, 
\begin{equation}
\label{eq:gronwall_m23}
\sup_{ u\leq s} \left\vert m_{ u}^{ 3} - m_{ u}^{ 2} \right\vert \leq \delta C \left\Vert p_{ 0}^{ 3} - p_{ 0}^{ 2} \right\Vert_{ L^{ 2}(w_{ \beta})} e^{ \delta C s}\, .
\end{equation}
Note also that 
\begin{equation}\label{eq:link p0 f}
\left\Vert p_{ 0}^{ 3}- p_{ 0}^{ 2} \right\Vert_{ L^{ 2}(w_{ \beta})}\, =\,  \left\Vert { f}_{ 1}(m_{ 0}^{ 2}) - { f}_{ 2}(m_{ 0}^{ 2}) \right\Vert_{ L^{ 2}(w_{ \beta})}\, \leq\, \left\Vert f_{ 1} - f_{ 2}\right\Vert_{ \infty,\beta}\, .
\end{equation}
This means that
\begin{equation}
\label{eq:control_p13_final}
\left\Vert p_{t}^{ 1}-p_{t}^{ 3} \right\Vert_{ L^{ 2}(w_{  \beta})} \, \leq\, \gd C  e^{ \gd C 2T} \left\Vert f_{ 1}- f_{ 2} \right\Vert_{ \infty, \beta}\, .
\end{equation}
Secondly, applying Lemma \ref{lem:lip_cont_pi}, we obtain
\begin{equation}
\left\Vert p_{t}^{ 3} - p_{t}^{ 2} \right\Vert_{ L^{ 2}(w_{  \beta})}^{ 2} \, \leq\,   \left\Vert p_{ 0}^{ 3}-p_{ 0}^{ 2} \right\Vert_{ L^{ 2}(w_{\beta})}^{ 2} e^{ - \chi(\gd)t} +  { \kappa}_{ 4}\delta  \sup_{ s\leq t} \left\vert m_{ s}^{ 3} - m_{ s}^{ 2} \right\vert^{ 2}\, , \label{eq:control_p23}
\end{equation}
so using again \eqref{eq:gronwall_m23} and \eqref{eq:link p0 f}, we obtain,
\begin{equation}
\left\Vert p_{t}^{ 3} - p_{t}^{ 2} \right\Vert_{ L^{ 2}(w_{  \beta})}^{ 2} \, \leq\,  \left( e^{ - \chi(\gd) t} +  { \kappa}_{ 4}\delta^{ 3} C^{ 2} e^{ 4\delta C T_{ 0}}\right) \left\Vert f_{ 1} - f_{ 2}\right\Vert_{ \infty, \beta}^{ 2}\,  .
\end{equation}
So, recalling the definition \eqref{eq:def chi} of $\chi(\gd)$, for $\gd$ small enough we indeed have for a positive constant $c_2$:
\begin{equation}
 \left\Vert \eta_{t,f_1}-\eta_{t,f_2} \right\Vert_{ \infty, \beta}\,\leq\, e^{-\left( k_{\beta}-c_2 \gd\right) t}  \left\Vert f_1- f_2 \right\Vert_{ \infty, \beta}\, ,
\end{equation}
and we obtain the result, choosing $\gd_8\leq \gd_7$ small enough (recall \eqref{eq:def T}).
\end{proof}

\subsection{Proof of Theorem \ref{th:main1}}

Let us denote for $t\geq 0$ the mapping $\Pi_t:f\mapsto \eta_{t,f}$. Lemma \ref{lem:etat_Lip} shows that $\Pi_t(\cH)\subset \cH$ for $t\in [T_{ 0},2T_{ 0}]$, and Lemma~\ref{lem:eta_contraction} implies that $\Pi_{ T_{ 0}}$ admits a unique fixed-point $f_0$ in $\cH$. Our first aim is to show that $\Pi_t(f_0)=f_0$ for all $t\geq 0$. The semi-group property implies directly that $\Pi_{kT_{ 0}}(f_0)=f_0$ for $k\in \bbN$. It remains to show that $\Pi_t(f_0)=f_0$ for $t\in (0,T_{ 0})$. But for such a $t$, we have $\Pi_t f_0=\Pi_t\Pi_{ T_{ 0}} f_0=\Pi_{T_{ 0}+t} f_0 \in \cH$, by Lemma \ref{lem:etat_Lip}, and $\Pi_{T_{ 0}}\Pi_t f_0=\Pi_t\Pi_{ T_{ 0}} f_0=f_0$, so $\Pi_t f_0=f_0$ by uniqueness of the fixed-point of $\Pi_{ T_{ 0}}$ on $\cH$.

\medskip
Since we have modified the dynamics of \eqref{eq:slow fast PDE} only for $m_t$ strictly outside of $B_L$, to prove that the manifold $ \mathcal{ M}^\gd$ is positively invariant for \eqref{eq:slow fast PDE} it remains only to show that, for a trajectory of \eqref{eq:slow fast PDE} starting from $(p_0,m_0)$ with $m_0\in \cV$, $m_t$ can not leave $\cV$. But if $m_0 \in \partial \cV$ we get
\begin{multline}
\frac{1}{\gd} n_{\partial\cV}(m_0)\cdot \dot m_0\, =\,   n_{\partial\cV}(m_0) \cdot \int_{\bbR^d} F(x) q^\gd_{m_0}(x) \dd x\\
=\, n_{\partial\cV}(m_0)\cdot \int_{\bbR^d} F(m_0+x) q_0(x)\dd x
+  n_{\partial\cV}(m_0) \cdot \int_{\bbR^d} F(x+m_0)\big(f_{ 0}(m_0)(x)-q_0(x)\big)\dd x\, .\label{eq:dot_m0}
\end{multline}
It remains to remark that \eqref{hyp:bound_F_exp} and the fact that $f_0\in \cH$ imply, by Cauchy-Schwartz inequality, that for some $C>0$
\begin{equation}
\left\vert  \int_{\bbR^d} F(x+m_0)\big(f_{ 0}(m_0)(x)-q_0(x)\big)\dd x\right\vert\, \leq\, C\gd\, .
\end{equation}
Combining the point (5) of Hypothesis \ref{hyp F} and the previous estimate shows that $n_{\partial\cV}(m_0)\cdot \dot m_0<0$ if we take $\gd$ small enough. This means precisely that $ \mathcal{ M}^\gd$ is positively invariant. This concludes the proof of Theorem \ref{th:main1}.

\medskip

Remark \ref{rem:stab} follows from \eqref{eq:lem contract}: any $p_0\in \cG(m_0)$ with $\vert m_0\vert \in \cV$ can be seen as a $f(m_0)$ for a $f\in \cH$, and with calculations similar as the ones we have just made we can show that the associated trajectory satisfies $\vert m_t\vert \in \cV$, and \eqref{eq:contract rem} follows from \eqref{eq:lem contract} (the constant $C'$ is necessary due to the fact that \eqref{eq:lem contract} provides informations only for $t\geq T$, for $t$ smaller we rely on \eqref{eq:bound ddt p-q}). Note that the third inequality of \eqref{eq:def H 2} is not needed to prove \eqref{eq:lem contract}.

\section{$ \mathcal{ C}^1$-regularity and approximated phase dynamics}
\label{sec:regularity_manifold}
The purpose of this section is to prove that the invariant manifold that we have found in the previous section is in fact $ \mathcal{ C}^{ 1}$. Following the approach of \cite{wiggins2013normally}, Section~3.3.1, the point is to first establish a formal equation that the derivative should satisfy, then to prove that this equation has a fixed-point and third, to show that this solution is indeed the derivative that we sought. After some preliminary estimates, we carry out this program in Section~\ref{sec:some_preliminary_def}.
\subsection{Linearized equation}
Consider the trajectories $p_t=p^{p_0,m_0}_t$ and $m_t=m^{p_0,m_0}_t$ for a $m_0\in \bbR^d$ and a $p_0\in L^2(w_{\ga})$ (recall \eqref{eq:slow fast PDE modified}).
Then for any initial condition $(y_0,n_0)$ that satisfies
\begin{equation}\label{eq:init cond y}
\int_{\bbR^d} y_0(x)\dd x=0\, ,\qquad \text{and}\qquad \int_{\bbR^d}x y_0(x)\dd x=0\, ,
\end{equation}
we consider the process $(y,n)$ defined by the following couple of equations:
\begin{equation}\label{eq:y_t n_t}
\left\{
\begin{array}{ccl}
\partial_t y_t & = &  \cL y_t +\nabla\cdot((\dot m_t-\gd F_{ t} -  \gd h(m_t))y_t)\\
&& \qquad \qquad \qquad \qquad \qquad  +  \nabla \cdot \left( (\dot n_t-\gd DF_t[n_t] -\gd Dh(m_t)[n_t])p_t\right)\\
\dot n_t & = & \gd \int F_t y_t +\gd Dh(m_t)[n_t]+ \gd \int DF_t[n_t]p_t
\end{array}
\right. \, ,
\end{equation}
where we have used the notation $DF_t(x)[u]=DF(x+m_t)[u]$.

\medskip

We first state the following existence and uniqueness result (recall the definition of $\gamma$ and $\gamma'$ in \eqref{hyp:alpha_beta_gamma}).

\begin{lemma}
For any $T>0$, $(p_0,m_0)\in L^2(w_{\ga})\times \bbR^d$ and any initial condition $(y_0,n_0)\in L^2(w_{\gamma})\times \bbR^d$ satisfying the hypotheses \eqref{eq:init cond y} there exists a unique couple $(y,n)$ element of $L^2((0,T),H^1(w_{\gamma}))\times \mathcal{ C}^1([0,T],\bbR^d)$ with $\frac{\dd}{\dd t} y\in L^2((0,T),H^{-1}(w_{\gamma'}))$ solution of \eqref{eq:y_t n_t}.
\end{lemma}

\begin{proof}
We rely here on results already obtained in the proof of Proposition~\ref{lem:regularity_p}. Let us first fix a trajectory $\tilde n \in \mathcal{ C}^1([0,T],\bbR^d)$ with initial condition $\tilde n_0=n_0$. 
Since $p\in \mathcal{ C}([0,T], L^2(w_{\ga'})$ (see the proof of Proposition~\ref{lem:regularity_p}) and by \eqref{hyp:bound_F_exp}, we deduce that $\tilde d_t:= \nabla \cdot \left( (\dot{ \tilde n_t}-\gd DF_t[\tilde n_t] -\gd Dh(m_t)[\tilde n_t])p_t\right)$ belongs to $ \mathcal{ C}((0,T),H^{-1}(w_{\gamma}))$. Thus, remarking that the estimates \eqref{eq:bound Q L2 H1} and \eqref{eq:Q coerc} are valid with $\ga$ replaced by $\gamma$, up to a redefinition of the involved constants, Lions Theorem ensures the existence of a unique $y^R\in H^1((0,T),L^2(w_{\gamma})$ with $\frac{\dd}{\dd t} y^R\in L^2((0,T),H^1(w_{\gamma})$ solution of
\begin{equation}
\left\langle \frac{\dd}{\dd t} y^R_t,v\right\rangle_{L^2(w_{\gamma})}+\left\langle Q_{t,R}\, y^R_t,v\right\rangle_{L^2(w_{\gamma})}\, =\, \left\langle \tilde d_t,v\right\rangle_{L^2(w_{\gamma})}\, ,
\end{equation} 
for all $\nu\in H^1(w_{\gamma})$ and almost all $t\in (0,T)$. Moreover a straightforward calculation leads to, for a constant $C>0$ and almost all $t\in [0,T]$,
\begin{equation}
\frac{1}{2} \frac{\dd}{\dd t} \Vert y^R_t\Vert_{L^2(w_{\gamma})}^2\, \leq\, C\left( \Vert y^R_0\Vert_{L^2(w_{\gamma})}^2+\sup_{s\in [0,T]} \Vert \tilde d_s\Vert_{H^{-1}(w_{\gamma})}^2\right)\, ,
\end{equation}
from which we obtain, by Gr\"onwall's inequality,
\begin{equation}\label{eq:gron yR}
\Vert y^R_t\Vert_{L^2(w_{\gamma})}^2 \, \leq\,  \left( \Vert y^R_0\Vert_{L^2(w_{\gamma})}^2+\sup_{s\in [0,T]} \Vert \tilde d_s\Vert_{H^{-1}(w_{\gamma})}^2\right)e^{Ct}\, ,
\end{equation}
and using again \eqref{eq:Q coerc} (with $\ga$ replaced by $\gamma$),
\begin{equation}
\Vert y^R\Vert_{L^2((0,T),L^2(w_{\gamma}))}^2\, \leq\, C(T)\left(\Vert y_0\Vert_{L^2(w_{\gamma})}^2+ \sup_{s\in [0,T]} \Vert \tilde d_s\Vert_{H^{-1}(w_{\gamma})}^2\right)\, .
\end{equation}
Following the same steps as in the proof of Lemma \ref{lem:regularity_p}, we also get
\begin{equation}
\left\Vert \frac{\dd}{\dd t}y^R\right\Vert_{L^2((0,T),L^2(w_{\gamma'}))}^2\, \leq\, C'(T)\left(\Vert y_0\Vert_{L^2(w_{\gamma})}^2+ \sup_{s\in [0,T]} \Vert \tilde d_s\Vert_{H^{-1}(w_{\gamma})}^2\right)\, ,
\end{equation}
and thus, with $R\rightarrow \infty$, the Banach Alaoglu Theorem gives the existence of a $y\in L^2((0,T),H^1(w_{\gamma}))$ with $\frac{\dd}{\dd t} y\in L^2((0,T),H^{-1}(w_{\gamma'}))$ that satisfies (using in particular \eqref{eq:bound Q H1} to pass to the limit, with $\ga$ and $\ga'$ replaced by $\gamma$ and $\gamma'$ respectively)
\begin{equation}\label{eq:sol y tilde f}
\left\langle \frac{\dd}{\dd t} y_t,v\right\rangle_{L^2(w_{\gamma'})}+\left\langle Q_{t,\infty}\, y_t,v\right\rangle_{L^2(w_{\gamma'})}\, =\, \left\langle \tilde d_t,v\right\rangle_{L^2(w_{\gamma'})}\, ,
\end{equation}
for all $v\in H^1(w_{\gamma'})$ and almost all $t\in (0,T)$. The uniqueness for each such $y$ (for each given $\tilde n$) follows by Gr\"onwall's inequality as above, since \eqref{eq:Q coerc} (with $\ga$ replaced by $\gamma'$) does not depend on $R$.

\medskip

Denote now, for each $\tilde n\in \mathcal{ C}^1([0,T],\bbR^d)$, by $\Psi(\tilde n)$ the solution $n$ of the equation
\begin{equation}
\dot n_t \, = \, \gd \int F_t y_t +\gd Dh(m_t)[n_t]+ \gd \int DF_t[n_t]p_t\, ,
\end{equation}
where $y$ is the solution of \eqref{eq:sol y tilde f} with initial condition $y_0$. For two trajectories $\tilde n^1$ and $\tilde n^2$ and the two associated solution $y^1$ and $y^2$, with similar arguments as the ones leading to \eqref{eq:gron yR} we get
\begin{equation}
\Vert y^1_t-y^2_t\Vert_{L^2(w_{\gamma})}\, \leq\, C(T) \int_0^t\Vert \tilde d^1_s-\tilde d^2_s\Vert_{H^{-1}(w_{\gamma})}\dd s
\, \leq\, C'(T)\int_0^t\Vert \tilde n^1-\tilde n^2\Vert_{ \mathcal{ C}^1([0,t],\bbR^d)}\dd s\, .
\end{equation}
So, using the fact that $\sup_{t\in[0,T]} \int_{\bbR^d}\vert F_t\vert^2w_{-\gamma}<\infty$ and
$\sup_{t\in[0,T]}\int_{\bbR^d} |F_t|p_t <\infty$, we obtain, relying again on Gr\"onwall's inequality,
\begin{equation}
\Vert \Psi(\tilde n^1)-\Psi(\tilde n^2)\Vert_{ \mathcal{ C}^1([0,T],\bbR^d)}\, \leq\, C''(T)\int_0^T \Vert \tilde n^1-\tilde n^2\Vert_{ \mathcal{ C}^1([0,t],\bbR^d)}\dd t\, .
\end{equation}
By induction
\begin{equation}
\Vert \Psi^{k+1 }(\tilde n)-\Psi^k(\tilde n)\Vert_{ \mathcal{ C}^1([0,T],\bbR^d)}\, \leq\, \big(C''(T)\big)^k\frac{T^k}{k!}\Vert \Psi(\tilde n)-\tilde n\Vert_{ \mathcal{ C}^1([0,T],\bbR^d)}\, ,
\end{equation}
which means that $\Psi^k(\tilde n)$ is a Cauchy sequence, with limit $n$ which is the unique fixed-point of $\Psi$. This completes the proof.
\end{proof}

We will use in the following the notations $y^{p_0,m_0}_t[y_0,n_0]$ and  $n^{p_0,m_0}_t[y_0,n_0]$ to emphasize the dependency on the initial conditions of the solutions to \eqref{eq:y_t n_t} and we denote by $f_0$ the fixed-point of the mapping $f\mapsto \eta_{t,f}$ obtained in the proof of Theorem~\ref{th:main1}. Recall the definition of $T_{ 0}$ in \eqref{eq:def T}.

We give now a first regularity result for the trajectories of \eqref{eq:PDE_mod_h} with respect to the initial mean $m_0$.
\begin{lemma}\label{lem:diff p}
There exist positive constants $\gd_9>0$ and $\kappa_9>0$ such that for all $0\leq \gd\leq \gd_9$, $m^{ i}_0 \in B_{2L}$ ($i=1,2$) and $p^i_0\in \cG(m^i_0)$ such that $\Vert p^2_0-p^1_0\Vert_{L^2(w_{\beta})}\leq \vert m^2_0-m^1_0\vert$, we have for all $t\in [0,2T_{ 0}]$, denoting $p^i_t=p^{p^i_0,m^i_0}_t$ and  $m^i_t=m^{p^i_0,m^i_0}_t$,
\begin{equation}
\left\Vert p^2_t-p^1_t-y^{p^1_0,m^1_0}_t[p^2_0-p^1_0,m^2_0-m^1_0]\right\Vert_{L^2(w_{\gamma})}\, \leq\, \kappa_{9} \vert m^2_0-m^1_0\vert^2\, ,
\end{equation}
and
\begin{equation}
\left\Vert m^2_t-m^1_t-n^{p^1_0,m^1_0}_t[p^2_0-p^1_0,m^2_0-m^1_0]\right\Vert_{L^2(w_{\gamma})}\, \leq\, \kappa_{9} \vert m'_0-m_0\vert^2\, .
\end{equation}
\end{lemma}

\begin{proof}[Proof of Lemma \ref{lem:diff p}]
For $y_t:=y^{p^1_0,m^1_0}_t[p^2_0-p^1_0,m^2_0-m^1_0]$, $n_t:=n^{p^1_0,m^1_0}_t[p^2_0-p^1_0,m^2_0-m^1_0]$, $u_t := p^2_t-p^1_t-y_t$ and $s_t := m^2_t-m^1_t-n_t$, we have
\begin{align}
\frac{1}{2} \frac{\dd}{\dd t}\Vert u_t\Vert^2_{L^2(w_{\gamma})}
\,  =&  \, \langle \cL u_t,u_t\rangle_{L^2(w_{\gamma})}
+\langle \nabla\cdot ((\dot m^1_t -\gd F^1_t-\gd h(m^1_t) )u_t),u_t\rangle_{L^2(w_{\gamma})} \nonumber\\
& +(\dot s_t - \gd (h(m^2_t)-h(m^1_t)-Dh(m^1_t)[n_t]))\cdot \langle \nabla p^1_t,u_t\rangle_{L^2(w_{\gamma})} \nonumber\\
& \qquad \qquad-\gd \langle \nabla\cdot ((F^2_t-F^1_t-DF^1_t[n_t])p^1_t),u_t\rangle_{L^2(w_{\gamma})}\nonumber\\
& +(\dot m^2_t-\dot m^1_t -\gd(h(m^2_t)-h(m^1_t))) \cdot \langle \nabla (p^2_t-p^1_t),u_t\rangle_{L^2(w_{\gamma})} \nonumber\\
&\qquad \qquad  -\gd \langle \nabla\cdot ((F^2_t-F^1_t)(p^2_t-p^1_t)),u_t\rangle_{L^2(w_{\gamma})}\, ,
\end{align}
where we have used the notations $F^i_t(x)=F(x+m^i_t)$ and $DF^1_t[n](x)=DF(x+m^1_t)[n]$.
Using similar arguments as in the previous sections (see in particular \eqref{eq:by part Ft}), we have for some $C>0$:
\begin{equation}
\langle \nabla\cdot ((\dot m^1_t -\gd F^1_t-\gd h(m^1_t) )u_t),u_t\rangle_{L^2(w_{\gamma})}\, \leq\, C\gd  \Vert u_t\Vert^2_{H^1(w_{\gamma})}\, .
\end{equation} 
The regularity of $h$ implies that
\begin{equation}
\vert h(m^2_t)-h(m^1_t)-Dh(m^1_t)[n_t]|\, \leq\, C_h(\vert s_t\vert+\vert m^2_t-m^1_t\vert^2)\, .
\end{equation}
Moreover, we have
\begin{align}
\big\Vert |F^2_t-&F^1_t-DF^1_t[n_t]|p^1_t\big\Vert^2_{L^2(w_{\gamma})}\\
& =\, \int_{\bbR^d}\vert F(x+m^2_t)-F(x+m^1_t)-DF(x+m^1_t)[n_t]\vert^2 (p^1_t(x))^2w_{\gamma}(x) \dd x\nonumber\\
& \leq\, 2\int_{\bbR^d}\vert F(x+m^2_t)-F(x+m^1_t)-DF(x+m^1_t)[m^2_t-m^1_t]\vert^2( p^1_t(x))^2 w_{\gamma}(x)\dd x \nonumber\\
& \qquad \qquad + \int_{\bbR^d}\vert DF(x+m^1_t)[s_t]\vert^2 (p^1_t(x))^2 w_{\gamma}(x)\dd x\, .
\end{align}
On one hand, by \eqref{hyp:bound_F_exp} and Lemma \ref{lem:mu bounded L2}, we get for some constant $C_{F,L}>0$,
\begin{align}
\int_{\bbR^d}\vert F(x+m^2_t)-&F(x+m^1_t)-DF(x+m^1_t)[m^2_t-m^1_t]\vert^2 (p^1_t(x))^2 w_{\gamma}(x) \dd x \nonumber\\
& \leq\, \vert m^2_t-m^1_t\vert^4 \int_{\bbR^d} \sup_{z\in [x+m_t, x+m'_t]} \vert D^2 F(z) \vert^2 (p^1_t(x))^2 w_{\gamma}(x)\dd x \nonumber\\
& \leq\, C_F\vert m^2_t-m^1_t\vert^4 \int_{\bbR^d} \left(w_{2\gep}(x+m^2_t)+w_{2\gep}(x+m^1_t)\right) (p^1_t(x))^2w_{\gamma}(x) \dd x \nonumber\\
&\leq\, C_{F,L} \vert m^2_t-m^1_t\vert^4 \Vert p^1_t\Vert_{L^2(w_{\ga})}^2\, \leq\, C_{F,L}\kappa_2^2 \vert m^2_0-m^1_0\vert^4\, ,
\end{align}
and on the other hand, with similar arguments, for some constant $C_{F,L}'>0$, we obtain
\begin{equation}
 \int_{\bbR^d}\vert DF(x+m^1_t)[s_t]\vert^2 p_t(x)^2 \dd x\, \leq\, C_{F,L}' \vert s_t\vert^2 \Vert p_t\Vert_{L^2(w_{\ga})}^2\, .
\end{equation}
Recalling now Lemma \ref{lem:lip_cont_pi}, Lemma \ref{lem:encadre m}, \eqref{eq:bound diff m_1} and using the fact that $\Vert p^2_0-p^1_0\Vert_{L^2(w_{\beta})}\leq \vert m^2_0-m^1_0\vert$, we get for some constant $C_{F,L}''>0$
\begin{equation}
\vert \dot m^2_t-\dot m^1_t -\gd(h(m^2_t)-h(m^1_t))\vert \leq\, C_{F,L}''\vert m^2_0-m^1_0\vert\, .
\end{equation}
Finally, using similar arguments (remark here that the choice of the weight $w_{\gamma}$ allows us to use hypothesis \eqref{hyp:bound_F_exp} and Lemma \ref{lem:lip_cont_pi} successively), with some constant $C_{F,L}'''>0$,
\begin{multline}
\int_{\bbR^d} \vert F^2_t-F^1_t\vert^2 (p^2_t-p^1_t)^2w_{\gamma}\, \leq\, \vert m^2_t-m^1_t\vert^2 \int_{\bbR^d} \sup_{z\in [x+m^1_t,x+m^2_t]}\vert DF(x)\vert^2 (p^2_t-p^1_t)^2 w_{\gamma}\\
\leq\,C_{F,L}''' \vert m^2_t-m^1_t\vert^2 \Vert p^2_t-p^1_t\Vert_{L^2(w_{\beta})}^2\, \leq\, C_{F,L}'''\kappa_4^2 \vert m^2_0-m^1_0\vert^4\, .
\end{multline}
We deduce, recalling Lemma \ref{lem:bound L OU weighted L2}, that for some constant $C'>0$,
\begin{multline}
\frac{1}{2} \frac{\dd}{\dd t}\Vert u_t\Vert^2_{L^2(w_{\gamma})}\, \leq\, -\frac{k_{\gamma}}{\gamma(\Tr(K)+ \underline{ k})}\Vert \nabla u\Vert_{L^2(w_{\gamma})}^2\\
+C'\Big(\vert \dot s_t\vert +(\vert m'_0-m_0\vert^2+\vert s_t\vert) \Big)\Vert \nabla u\Vert_{L^2(w_{\gamma})}\, ,
\end{multline}
and thus, using Young's inequality, there exists a constant $c_1$ such that
\begin{equation}
 \frac{\dd}{\dd t}\Vert u_t\Vert^2_{L^2(w_{\gamma})}\, \leq\,c_1(\vert \dot s_t\vert^2+\vert m'_0-m_0\vert^4+\vert s_t\vert^2) \left\Vert u_{ t} \right\Vert^{ 2}_{ L^{ 2}(w_{ \gamma})}\, .
\end{equation}
Now $s_t$ satisfies
\begin{multline}
\dot s_t\, =\, \gd(h(m^2_t)-h(m^1_t)-Dh(m^1_t)[s_t])\\+\gd\left(\int_{\bbR^d}F^2_t p^2_t-\int_{\bbR^d}F^1_tp^1_t-\int_{\bbR^d} DF^1_t[n_t]p^1_t-\int_{\bbR^d} F^1_t y_t\right)\, ,
\end{multline}
and, using the same arguments as before, remarking in particular that
\begin{multline}
\int_{\bbR^d}F^2_t p^2_t-\int_{\bbR^d}F^1_tp^1_t-\int_{\bbR^d} DF^1_t[n_t]p^1_t-\int_{\bbR^d} F^1_t y_t\,
=\, \int_{\bbR } F^1_t u_t +\int_{\bbR^d} (F^2_t)-F^1_t)(p^2_t-p^1_t) \\+\int_{\bbR^d}(F^2_t-F^1_t-DF^1_t[n_t])p^1_t\, ,
\end{multline}
we obtain, for some constant $c_2$,
\begin{equation}
\vert \dot s_t \vert\, \leq\, c_2(\vert s_t\vert+ \Vert u_t\Vert_{L^2(w_{\gamma})} +\vert m^2_0-m^1_0\vert^2)\, .
\end{equation}
The result follows, via the application of Gr\"onwall's lemma to $\Vert u_t\Vert_{L^2(w_{ \gamma})}^2 + \vert s_t\vert^2$.
\end{proof}

\begin{lemma}
\label{lem:contraction_y}
There exists positive constants $\gd_{ 10}>0$ and $\kappa_{ 10}>0$ such that for all $0\leq \gd\leq \gd_{ 10}$, $m_0 \in B_{2L}$, $p_0\in \cG(m_0)$, $y_0\in L^2(w_{\gamma})$ and $n_0\in \bbR^d$ we have for all $t\in [T_{ 0},2T_{ 0}]$,
\begin{equation}
\left\Vert y^{p_0,m_0}_{ t}[y_0,n_0] \right\Vert_{ L^{ 2}(w_{ \gamma})}\, \leq\, \left\Vert y_{ 0} \right\Vert_{ L^{ 2}(w_{ \gamma})} e^{ - \chi'(\gd)t} + \kappa_{10}\gd \left( \left\vert n_0 \right\vert^{ 2} + \left\Vert y_{ 0} \right\Vert_{ L^{ 2}(w_{ \gamma})}^{ 2}\right)^{ \frac{ 1}{ 2}}\, ,
\end{equation}
where, for some constant $c_2$ depending on $F$, $L$, $h$ and $T_{ 0}$, $\chi'(\gd)= k_{ \gamma} -c_2 \delta $ and
\begin{equation}
\left\vert n^{p_0,m}_t[y_0,n_0]-n_0\right\vert \, \leq\, \kappa_{10}\gd \left(\vert n_0\vert +\Vert y_0\Vert_{L^2(w_{\gamma})}\right)\, .
\end{equation}
\end{lemma}

\begin{proof}[Proof of Lemma~\ref{lem:contraction_y}]
Let us denote $y_t=y^{p_0,m_0}_t[y_0,n_0]$. A simple calculation leads to
\begin{multline}
\frac{ 1}{ 2}\frac{ {\rm d}}{ {\rm d}t} \left\Vert y_{ t} \right\Vert_{ L^{ 2}(w_{ \gamma})}^{ 2} \, =\,\left\langle \cL y_{ t}\, ,\, y_{ t}\right\rangle_{ L^{ 2}(w_{ \gamma})}+ (\dot m_t-\gd h(m_t))\cdot \langle \nabla y_t,y_t\rangle_{L^2(w_{\gamma})}
\\
-\gd  \langle F_t\cdot \nabla y_t,y_t\rangle_{L^2(w_{\gamma})}
+(\dot n_t -\gd Dh(m_t)[n_t]) \cdot \langle \nabla p_t,y_t\rangle_{L^2(w_{\gamma})}\\
-\gd\langle \nabla\cdot (DF_t[n_t] p_t),y_t\rangle_{L^2(w_{\gamma})}\, .\label{eq:diff_xit_beta}
\end{multline}
Using similar arguments as in the previous proofs (relying in particular on Lemma \ref{lem:bound L OU weighted L2}, Lemma \ref{lem:mu bounded L2} and Lemma \ref{lem:nablamu_bounded_L2}), we get for some positive constants $c$ and $C$, 
\begin{equation}
\frac{ 1}{ 2}\frac{ {\rm d}}{ {\rm d}t} \left\Vert y_{ t} \right\Vert_{ L^{ 2}(w_{  \gamma})}^{ 2} \leq - \left(k_{ \gamma} - c\gd \right) \left\Vert y_{ t} \right\Vert_{ L^{ 2}(w_{  \gamma})}^{ 2} + C  \left( \vert \dot n_t \vert +\gd \vert n_t\vert \right)\left\Vert y_{ t} \right\Vert_{ L^{ 2}(w_{  \gamma})} \, . \label{eq:yt_beta_1}
\end{equation}
But using \eqref{eq:y_t n_t} and the fact that $\left| \int_{\bbR^d} F_t w_{-\gamma/2}\right|$ and $\left| \int_{\bbR^d} |DF_t| w_{-\gamma/2}\right|$ are uniformly bounded for $m_0\in B_{2L}$ (since $2 \varepsilon< \gamma$) and $p_0\in \cG(m_0)$, we get
\begin{equation}\label{eq:bound dot n}
\vert \dot n_t\vert \, \leq\, C \gd ( \vert n_t\vert +\Vert y_t\Vert_{L^2(w_{\gamma})})\, .
\end{equation}
Hence, Gr\"onwall's lemma leads to
\begin{equation}
\label{eq:apriori_bound_n_t}
\vert n_t\vert \, \leq\, C' (\vert n_0\vert +\sup_{s\leq t}\Vert y_s\Vert_{L^2(w_{\gamma})})\, .
\end{equation}
So we deduce from \eqref{eq:yt_beta_1} and \eqref{eq:apriori_bound_n_t} the following rough bound: for some constant $C''>0$
\begin{equation}
\left\Vert y_{ t} \right\Vert_{ L^{ 2}(w_{  \gamma})}^{ 2}\,  \leq\,  \left\Vert y_{ 0} \right\Vert_{ L^{ 2}(w_{ \gamma})}^{ 2} +C''  \left(\left\vert n_0 \right\vert^{ 2} +\int_{ 0}^{t}\sup_{ v\leq s}\left\Vert  y_{ v} \right\Vert_{ L^{ 2}(w_{ \gamma})}^{ 2} {\rm d}s\right)\, ,
\end{equation}
so that, by Gr\"onwall's lemma, for some constant $C'''>0$
\begin{equation}
\sup_{ s\leq 2T_{ 0}} \left\Vert y_{ s} \right\Vert_{ L^{ 2}(w_{  \gamma})}^{ 2} \, \leq\,  C''' \left(\left\Vert y_{ 0} \right\Vert_{ L^{ 2}(w_{  \gamma})}^{ 2} + \left\vert n_{ 0} \right\vert^{ 2}\right)\, .
\end{equation}
Putting finally this estimate into \eqref{eq:apriori_bound_n_t} gives the following a priori bound
\begin{equation}
\label{eq:barut_y0}
\sup_{ t\leq 2T_{ 0}}\left\vert n_{ t} \right\vert^{ 2} \leq C \left( \left\vert n_{ 0} \right\vert^{ 2} +  \left\Vert y_{ 0} \right\Vert_{ L^{ 2}(w_{  \gamma})}^{ 2}\right)
\end{equation}
With this estimate at hand, one can conclude on the estimation of $ \left\Vert y_{ t} \right\Vert_{ L^{ 2}(w_{  \gamma})}$: applying \eqref{eq:barut_y0} to \eqref{eq:yt_beta_1}, we obtain, for some constants $c,C>0$
\begin{equation}
\frac12 \frac{ {\rm d}}{ {\rm d}t} \left\Vert y_{ t} \right\Vert_{ L^{ 2}(w_{  \gamma})}^{ 2} \, \leq \, - \left(k_{ \beta} -c \delta \right) \left\Vert y_{ t} \right\Vert_{ L^{ 2}(w_{  \gamma})}^{ 2} + C\gd \left\Vert y_{ t} \right\Vert_{ L^{ 2}(w_{ \gamma})} \left( \left\vert n_{ 0} \right\vert^{ 2} + \left\Vert y_{ 0} \right\Vert_{ L^{ 2}(w_{  \gamma})}^{ 2}\right)^{ \frac{ 1}{ 2}} \, ,\label{eq:yt_beta_2}
\end{equation}
and we are now in position to apply Lemma~\ref{lem:gronwall_sqrt} and conclude the proof, with $\gd_{ 10}=\min\left(\gd_{ 9},\frac{ k_{ \beta}}{2c} \right)$, and $\kappa_{ 10}^2=\frac{2C}{ k_{ \beta}}$.
\end{proof}

\subsection{Some preliminary definitions}
\label{sec:some_preliminary_def}
Let $\{ \tau_{ i}\}_{ i=1, \ldots, d}$ be the canonical basis of $ \mathbb{ R}^{ d}$. Recall the definition of $g_{ t, f}$ in \eqref{eq:g_ft}.
The fixed-point problem writes (recall \eqref{eq:eta_f})
\begin{equation}
\label{eq:fixed_point_f_1}
f_{ 0}(g_{ t, f_{ 0}}(m)) = p_{ t}^{ f_{ 0}(m),m},
\end{equation}
for all $m \in B_{ 2L}$, $t\in [0, T_{ 0}]$ such that $g_{t,f_{ 0}}(m)\in B_{2L}$. We use here for simplicity $f$ instead of $f_{ 0}$ (i.e. the fixed-point solution to \eqref{eq:eta_f}) in the following of the section. Our aim is to show that $ m \mapsto f(m)$ is continuously differentiable in the following sense: there exists a $\nabla_m f= \left(\partial_{ 1} f, \ldots, \partial_{ d}f\right)\in \mathcal{ C}(B_{2L},L^2(w_{\gamma}))$ such that (by definition)
\begin{equation}
\label{eq:diff_f_H}
\left\Vert  \frac{ f(m+ \varepsilon \tau_{ i}) - f(m)}{ \varepsilon} - \partial_{ i}f(m) \right\Vert_{L^2(w_{\gamma})} \underset{\gep \rightarrow 0}{\longrightarrow} 0,\ i=1,\ldots, d\, .
\end{equation}

Define the following set of functions :
\begin{multline}
\label{eq:G_alpha}
\mathcal{ I}_{ \gamma}:= \Big\lbrace \varphi=( \varphi_{ 1}, \ldots, \varphi_{ d})\in \left( \mathcal{ C}(B_{2L},L^2(w_{\gamma}))\right)^{ d}:\\
 \left\langle \varphi_{ k}(m)\, ,\, x\right\rangle= 0,\ k=1,\ldots, d,\ m\in B_{2L}\Big\rbrace.
\end{multline}
\begin{rem}
Note that by \eqref{eq:diff_f_H}, we have necessarily that $ \nabla f\in \mathcal{ I}_{ \gamma}$.
\end{rem}

\subsection{ Formal derivation of the fixed-point relation for the derivative}

Supposing \eqref{eq:diff_f_H}, and relying on Lemma \ref{lem:diff p}, choosing $m^1_0=m$ and $m^{2}_0= m+ \varepsilon \tau_{ i}$, we obtain for the right-hand side of \eqref{eq:fixed_point_f_1}, as $ \varepsilon\to 0$,
\begin{equation}
\label{aux:diff_mu}
p_{ t}^{ f(m+ \varepsilon \tau_{ i}),m+\varepsilon \tau_i}- p_t^{ f(m),m} \, =\, \varepsilon y_{ t}^{ f(m),m}\left[\partial_{ i}f(m), \tau_i\right]+o(\gep)\, ,
\end{equation}
and concerning the left-hand side,
\begin{multline}
f(g_{ t,f}(m+ \varepsilon \tau_{ i}))\,=\,   f\left(g_{ t, f}(m) + \varepsilon n_{ t}^{ f(m),m} \left[ \partial_{ i}f(m), \tau_i\right]+o( \varepsilon)\right),\\
= \,  f(g_{ t,f}(m)) + \varepsilon \nabla_m f(g_{ t,f}(m))\cdot n_{ t}^{ f(m),m} \left[ \partial_{ i}f(m), \tau_i\right]  + o(\varepsilon)\, .
\end{multline}
Identifying in both terms the contribution of order $ \varepsilon$ gives: for all $i=1, \ldots, d$
\begin{equation}
\label{eq:fixed_point_f_diff_1}
\nabla f(g_{t,f}(m))\cdot n_{ t}^{ f(m),m}[\partial_{ i}f(m), \tau_i]\, =\, y_{ t}^{ f(m),m}[\partial_{ i}f(m), \tau_i]\, .
\end{equation}
Hence, $ \nabla f$ is necessarily a solution $ \varphi$ in $ \mathcal{ I}_{ \gamma}$ to the following fixed-point equation 
\begin{equation}
\label{eq:fixed_point_phi_1}
\varphi(g_{t,f}(m))\cdot n_{ t}^{ f(m),m}[ \varphi_{ i}(m), \tau_i]\, =\, y_{ t}^{ f(m),m}[ \varphi_{ i}(m), \tau_i]\, \qquad  i=1, \ldots, d.
\end{equation}
Expanding the scalar product, we obtain the relation
\begin{equation}
\label{eq:fixed_point_partialf}
\varphi_{ i}(\xi)= \Psi_{ t}( \varphi)_{ i}(\xi)\, ,\qquad \text{for all} \quad \xi\in B_{ 2L},\ i=1,\ldots, d,\ t\in[0, 2T_{ 0}],
\end{equation}
where, for $m=g_{t,f}^{-1}(\xi)$,
\begin{equation}
\label{eq:def_Psit}
\Psi_{ t}(\varphi)_{ i}(\xi)\, :=\, 
\frac{1}{n_{ t,i}^{ f(m),m}[ \varphi_{ i}(m), \tau_i]}\bigg(y^{f(m),m}_t[\varphi_i(m), \tau_i]  - \sum_{ k\neq i}\varphi_{ k}(\xi) n_{ t,k}^{ f(m),m}[ \varphi_{ i}(m), \tau_i]\bigg)\, .
\end{equation}

\subsection{Well-posedness of the fixed-point problem}
\label{sec:fixed_point_phi_wellposed}

\begin{lemma}
\label{lem:proximity_Psi_t}
There exist constants $ \gd_{ 11}>0$ and $\kappa_{ 11}>0$ such that if $0\leq\gd\leq \gd_{ 11}$ and $\varphi$ satisfies $ \left\Vert \varphi_{ i}\right\Vert_{ \infty, \gamma} \leq \kappa_{11} \delta$, for all  all $i=1, \ldots, d$, then for all $t\in[T_{ 0}, 2T_{ 0}]$ and $i=1, \ldots, d$ we have $\left\Vert \Psi_{ t}(\varphi)_{ i}\right\Vert_{ \infty, \gamma} \leq \kappa_{11} \delta$.
\end{lemma}

\begin{proof}[Proof of Lemma~\ref{lem:proximity_Psi_t}]
Applying Lemma \ref{lem:contraction_y}, we get for all $m\in B_{2L}$ and $t\in [T_{ 0},2T_{ 0}]$:
\begin{equation}
\left\Vert y^{f(m),m}_t[\varphi_i(m), \tau_i] \right\Vert_{L^2(w_{\gamma})} \, \leq\, \left(e^{-\chi'(\gd)t}+\kappa_{ 10} \gd\right)\Vert \phi_i\Vert_{\infty,\gamma}+ \kappa_{ 10} \gd\, ,
\end{equation}
and
\begin{equation}
\left\vert n^{f(m),m}_t[\varphi_i(m), \tau_i]- \tau_i\right\vert\, \leq\, \kappa_{ 10} \gd(1+\Vert \phi_i\Vert_{\infty,\gamma})\, .
\end{equation}
We deduce, recalling \eqref{eq:def_Psit}, that for a constant $C>0$,
\begin{equation}
\left \Vert \Psi_{ t}(\varphi)_{ i}\right\Vert_{ \infty, \gamma} \, \leq
\, \left(e^{-\chi'(\gd)t}+C\gd\right) \max_k \Vert \phi_k\Vert_{\infty,\gamma}+ C\gd\, ,
\end{equation}
and this concludes the proof with the choice $\kappa_{11}=\frac34 C$, since $C\gd\leq \frac18$ and $e^{-\chi'(\gd_{ 11})T_{ 0}}\leq\frac18$ for $\gd_{ 11}$ small enough (recall \eqref{eq:def T}).
\end{proof}

\begin{definition}
We say that $ \varphi\in \mathcal{J}_{ \gamma}=\mathcal{J}_\gamma(\gd,\kappa_{11})$ if
$ \varphi\in \mathcal{ I}_{ \gamma}$ and
\begin{equation}
\sup_{i=1,\dots,d}\left\Vert \varphi_{ i} \right\Vert_{ \infty, \gamma} \,\leq\,  \kappa_{11} \delta\, .
\end{equation}
\end{definition}

\begin{lemma}
\label{lem:contraction_Psi}
There exists a constant $ \delta_{ 12}>0$ such that if $ 0\leq\delta\leq \delta_{12}$, for all $ \varphi, \psi\in \mathcal{J}_{ \gamma}$ and $t\in[T_{ 0}, 2T_{ 0}]$, we have
\begin{equation}
\sup_{i=1,\ldots,d}\left\Vert \Psi_{ t}(\varphi)_i - \Psi_{ t}(\psi)_i\right\Vert_{ \infty,\gamma} \leq \frac{ 1}{ 2} \sup_{i=1,\ldots,d}\left\Vert \varphi_i- \psi_i \right\Vert_{ \infty, \gamma}\, .
\end{equation}
\end{lemma}

\begin{proof}[Proof of Lemma~\ref{lem:contraction_Psi}]
Let $i=1,\ldots, d$, $m\in B_{2L}$. We have
\begin{align}
\Vert \Psi_{ t}(\varphi)_{ i}(\xi)& - \Psi_{ t}(\psi)_{ i}(\xi)\Vert_{ L^{ 2}(w_{ - \gamma})} \nonumber\\
& = \left\Vert \frac{ y^{f(m),m}_t[\varphi_i(m), \tau_i] - a_{ t, i}(\varphi)}{ b_{ t, i}(\varphi)} - \frac{ y^{f(m),m}_t[\psi_i(m), \tau_i]  - a_{ t, i}(\psi)}{b_{ t, i}(\psi)} \right\Vert_{ L^{ 2}(w_{\gamma})} \nonumber\\
& \ \leq \frac{ \left\Vert y^{f(m),m}_t[\varphi_i(m)-\psi_i(m),0] \right\Vert_{ L^{ 2}(w_{\gamma})}}{ \left\vert b_{ t, i}(\varphi) \right\vert}+ \frac{ \left\Vert a_{ t, i}(\psi) - a_{ t, i}(\varphi)\right\Vert_{ L^{ 2}(w_{\gamma})}}{ \left\vert b_{ t, i}(\psi) \right\vert}  \nonumber\\
& \qquad + \frac{ \left\vert b_{ t, i}(\psi) - b_{ t, i}(\varphi)\right\vert }{ \left\vert b_{ t, i}(\varphi) b_{ t, i}(\psi) \right\vert} \left(\left\Vert y^{f(m),m}_t[\psi_i(m), \tau_i] \right\Vert_{ L^{ 2}(w_{\gamma})} + \left\Vert a_{ t, i}(\varphi) \right\Vert_{ L^{ 2}(w_{\gamma})} \right)\, ,\label{eq:Bt}
\end{align}
where we have used the notations
\begin{align*}
a_{ t, i}(\varphi)&:= \sum_{ k\neq i}\varphi_{ k}(\xi) n_{ t,k}^{ f(m),m}[ \varphi_{ i}(m), \tau_i],\\
b_{ t, i}(\varphi)&:= n_{ t,i}^{ f(m),m}[ \varphi_{ i}(m), \tau_i].
\end{align*}
But Lemma \ref{lem:contraction_y} implies
\begin{equation}
\left\Vert y^{f(m),m}_t[\varphi_i(m)-\psi_i(m),0] \right\Vert_{ L^{ 2}(w_{  \gamma})}\, \leq\, 
\left(e^{-\chi'(\gd)t}+\kappa_{10}\gd\right)\Vert \varphi_i-\psi_i\Vert_{\infty,\gamma}\, ,
\end{equation}
and
\begin{equation}
\left\vert b_{ t, i}(\psi) - b_{ t, i}(\varphi)\right\vert\, =\, \left\vert n_{ t,i}^{ f(m),m}[ \varphi_{ i}(m)-\psi_i(m),0]\right\vert\, \leq\, \kappa_{10}\gd \Vert \varphi_i-\psi_i\Vert_{\infty,\gamma}\, ,
\end{equation}
and by Lemma \ref{lem:contraction_y} and Lemma \ref{lem:proximity_Psi_t}, for $\gd$ small enough,
\begin{equation}
\vert b_{t,i}(\varphi)\vert \, \geq \, \frac12\, , \qquad \vert b_{t,i}(\psi)\vert \, \geq \, \frac12\, .
\end{equation}
Secondly, using similar arguments, we obtain for a constant $C>0$,
\begin{align}
\Vert a_{ t, i}(\psi) - a_{ t, i}(\varphi)\Vert_{ L^{ 2}(w_{ \gamma})}\, & \leq  \sup_{j=1,\ldots,d}\Vert \varphi_j-\psi_j\Vert_{\infty,\gamma}\sum_{ k\neq i} \left\vert n_{ t,k}^{ f(m),m}[ \varphi_{ i}(m),0]\right\vert \nonumber\\
&\qquad + \sup_{j=1,\ldots,d}\Vert \psi_j\Vert_{\infty,\gamma}\sum_{ k\neq i} \left\vert n_{ t,k}^{ f(m),m}[ \varphi_{ i}(m)-\psi_i(m),0]\right\vert\nonumber\\
&\leq 
\, C \gd \,\sup_{j=1,\ldots,d}\Vert \varphi_j-\psi_j\Vert_{\infty,\gamma}\, .
\end{align}
Gathering all these estimates we get, for a constant $C'>0$
\begin{multline}
\Vert \Psi_{ t}(\varphi)_{ i}(m)(\cdot +m) - \Psi_{ t}(\psi)_{ i}(m)(\cdot +m)\Vert_{ L^{ 2}(w_{  \gamma})}\\ \leq\, \left(2 e^{-\chi'(\gd)t}+C'\gd\right) \sup_{j=1,\ldots,d}\Vert \varphi_j-\psi_j\Vert_{\infty,\gamma}\, ,
\end{multline}
which implies the result for $\gd_{ 12}$ small enough (recall \eqref{eq:def T}).
\end{proof}

\subsection{ Identification with the derivative of $f$}
\label{sec:identification_phi_diff_f}

Lemma \ref{lem:contraction_Psi} implies the existence of a unique fixed-point (that is denoted $\varphi$) to the mapping $\Psi_{ T_{ 0}}$ in $\mathcal{J}_\gamma$ and we are now ready to prove that $f$ is $ \mathcal{ C}^1$ with derivative given by $\varphi$. As in \cite{wiggins2013normally}, the idea we follow is to apply the following classical result (which corresponds to Lemma 3.3.8 in \cite{wiggins2013normally}).

\begin{lemma}\label{lem:rho}
Suppose that $ \rho:\bbR_+\mapsto \bbR_+$ is nondecreasing and satisfies the inequality
\begin{equation}\label{eq:ineq rho}
\rho(a)\, \leq\, \zeta_1 \rho(\zeta_2 a)+ \iota(a)\, ,
\end{equation}
where $a$ is small, $0\leq \zeta_1\leq 1$, $ \zeta_{ 2}\in \mathbb{ R}$ and $\lim_{a\rightarrow 0}\iota(a)=0$. Then $\lim_{a\rightarrow 0}\rho(a)=0$.
\end{lemma}

We refer to \cite{wiggins2013normally} for a proof of this result. We consider the function $\rho$ defined as follows:
\begin{equation}
\gamma(a) := \sup_{ \xi, \xi^{ \prime}\in B_{2L},\, \left\vert \xi^{ \prime}- \xi \right\vert< a} \quad \frac{1}{\vert \xi'-\xi\vert }\left\Vert f(\xi^{ \prime}) - f(\xi) -  \varphi(\xi)\cdot ( \xi^{ \prime}- \xi) \right\Vert_{L^2(w_{\gamma})}\, ,a>0\,,
\end{equation}
and we aim at proving \eqref{eq:ineq rho}, so that Lemma \ref{lem:rho} ensures that $f$ is $ \mathcal{ C}^1$ in the sense of \eqref{eq:diff_f_H}.

\medskip

For any $ \xi, \xi^{ \prime}\in B_{2L}$, consider 
\begin{equation}
\Delta_{ \xi, \xi^{ \prime}} \left[f, \varphi\right]\, :=\,  f(\xi^{ \prime}) - f(\xi) -  \varphi(\xi) \cdot(\xi^{ \prime} - \xi)\, .
\end{equation}
By Lemma~\ref{lem:bij}, one has $ \xi = g_{T_{ 0}, f}(m)$ and $ \xi^{ \prime}= g_{ T_{ 0}, f}(m^{ \prime})$ for $m,m^{ \prime}\in B_{2L}$. Hence, by Lemma~\ref{lem:diff p}, we have
\begin{multline}
f(\xi')-f(\xi)\, =\, p^{f(m'),m'}_{ T_{ 0}}-p^{f(m),m}_{ T_{ 0}}\\ =\, y^{f(m),m}_{ T_{ 0}}[f(m')-f(m),m'-m]+O_{L^2(w_{\gamma})}(\vert m'-m\vert^2)\, .
\end{multline}
On the other hand, applying again Lemma \ref{lem:diff p} we obtain
\begin{equation}
\varphi(\xi) (\xi^{ \prime} - \xi)\, =\, 
 \varphi(\xi)\cdot n^{f(m),m}_{T_{ 0}}[f(m')-f(m),m'-m] +O_{L^2(w_{\gamma})}(\vert m'-m\vert^2)\, .
\end{equation}
But we have the decomposition
\begin{multline}
\varphi(\xi)\cdot n^{f(m),m}_{T_{ 0}}[f(m')-f(m),m'-m]\, =\,  \varphi(\xi)\cdot n^{f(m),m}_{T_{ 0}}\left[\Delta_{m,m'}[f,\varphi],0\right]\\
+ \varphi(\xi)\cdot n^{f(m),m}_{T_{ 0}}\left[\varphi(m)\cdot (m'-m),m'-m\right]\, , 
\end{multline}
and, recalling the fixed-point relation \eqref{eq:fixed_point_phi_1},
\begin{multline}
 \varphi(\xi)\cdot n^{f(m),m}_{T_{ 0}}\left[\varphi(m)\cdot (m'-m),m'-m\right]\,
 =\, \sum_{k=1}^d (m'_k-m_k)\varphi(\xi) \cdot n^{f(m),m}_{T_{ 0}}[\phi_k(m), \tau_k]
 \\ =\, \sum_{k=1}^d (m'_k-m_k) y_{T_{ 0}}^{f(m),m}[\phi_k(m), \tau_k]\,
 =\, y_{T_{ 0}}^{f(m),m}\left[\phi(m)\cdot (m'-m),m'-m\right]\, .
\end{multline}
Finally, all these estimates lead to
\begin{multline}
\Delta_{ \xi, \xi^{ \prime}} \left[f, \varphi\right]\, =\, y_{ T_{ 0}}^{f(m),m}\left[\Delta_{m,m'}[f,\varphi],0 \right] - \varphi(\xi)\cdot n^{f(m),m}_{T_{ 0}}\left[\Delta_{m,m'}[f,\varphi],0\right]
 \\+O_{L^2(w_{\gamma})}(\vert m'-m\vert^2)\, .
\end{multline}
Applying Lemma \ref{lem:contraction_y} and Lemma \ref{lem:encadre m} (which ensures that $|\xi'-\xi|\geq \frac12 |m'-m|$), we obtain (here $n_0=0$):
\begin{equation}
\frac{1}{\vert \xi'-\xi\vert}\left\Vert \Delta_{ \xi, \xi^{ \prime}} \left[f, \varphi\right]\right\Vert_{L^2(w_{\gamma})}\, \leq\,
\frac{2 \left( e^{-\chi'(\gd) T_{ 0}}+2\kappa_{10} \gd\right)}{\vert m'-m\vert} \left\Vert \Delta_{ m,m'} \left[f, \varphi\right]\right\Vert_{L^2(w_{\gamma})}+O(\vert m'-m\vert)\, ,
\end{equation}
and thus, taking $\gd$ small enough so that $2 \left( e^{-\chi'(\gd) T_{ 0}}+2\kappa_{10} \gd\right)\leq 1$ (recall \eqref{eq:def T}) and $\zeta_2=2$ (relying again on Lemma~\ref{lem:encadre m}), $\rho$ indeed satisfies \eqref{eq:ineq rho}, which ensures that $f$ is $ \mathcal{ C}^1$, with $\nabla_m f=\varphi$. Recalling moreover Lemma \ref{lem:proximity_Psi_t}, we have just proved the following proposition:

\begin{proposition}\label{prop:reg f}
There exists a constant $\gd^{ \prime}>0$ such that for all $\gd\leq \gd^{ \prime}$ the mapping $f$ has $ \mathcal{ C}^1$-regularity in the sense of \eqref{eq:diff_f_H} and satisfies
\begin{equation}
\sup_{|m|< 2L} \quad \max_{i=1,\ldots,d}\quad \Vert \partial_i f(m)\Vert_{L^2(w_{\gamma})}\, \leq\,  \kappa_{ 11} \gd\, .
\end{equation}
\end{proposition}

\subsection{Proof of Theorem \ref{th:main2}}

With the notations introduced in the proof of Theorem \ref{th:main1}, the phase dynamics on $ \mathcal{ M}^\gd$ is given by
\begin{equation}
\frac{1}{\gd}\dot m^\gd_t\, =\, \int_{\bbR^d} Fq^\gd_{m^\gd_t}\, =\,\int_{\bbR^d}F(x+m^\gd_t) q_0(x)\dd x + \int_{\bbR^d} F(x+m^\gd_t)\big(f(m^\gd_t)(x)-q_0(x)\big)\dd x\, ,
\end{equation}
and thus \eqref{eq:dot m th} holds with
\begin{equation}
g^\gd(m)\, :=\, \frac{1}{\gd}\int_{\bbR^d} F(x+m)\big(f(m)(x)-q_0(x)\big)\dd x\, .
\end{equation}
It remains to control $g^\gd$: since $f\in \cH$ and
$\sup_{m\in B_L} \left|\int F(x+m)w_{-\ga}(x) {\rm d}x\right|< \infty$, we have $\Vert g^\gd\Vert_{ \mathcal{ C}(B_L,\bbR^d)}\leq C$ for some constant $C>0$ and the bound on the derivative follows from the same arguments, relying on Proposition \ref{prop:reg f} and remarking that
\begin{equation}
\nabla g(m)\, =\, \frac{1}{\gd}\int_{\bbR^d}\nabla F(x+m)\big( f(m)(x)-q_0(x)\big)\dd x+\frac{1}{\gd} \int_{\bbR^d} F(x+m)\nabla_m f(m)(x)\dd x\, .
\end{equation}

\appendix
\section{ Some technical lemmas}
\begin{lemma}
\label{lem:rough_bound_q}
For any $ \alpha, a, b, c>0$, we have
\begin{equation}
\label{eq:rough_bound_q}
q(u):=-au^{ 2} + bu+c \leq \left( \frac{ b^{ 2}}{ 4a} + c\right) \exp\left(\frac{ \alpha}{ 2a} \left( \frac{ b^{ 2}}{ a}+2c\right)\right)e^{ - \frac{ \alpha u^{ 2}}{ 2}},\ u\geq 0.
\end{equation}
\end{lemma}
\begin{proof}[Proof of Lemma~\ref{lem:rough_bound_q}]
This is obvious since
\begin{equation}
\label{eq:q_u}
q(u)= -a \left(u- \frac{ b}{ 2a}\right)^{ 2}+ \frac{ b^{ 2}}{ 4a}+c\leq\begin{cases}
0& \text{ if }u\geq \frac{ b}{ 2a}+ \left( \frac{ b^{ 2}}{ 4a^{ 2}}+ \frac{ c}{ a}\right)^{ \frac{ 1}{ 2}},\\
\frac{ b^{ 2}}{ 4a} +c & \text{ if }u< \frac{ b}{ 2a}+ \left( \frac{ b^{ 2}}{ 4a^{ 2}}+ \frac{ c}{ a}\right)^{ \frac{ 1}{ 2}}\,.
\end{cases}
\end{equation}
Noting that in the second case, we have in particular $u^{ 2}\leq \frac{ b^{ 2}}{ a^{ 2}}+ \frac{ 2c}{ a} $ and \eqref{eq:rough_bound_q} follows.
\end{proof}
\begin{lemma}
\label{lem:gronwall_sqrt}
Let $v$ be a continuously differentiable function on $[0, +\infty)$ such that $v(t)\geq0$ for all $t\geq0$. Suppose that there exists $ \alpha, \beta>0$ such that
\begin{equation}
v^{ \prime}(t)\leq - \alpha v(t) + \beta \sqrt{ v(t)}.
\end{equation}
Then, for all $t\geq 0$,
\begin{align}
v(t) &\leq \max \left(v(0),\ \frac{ \beta^{ 2}}{ \alpha^{ 2}}\right), \label{eq:gronwall_sqrt_1}\\
v(t) &\leq \left(\sqrt{ v(0)}e^{ - \alpha t/2} + \frac{ \beta}{ \alpha} \right)^{ 2}. \label{eq:gronwall_sqrt_2}
\end{align}
\end{lemma}
\begin{proof}[Proof of Lemma~\ref{lem:gronwall_sqrt}]
The first result follows from the fact that $v(t)$ is always non-increasing unless $ \sqrt{ v(t)}\leq \frac{ \beta}{ \alpha}$. We now prove the second inequality: let $t\geq0$ such that $v(t)>0$. Consider the maximal interval $I:=(t^{ -}, t^{ +})$ (that is non empty by continuity of $v$) containing $t$ such that $v(u)>0$ on $I$. Let $f(u):= \left( \alpha \sqrt{ v(u)} - \beta\right)e^{ \alpha u/2}$. On $I$, $f^{ \prime}(u) = \frac{ \alpha e^{ \alpha u/2}}{ 2 \sqrt{ v(u)}} \left( \alpha v(u) - \beta \sqrt{ v(u)} + v^{ \prime}(u)\right)\leq 0$, so that $f$ is nonincreasing on $I$. Consider now the solution $w$ of the equation $w^{ \prime}(t)= - \alpha w(t) + \beta \sqrt{ w(t)}$ such that $w(t^{ -})=v(t^{ -})$. Then, by the same calculation, $g:= u \mapsto \left( \alpha \sqrt{ w(u)} - \beta\right)e^{ \alpha u/2}$ is constant on $I$, equal to $g(t^{ -})$. This means that $f(u) \leq f(t^{ -})= g(t^{ -})= g(u)$. By definition of $f$ and $g$, this implies that $v(u)\leq w(u)$ for all $u\in I$. This means that 
\begin{equation}
v(u) \leq  \frac{ e^{ - \alpha u}}{ \alpha^{ 2}} \left(g(t^{ -}) + \beta e^{ \alpha u/2}\right)^{ 2}=\frac{ e^{ - \alpha u}}{ \alpha^{ 2}} \left(\left( \alpha \sqrt{ v(t^{ -})} - \beta\right)e^{ \alpha t^{ -}/2} + \beta e^{ \alpha u/2}\right)^{ 2}.
\end{equation}
We have now two possibilities: if $t^{ -}=0$, then,
\begin{align*}
v(t) &\leq  \frac{ 1}{ \alpha^{ 2}} \left(\alpha \sqrt{ v(0)}e^{ - \alpha t/2} + \beta \left(1-e^{ -\alpha t/2}\right)\right)^{ 2},\\
&\leq  \left(\sqrt{ v(0)}e^{ - \alpha t/2} + \frac{ \beta}{ \alpha} \right)^{ 2}.
\end{align*}
In the case $t^{ -}>0$, by continuity of $v$, $v(t^{ -})=0$. In this case,
\begin{align*}
v(t) &\leq  \frac{ 1}{ \alpha^{ 2}} \left(\beta \left(1-e^{ -\alpha t/2}\right)\right)^{ 2}\leq \frac{ \beta^{ 2}}{ \alpha^{ 2}}.
\end{align*}
This proves Lemma~\ref{lem:gronwall_sqrt}.
\end{proof}

\section{Proof of Lemma~\ref{lem:F_vs_q0}}
\label{sec:proof_lemma_F_VS_q0}
For $m_0\in \bbR^d$ and $L_0>0$ that $ \left\vert m_{ 0} \right\vert_{ K \sigma^{ -2}}=L_0$, it is easy to see that $\gep \vert m_0+x\vert_{ K \sigma^{ -2}}^2 -\vert x\vert_{ K \sigma^{ -2}}^2\leq -\frac{1}{2}\vert x\vert_{ K \sigma^{ -2}}^2$ as soon as $\vert x\vert_{ K \sigma^{ -2}}\geq r_{ \varepsilon}L_0$, for $r_{ \varepsilon}:= \frac{ \sqrt{ 2\varepsilon}}{ 1- 2 \varepsilon} \left(\sqrt{ 2\varepsilon} + 1\right)\geq \sqrt{ 2 \varepsilon}$. In the following, we consider $ \varepsilon$ small enough such that $ r_{ \varepsilon}< \frac{ \min(1, c_{ F})}{ 4}$. A small calculation shows that this is true when $\gep< \rho_{ 0}$ with 
\[\rho_{ 0}:= \frac{ \min(c_{ F}, 1)^{ 2}}{ 2 \left(4+ \min(c_{ F}, 1)\right)^{ 2}}\leq \frac{ 1}{ 32}.\]
With this notations at hand, 
\begin{align}\label{eq: ugly 1}
\left\vert \int_{\bbR^d} x\cdot K \sigma^{ -2}F(m_0+x) q_0(x)\dd x\right\vert & \leq \, r_{ \varepsilon}L\int_{\vert x\vert_{ K \sigma^{ -2}} <r_{ \varepsilon}L_0}\vert F(x+m_0)\vert_{ K \sigma^{ -2}} q_0(x)\dd x \nonumber\\ & \qquad + \int_{\vert x\vert_{ K \sigma^{ -2}}\geq r_{ \varepsilon}L_0} \left\vert x \right\vert_{ K \sigma^{ -2}} \left\vert F(m_0+x) \right\vert_{ K \sigma^{ -2}} q_0(x)\dd x
\nonumber\\ 
&:= I_{ 1}+I_{ 2}.
\end{align}
Concerning the second term $I_{ 2}$ in the above sum, using the introductory remark on $r_{  \varepsilon}$,
\begin{align}
I_{ 2} &\leq  \frac{ C_{ F}}{ \left((2 \pi)^{ d} \det \left( \sigma^{ 2}K^{ -1}\right)\right)^{ \frac{ 1}{ 2}}}\int_{\vert x\vert_{ K \sigma^{ -2}}\geq r_{ \varepsilon}L_0} \left\vert x \right\vert_{ K \sigma^{ -2}} \exp \left(- \frac{ 1}{ 4} \left\vert x \right\vert_{ K \sigma^{ -2}}^{ 2} \right)\dd x, \nonumber\\
&\leq \frac{ (\bar k)^{ \frac{ 1}{ 2}}}{ \underline{ \sigma}}Q_{ 1}(L_0) e^{ - \frac{ 1}{ 2} \varepsilon L_0^{ 2}}, \label{eq:I2}
\end{align}
for some polynomial function $Q_{ 1}(L_0)$ in $L_0$, independent of $K$, $ \sigma$ and that can be made independent of $ \varepsilon< \rho_{ 0}$.
Secondly, applying \eqref{hyp:F_dot_x_bound 2},
\begin{multline}
\int_{\bbR^d}(x+m_0)\cdot K \sigma^{ -2}F(x+m_0)q_0(x)\dd x\\
\leq\, C_F\int_{\vert x+m_0\vert_{ K \sigma^{ -2}} \leq r}q_0(x)\dd x- c_{ F}\int_{ \mathbb{ R}^{ d}}\vert x+m_0\vert_{ K \sigma^{ -2}}\, \vert F(x+m_0)\vert_{ K \sigma^{ -2}} q_0(x)\dd x\\
\leq\, C_F\int_{\vert x+m_0\vert_{ K \sigma^{ -2}} \leq r}q_0(x)\dd x- c_{ F}\int_{\vert x\vert_{ K \sigma^{ -2}} <r_{ \varepsilon}L_0}\vert x+m_0\vert_{ K \sigma^{ -2}}\, \vert F(x+m_0)\vert_{ K \sigma^{ -2}} q_0(x)\dd x\,.
\end{multline}
Recall that $r_{ \varepsilon}< \frac{ 1}{ 4}$ and assume that $L_0 >\frac{ 4r}{ 3}$ so that under these two assumptions, 
\begin{equation}
\label{eq:reps_m0}
 \left\lbrace \left\vert x+ m_{ 0} \right\vert_{ K \sigma^{ -2}}\leq r\right\rbrace\subset \left\lbrace \left\vert x+ m_{ 0} \right\vert_{ K \sigma^{ -2}}\leq \frac{ 3L_0}{ 4}\right\rbrace \subset \left\{\vert x\vert_{ K \sigma^{ -2}} \geq \frac{ L_0}{ 4} \right\} \subset \left\{\vert x\vert_{ K \sigma^{ -2}} \geq r_{ \varepsilon}L_0 \right\}.
\end{equation} 
Hence,
\begin{multline}
\int_{\bbR^d}(x+m_0)\cdot K \sigma^{ -2}F(x+m_0)q_0(x)\dd x\\
\leq\, C_F\int_{\vert x\vert_{ K \sigma^{ -2}} \geq r_{ \varepsilon}L_0}q_0(x)\dd x- \frac{ 3c_{ F}}{ 4}L_0\int_{\vert x\vert_{ K \sigma^{ -2}} < r_{ \varepsilon}L_0} \vert F(x+m_0)\vert_{ K \sigma^{ -2}} q_0(x)\dd x\\
\leq\, Q_{ 2}(L_0)e^{ - \varepsilon L_0^{ 2}}- \frac{ 3c_{ F}}{ 4}L_0\int_{\vert x\vert_{ K \sigma^{ -2}} < r_{ \varepsilon}L_0} \vert F(x+m_0)\vert_{ K \sigma^{ -2}} q_0(x)\dd x\,.\label{eq: ugly 2}
\end{multline}
for another polynomial $Q_{ 2}$ in $L_0$.
Gathering \eqref{eq: ugly 1} \eqref{eq:I2} and \eqref{eq: ugly 2}, using $r_{ \varepsilon}< \frac{ c_{ F}}{ 4}$, the second part of  \eqref{hyp:F_dot_x_bound 2} and \eqref{eq:reps_m0} again, we deduce that, for $L_0> \frac{ r}{ r_{ \varepsilon}}$,
\begin{align}
m_0\cdot \int_{\bbR^d} F(m_0+x) q_0(x)\dd x&\leq \frac{ (\bar k)^{ \frac{ 1}{ 2}}}{ \underline{ \sigma}}Q_{ 1}(L_0) e^{ - \frac{ 1}{ 2} \varepsilon L_0^{ 2}} + Q_{ 2}(L_0) e^{ - \varepsilon L_0^{ 2}} \\
& \qquad \qquad \qquad \qquad - \frac{ 3c_{ F}^{ 2}}{ 8}L_0^{ 2} \int_{ \left\vert x \right\vert_{ K \sigma^{ -2}}< r} q_{ 0}(x) {\rm d}x,\nonumber\\
&\leq \frac{ (\bar k)^{ \frac{ 1}{ 2}}}{ \underline{ \sigma}}Q_{ 1}(L_0) e^{ - \frac{ 1}{ 2} \varepsilon L_0^{ 2}} + Q_{ 2}(L_0) e^{ - \varepsilon L_0^{ 2}} - cL_0^{ 2}\,. \label{eq:bound m_0 cdot int}
\end{align}
for some $c>0$. Choosing now $L_0> \frac{ 1}{ \varepsilon}$ sufficiently large, we obtain the result.
\section{Proof of Proposition \ref{prop:PDE_well_posed}}\label{app: proof uniqueness}

 One weak solution to \eqref{eq:PDE_mod_h} is provided by $ \mu_{ t}:= \mathcal{ L}(X_{ t})$ where $X_{ t}$ is the nonlinear process defined in \eqref{eq:modif mc kean}. It remains to show uniqueness. For the rest of the proof, $ \mu$ is the law of the process $X_{ \cdot}$ and $ \nu$ stands for any other weak solution to \eqref{eq:PDE_mod_h} in $ \mathcal{ C}([0, \infty), \mathcal{ P}_{ 2})$ such that $ \nu_{ 0}= \mu_{ 0}$. The point is to prove that $ \nu_{ t}= \mu_{ t}$, $t\in[0, T]$, for any arbitrary $T>0$.
 
 For such a $ \nu$, define the following diffusion
\begin{equation}
\label{eq:diff_nu}
\dd Y_t\, =\, \left(\delta F(Y_t) + \delta h\left( \int_{ \mathbb{ R}^{ d}} z \nu_{ t}( {\rm d} z)\right)-K \left(Y_t- \int_{ \mathbb{ R}^{ d}} z \nu_{ t}({\rm d} z)\right)\right)\dd t+\sqrt{2}\gs \dd B_t\, ,
\end{equation}
with $Y_{ 0}\sim \nu_{ 0}$. For any $0\leq s\leq t$, denote by $ \left\lbrace \varphi_{ s}^{ t}(y)\right\rbrace_{ s\leq t \leq T}$ the unique solution of \eqref{eq:diff_nu} with initial condition at $t=s$ such that $ \varphi_{ s}^{ s}=y$. Finally, for any test function $f$, define 
\begin{equation}
\label{eq:propagator}
P_{ s, t}f(y):= \mathbf{ E}_{ B} f \left( \varphi_{ s}^{ t}(y)\right).
\end{equation}
Let us suppose that $F$ is uniformly Lispchitz on $ \mathbb{ R}^{ d}$ (one can remove this assumption by replacing $F$ by its Yosida approximation $ \left(F_{\lambda}\right)_{ \lambda>0}$, as in \cite{LucSta2014}, Section~7, (7.3); see also \cite{MR1840644}, Appendix~A). Having proven uniqueness for $ \mu_{ \lambda}$ (that is when $F$ has been replaced by $F_{ \lambda}$), it suffices to see that the distance \eqref{eq:W1} between $ \mu$ (resp. $ \nu$) and $ \mu_{ \lambda}$ (resp. $ \nu_{ \lambda}$) converges to $0$ as $ \lambda\to\infty$, see \cite{LucSta2014}, Proposition~7.1). Under the assumptions made on the model, the propagator $P$ satisfies the following Backward Kolmogorov equation (see \cite{doi:10.1080/07362999808809576}, Remark 2.3): for any regular test function $ \varphi$,
\begin{multline}
\label{eq:kolmogorov}
\partial_{ s} P_{ s, t} \varphi(y) + \nabla\cdot (\sigma^{ 2} \nabla )P_{ s, t} \varphi(y)\\
+ \left(  \left(\delta F(y) + \delta h\left( \int_{ \mathbb{ R}^{ d}} z \nu_{ t}( {\rm d} z)\right)-K \left(y- \int_{ \mathbb{ R}^{ d}} z \nu_{ t}({\rm d} z)\right)\right)\cdot \nabla\right)P_{ s, t} \varphi(y)=0.
\end{multline}
Let us now apply Ito formula to $ t \mapsto P_{ t, T} \varphi(X_{ t})$, where we recall that $X$ solves \eqref{eq:modif mc kean}:
\begin{align}
P_{ t, T} \varphi(X_{ t}) &= P_{ 0, T} \varphi(X_{ 0}) + \int_{ 0}^{t} \partial_{ s} P_{ s, T} \varphi(X_{ s}) {\rm d} s \nonumber\\
&+ \int_{ 0}^{t} \nabla P_{ s, T} \varphi(X_{ s}) \cdot \left( \delta F(X_{ s}) + \delta h \left( \int_{ \mathbb{ R}^{ d}} z \mu_{ s}({\rm d}z)\right) - K \left(X_{ s} - \int z \mu_{ s}({\rm d} z)\right)\right){\rm d} s \nonumber\\
&+  \int_{ 0}^{t} \nabla\cdot \gs^2\nabla P_{ s, T} \varphi(X_{ s}) {\rm d} s + \sqrt{ 2} \sigma\int_{ 0}^{t} \nabla P_{ s, T} \varphi(X_{ s}) {\rm d}B_{ s}.
\end{align}
Using \eqref{eq:kolmogorov}, this simplifies into
\begin{align}
P_{ t, T} \varphi(X_{ t}) &= P_{ 0, T} \varphi(X_{ 0}) + \sqrt{ 2} \sigma\int_{ 0}^{t} \nabla P_{ s, T} \varphi(X_{ s}) {\rm d}B_{ s}\nonumber\\
&+ \int_{ 0}^{t} \nabla P_{ s, T} \varphi(X_{ s}) \cdot \Bigg(\delta \left\lbrace h \left( \int_{ \mathbb{ R}^{ d}} z \mu_{ s}({\rm d}z)\right) - h\left( \int_{ \mathbb{ R}^{ d}} z \nu_{ s}( {\rm d} z)\right) \right\rbrace \nonumber\\ &\quad \quad- K \left( \int z \mu_{ s}({\rm d} z) - \int z \nu_{ s}({\rm d} z)\right)\Bigg){\rm d} s.
\end{align}
Taking the expectation w.r.t. the Brownian motion and for $t=T$ (recall that $P_{ T, T} \varphi= \varphi$), we obtain (using the notation $ \left\langle \varphi\, ,\, \mu\right\rangle := \int \varphi(x) \mu({\rm d}x)$),
\begin{multline}
\left\langle \varphi\, ,\, \mu_{ T}\right\rangle = \left\langle P_{ 0, T}\varphi\, ,\, \mu_{ 0}\right\rangle\\
+ \int_{ 0}^{T}  \left\langle \nabla P_{ s, T} \varphi\, ,\, \mu_{ s}\right\rangle \cdot \left( \delta \left\lbrace h \left( \int_{ \mathbb{ R}^{ d}} z \mu_{ s}({\rm d}z)\right) - h\left( \int_{ \mathbb{ R}^{ d}} z \nu_{ s}( {\rm d} z)\right) \right\rbrace - K \left\langle z\, ,\, \mu_{ s}- \nu_{ s}\right\rangle\right){\rm d} s.
\end{multline}
Furthermore, for any regular fonction $ \varphi$, by definition of $P_{ t, T}$ we have $\partial_{ t} \left\langle P_{ t, T} \varphi\, ,\, \nu_{ t}\right\rangle=0$.
We obtain finally
\begin{multline}
\label{eq:mu_minus_nu}
\left\langle \varphi\, ,\, \mu_{ T} - \nu_{ T}\right\rangle = \left\langle P_{ 0, T} \varphi\, ,\, \mu_{ 0} - \nu_{ 0}\right\rangle\\
+ \int_{ 0}^{T}  \left\langle \nabla P_{ s, T} \varphi\, ,\, \mu_{ s}\right\rangle \cdot \left( \delta \left\lbrace h \left( \int_{ \mathbb{ R}^{ d}} z \mu_{ s}({\rm d}z)\right) - h\left( \int_{ \mathbb{ R}^{ d}} z \nu_{ s}( {\rm d} z)\right) \right\rbrace - K \left\langle z\, ,\, \mu_{ s}- \nu_{ s}\right\rangle\right){\rm d} s.
\end{multline}
Let us introduce the usual Wasserstein distance $W_{ 1}$ between two measures. By the Kantorovich-Rubinstein duality, an expression of this distance is
\begin{equation}
\label{eq:W1}
W_{ 1}( \mu, \nu)= \sup_{ \left\Vert f \right\Vert_{ Lip}\leq 1} \left\vert \int \varphi {\rm d} \mu - \int \varphi {\rm d} \nu \right\vert.
\end{equation}
An important point is to note that there exists a constant $C>0$ such that for every regular $ \varphi$ such that $ \left\Vert \varphi \right\Vert_{ Lip}\leq 1$, $ \left\vert \nabla P_{ s, T} \varphi \right\vert \leq C$ (where the constant is independent of $ \varphi$: see \cite{LucSta2014}, Lemma~4.4 for a similar proof in a more complicated context of singular interactions). By the Lipschitz continuity of $h$ and since $ z \mapsto z$ is obviously $1$-Lipschitz, one can bound the righthand part of \eqref{eq:mu_minus_nu} by $\tilde C \int_{ 0}^{t} \sup_{ u\leq s} W_{ 1}( \mu_{ u}, \nu_{ u}) {\rm d} s$. Taking now the supremum in $ \varphi$ and using the fact that $ \mu_{ 0}= \nu_{ 0}$, we obtain from Gr\"onwall's lemma that $ \sup_{ s\leq T} W_{ 1}( \mu_{ s}, \nu_{ s})=0$, which gives uniqueness.

\section*{Acknowledgements}

We would like to thank Giambattista Giacomin, Charles-Edouard Br\'ehier, Ivan Gentil and Arnaud Guillin for very fruitful discussions. C. Poquet benefited from the support of the ANR-17-CE40-0030.


\begin{thebibliography}{10}
\bibitem{MR2098593}
J.~A. Acebr{\'o}n, A.~R. Bulsara, and W.-J. Rappel.
\newblock Noisy {F}itz{H}ugh-{N}agumo model: from single elements to globally
  coupled networks.
\newblock {\em Phys. Rev. E (3)}, 69(2):026202, 9, 2004.

\bibitem{Araujo2010}
V.~Araujo, M.~Pacifico, and M.~Viana.
\newblock {\em Three-Dimensional Flows}.
\newblock Springer Berlin Heidelberg, 2010.

\bibitem{MR3155209}
D.~Bakry, I.~Gentil, and M.~Ledoux.
\newblock {\em Analysis and geometry of {M}arkov diffusion operators}, volume
  348 of {\em Grundlehren der Mathematischen Wissenschaften [Fundamental
  Principles of Mathematical Sciences]}.
\newblock Springer, Cham, 2014.

\bibitem{22657695}
J.~Baladron, D.~Fasoli, O.~Faugeras, and J.~Touboul.
\newblock Mean-field description and propagation of chaos in networks of
  {H}odgkin-{H}uxley and {F}itz{H}ugh-{N}agumo neurons.
\newblock {\em The Journal of Mathematical Neuroscience}, 2(1):10, 2012.

\bibitem{MR3541988}
A.~B.~T. Barbaro, J.~A. Ca\~nizo, J.~A. Carrillo, and P.~Degond.
\newblock Phase transitions in a kinetic flocking model of {C}ucker-{S}male
  type.
\newblock {\em Multiscale Model. Simul.}, 14(3):1063--1088, 2016.


\bibitem{Bates1998}
P.~Bates, K.~Lu, and C.~Zeng.
\newblock {\em Existence and persistence of invariant manifolds for semiflows
  in Banach space.}, volume 135.
\newblock Mem. Amer. Math. Soc., 1998.

\bibitem{MR2855983}
F.~Bolley, J.~A. Ca\~nizo, and J.~A. Carrillo.
\newblock Mean-field limit for the stochastic {V}icsek model.
\newblock {\em Appl. Math. Lett.}, 25(3):339--343, 2012.

\bibitem{Bolley:2012fk}
F.~Bolley, I.~Gentil, and A.~Guillin.
\newblock Uniform convergence to equilibrium for granular media.
\newblock {\em Archive for Rational Mechanics and Analysis}, pages 1--17, 2012.

\bibitem{MR3392551}
M.~Bossy, O.~Faugeras, and D.~Talay.
\newblock Clarification and complement to ``{M}ean-field description and
  propagation of chaos in networks of {H}odgkin-{H}uxley and
  {F}itz{H}ugh-{N}agumo neurons''.
\newblock {\em J. Math. Neurosci.}, 5:Art. 19, 23, 2015.

\bibitem{bressloff2014stochastic}
P.~C. Bressloff.
\newblock {\em Stochastic processes in cell biology}, volume~41 of {\em
  Interdisciplinary Applied Mathematics}.
\newblock Springer, Cham, 2014.

\bibitem{MR697382}
H.~Brezis.
\newblock {\em Analyse fonctionnelle}.
\newblock Collection Math{\'e}matiques Appliqu{\'e}es pour la Ma\^\i trise.
  [Collection of Applied Mathematics for the Master's Degree]. Masson, Paris,
  1983.
\newblock Th{\'e}orie et applications. [Theory and applications].

\bibitem{MR1840644}
S.~Cerrai.
\newblock {\em Second order {PDE}'s in finite and infinite dimension}, volume
  1762 of {\em Lecture Notes in Mathematics}.
\newblock Springer-Verlag, Berlin, 2001.
\newblock A probabilistic approach.

\bibitem{Collet:2015aa}
F.~Collet, P.~{Dai Pra}, and M.~Formentin.
\newblock Collective periodicity in mean-field models of cooperative behavior.
\newblock {\em Nonlinear Differential Equations and Applications NoDEA},
  22(5):1461--1482, 10 2015.

\bibitem{collet2016rhythmic}
F.~Collet, M.~Formentin, and D.~Tovazzi.
\newblock Rhythmic behavior in a two-population mean field Ising model.
\newblock {\em arXiv preprint arXiv:1606.06634}, 2016.

\bibitem{doi:10.1080/07362999808809576}
G.~Da~Prato and L.~Tubaro.
\newblock Some remarks about backward {I}t{\^o} formula and applications.
\newblock {\em Stochastic Analysis and Applications}, 16(6):993--1003, 1998.

\bibitem{dai2014noise}
P.~Dai~Pra, G.~Giacomin, and D.~Regoli.
\newblock Noise-induced periodicity: Some stochastic models for complex
  biological systems.
\newblock In {\em Mathematical Models and Methods for Planet Earth}, pages
  25--35. Springer, 2014.

\bibitem{Degond:2012uq}
P.~Degond, A.~Frouvelle, and J.-G. Liu.
\newblock Macroscopic limits and phase transition in a system of self-propelled
  particles.
\newblock {\em Journal of Nonlinear Science}, pages 1--30, 2012.


\bibitem{MR0120319}
J.~Dieudonn{\'e}.
\newblock {\em Foundations of modern analysis}.
\newblock Pure and Applied Mathematics, Vol. X. Academic Press, New
  York-London, 1960.

\bibitem{Ditlevsen:2015fk}
S.~Ditlevsen and E.~L{\"o}cherbach.
\newblock Multi-class oscillating systems of interacting neurons.
\newblock {\em Stochastic Process. Appl.}, 127(6):1840--1869, 06 2017.

\bibitem{Durmus2018}
A.~Durmus, A.~Eberle, A.~Guillin, and R.~Zimmer.
\newblock An Elementary Approach To Uniform In Time Propagation Of Chaos.
\newblock arXiv:1805.11387.

\bibitem{MR833476}
G.~B. Ermentrout and N.~Kopell.
\newblock Parabolic bursting in an excitable system coupled with a slow
  oscillation.
\newblock {\em SIAM J. Appl. Math.}, 46(2):233--253, 1986.

\bibitem{fenichel1971persistence}
N.~Fenichel.
\newblock Persistence and smoothness of invariant manifolds for flows.
\newblock {\em Indiana Univ. Math. J.}, 21:193--226, 1972.

\bibitem{fenichel1979geometric}
N.~Fenichel.
\newblock Geometric singular perturbation theory for ordinary differential
  equations.
\newblock {\em Journal of Differential Equations}, 31:53--98, 1979.

\bibitem{FitzHugh1961}
R.~FitzHugh.
\newblock Impulses and physiological states in theoretical models of nerve
  membrane.
\newblock {\em Biophysical Journal}, 1(6):445--466, 1961.

\bibitem{MR1697197}
J.~Garc\'\i~a Ojalvo and J.~M. Sancho.
\newblock {\em Noise in spatially extended systems}.
\newblock Institute for Nonlinear Science. Springer-Verlag, New York, 1999.

\bibitem{MR3024608}
A.~Genadot and M.~Thieullen.
\newblock Averaging for a fully coupled piecewise-deterministic {M}arkov
  process in infinite dimensions.
\newblock {\em Adv. in Appl. Probab.}, 44(3):749--773, 2012.

\bibitem{MR3207725}
G.~Giacomin, E.~Lu{\c{c}}on, and C.~Poquet.
\newblock Coherence {S}tability and {E}ffect of {R}andom {N}atural
  {F}requencies in {P}opulations of {C}oupled {O}scillators.
\newblock {\em J. Dynam. Differential Equations}, 26(2):333--367, 2014.

\bibitem{doi:10.1137/110846452}
G.~Giacomin, K.~Pakdaman, X.~Pellegrin, and C.~Poquet.
\newblock Transitions in active rotator systems: Invariant hyperbolic manifold
  approach.
\newblock {\em SIAM Journal on Mathematical Analysis}, 44(6):4165--4194, 2012.

\bibitem{menonhaller}
M.~Govind and G.~Haller.
\newblock Infinite dimensional geometric singular perturbation theory for the
  maxwell--bloch equations.
\newblock {\em SIAM J. Math. Anal.}, 33(2):315--346, 2001.

\bibitem{guckenheimer1979structural}
J.~Guckenheimer and R.~F. Williams.
\newblock Structural stability of lorenz attractors.
\newblock {\em Publications Math{\'e}matiques de l'IH{\'E}S}, 50:59--72, 1979.

\bibitem{hirsch1977invariant}
M.~W. Hirsch, C.~C. Pugh, and M.~Shub.
\newblock {\em Invariant manifolds}.
\newblock Lecture Notes in Mathematics, Vol. 583. Springer-Verlag, Berlin,
  1977.

\bibitem{MR2263523}
E.~M. Izhikevich.
\newblock {\em Dynamical systems in neuroscience: the geometry of excitability
  and bursting}.
\newblock Computational Neuroscience. MIT Press, Cambridge, MA, 2007.

\bibitem{Ko2010}
C.~H. Ko, Y.~R. Yamada, D.~K. Welsh, E.~D. Buhr, A.~C. Liu, E.~E. Zhang, M.~R.
  Ralph, S.~A. Kay, D.~B. Forger, and J.~S. Takahashi.
\newblock Emergence of noise-induced oscillations in the central circadian
  pacemaker.
\newblock {\em PLOS Biology}, 8(10):1--19, 10 2010.

\bibitem{Kuramoto1975}
Y.~Kuramoto.
\newblock Self-entrainment of a population of coupled non-linear oscillators.
\newblock In {\em International {S}ymposium on {M}athematical {P}roblems in
  {T}heoretical {P}hysics ({K}yoto {U}niv., {K}yoto, 1975)}, pages 420--422.
  Lecture Notes in Phys., 39. Springer, Berlin, 1975.

\bibitem{LINDNER2004321}
B.~Lindner, J.~Garcia-Ojalvo, A.~Neiman, and L.~Schimansky-Geier.
\newblock Effects of noise in excitable systems.
\newblock {\em Physics Reports}, 392(6):321 -- 424, 2004.

\bibitem{PhysRevE.60.7270}
B.~Lindner and L.~Schimansky-Geier.
\newblock Analytical approach to the stochastic fitzhugh-nagumo system and
  coherence resonance.
\newblock {\em Phys. Rev. E}, 60:7270--7276, Dec 1999.

\bibitem{LucSta2014}
E.~Lu\c{c}on and W.~Stannat.
\newblock Mean field limit for disordered diffusions with singular
  interactions.
\newblock {\em Ann. Appl. Probab.}, 24(5):1946--1993, 2014.

\bibitem{MR1022538}
M.~Marion.
\newblock Inertial manifolds associated to partly dissipative
  reaction-diffusion systems.
\newblock {\em J. Math. Anal. Appl.}, 143(2):295--326, 1989.

\bibitem{McKean1967}
H.~P. McKean, Jr.
\newblock Propagation of chaos for a class of non-linear parabolic equations.
\newblock In {\em Stochastic {D}ifferential {E}quations ({L}ecture {S}eries in
  {D}ifferential {E}quations, {S}ession 7, {C}atholic {U}niv., 1967)}, pages
  41--57. Air Force Office Sci. Res., Arlington, Va., 1967.

\bibitem{Mischler2016}
S.~Mischler, C.~Qui{\~{n}}inao, and J.~Touboul.
\newblock On a kinetic fitzhugh--nagumo model of neuronal network.
\newblock {\em Communications in Mathematical Physics}, 342(3):1001--1042,
  2016.

\bibitem{Nagumo1962}
J.~Nagumo, S.~Arimoto, and S.~Yoshizawa.
\newblock An active pulse transmission line simulating nerve axon.
\newblock {\em Proceedings of the IRE}, 50(10):2061--2070, Oct 1962.

\bibitem{MR2779558}
K.~Pakdaman, M.~Thieullen, and G.~Wainrib.
\newblock Fluid limit theorems for stochastic hybrid systems with application
  to neuron models.
\newblock {\em Adv. in Appl. Probab.}, 42(3):761--794, 2010.

\bibitem{MR1881888}
J.~C. Robinson.
\newblock {\em Infinite-dimensional dynamical systems}.
\newblock Cambridge Texts in Applied Mathematics. Cambridge University Press,
  Cambridge, 2001.
\newblock An introduction to dissipative parabolic PDEs and the theory of
  global attractors.

\bibitem{MR1779040}
C.~Roc\c~soreanu, A.~Georgescu, and N.~Giurgi\c~teanu.
\newblock {\em The {F}itz{H}ugh-{N}agumo model}, volume~10 of {\em Mathematical
  Modelling: Theory and Applications}.
\newblock Kluwer Academic Publishers, Dordrecht, 2000.
\newblock Bifurcation and dynamics.

\bibitem{Sakaguchi1986}
H.~Sakaguchi and Y.~Kuramoto.
\newblock A soluble active rotator model showing phase transitions via mutual
  entrainment.
\newblock {\em Progr. Theoret. Phys.}, 76(3):576--581, 1986.

\bibitem{Sakaguchi1988a}
H.~Sakaguchi, S.~Shinomoto, and Y.~Kuramoto.
\newblock Phase transitions and their bifurcation analysis in a large
  population of active rotators with mean-field coupling.
\newblock {\em Progr. Theoret. Phys.}, 79(3):600--607, 1988.

\bibitem{MR2254750}
F.~Sauvigny.
\newblock {\em Partial differential equations. 2}.
\newblock Universitext. Springer-Verlag, Berlin, 2006.
\newblock Functional analytic methods, With consideration of lectures by E.
  Heinz, Translated and expanded from the 2005 German original.

\bibitem{MR808166}
M.~Scheutzow.
\newblock Noise can create periodic behavior and stabilize nonlinear
  diffusions.
\newblock {\em Stochastic Process. Appl.}, 20(2):323--331, 1985.

\bibitem{MR843504}
M.~Scheutzow.
\newblock Periodic behavior of the stochastic {B}russelator in the mean-field
  limit.
\newblock {\em Probab. Theory Relat. Fields}, 72(3):425--462, 1986.

\bibitem{sell2013dynamics}
G.~R. Sell and Y.~You.
\newblock {\em Dynamics of evolutionary equations}, volume 143 of {\em Applied
  Mathematical Sciences}.
\newblock Springer-Verlag, New York, 2013.

\bibitem{Shinomoto1986}
S.~Shinomoto and Y.~Kuramoto.
\newblock Phase transitions in active rotator systems.
\newblock {\em Progress of Theoretical Physics}, 75(5):1105--1110, 1986.

\bibitem{strogatz2014nonlinear}
S.~Strogatz.
\newblock {\em Nonlinear Dynamics and Chaos: With Applications to Physics,
  Biology, Chemistry, and Engineering}.
\newblock Studies in Nonlinearity. Avalon Publishing, 2014.

\bibitem{SznitSflour}
A.-S. Sznitman.
\newblock Topics in propagation of chaos.
\newblock In {\em {\'E}cole d'{\'E}t{\'e} de {P}robabilit{\'e}s de
  {S}aint-{F}lour {XIX}---1989}, volume 1464 of {\em Lecture Notes in Math.},
  pages 165--251. Springer, Berlin, 1991.

\bibitem{MR2902610}
J.~Touboul, G.~Hermann, and O.~Faugeras.
\newblock Noise-induced behaviors in neural mean field dynamics.
\newblock {\em SIAM J. Appl. Dyn. Syst.}, 11(1):49--81, 2012.

\bibitem{MR3180036}
J.~Tugaut.
\newblock Phase transitions of {M}c{K}ean-{V}lasov processes in double-wells
  landscape.
\newblock {\em Stochastics}, 86(2):257--284, 2014.


\bibitem{wiggins2013normally}
S.~Wiggins.
\newblock {\em Normally hyperbolic invariant manifolds in dynamical systems},
  volume 105 of {\em Applied Mathematical Sciences}.
\newblock Springer, 2013.

\end{thebibliography}
\end{document}